\documentclass[11pt]{amsart}
\usepackage{amssymb,amsmath, amsthm, amsfonts, tikz-cd}

\usepackage{graphicx}
\usepackage{listings}
\usepackage[margin=1in]{geometry}
\usepackage{lstautogobble}
\usepackage{enumerate}
\usepackage[shortlabels]{enumitem}
\usepackage{thmtools}
\usepackage{thm-restate}
\usepackage{amsthm}
\usepackage{verbatim}

\usepackage{mathtools}
\usepackage{physics}
\usepackage{color}
\usepackage{hyperref}
\usepackage[capitalise]{cleveref}
\usepackage[bottom]{footmisc}
\crefformat{equation}{(#2#1#3)}

\usepackage{float}
\restylefloat{table}

\usepackage{mathrsfs}
\setlist{  
  listparindent=\parindent,
  parsep=0pt,
}

\theoremstyle{plain}
\newtheorem{thm}{Theorem}[section]
\newtheorem{prop}[thm]{Proposition}
\newtheorem{lemma}[thm]{Lemma}
\newtheorem{cor}[thm]{Corollary}

\theoremstyle{definition}

\newtheorem{remark}[thm]{Remark}

\Crefname{thm}{Theorem}{Theorems}
\Crefname{prop}{Proposition}{Propositions}

\numberwithin{equation}{section} %Equation numbering

\DeclarePairedDelimiter\ipp{\langle}{\rangle}

\DeclarePairedDelimiter{\paren}{\lparen}{\rparen}

\DeclareMathOperator{\supp}{supp}

\DeclareMathOperator{\sgn}{sgn}

\newcommand{\M}{{\mathcal{M}}}

\newcommand{\p}{{\partial}}

\renewcommand{\d}{\delta}

\newcommand{\R}{{\mathbb{R}}}

\newcommand{\N}{{\mathbb{N}}}

\renewcommand{\H}{{\mathcal{H}}}
\newcommand{\T}{{\mathbb{T}}}
\newcommand{\g}{{\mathsf{g}}}
\newcommand{\G}{{\mathsf{G}}}

\newcommand{\Sc}{{\mathcal{S}}}

\renewcommand{\M}{{\mathbb{M}}}
\newcommand{\I}{\mathbb{I}}

\renewcommand{\k}{\mathsf{k}}
\newcommand{\ga}{\gamma}

\newcommand{\tl}{\tilde}

\newcommand{\D}{\Delta}

\newcommand{\ph}{\phantom{=}}
\newcommand{\nn}{\nonumber}

\newcommand{\ux}{\underline{x}}

\newcommand{\ep}{\epsilon}
\newcommand{\vep}{\varepsilon}
\newcommand{\al}{\alpha}
\newcommand{\be}{\beta}
\newcommand{\ka}{\kappa}
\newcommand{\la}{\lambda}

\newcommand{\Tc}{\mathcal{T}}

\newcommand{\indic}{\mathbf{1}}

\newcommand{\E}{{\mathbb{E}}}

\newcommand{\wh}{\widehat}

\newcommand{\Dm}{|\nabla|}

\newcommand{\Fr}{F}
\renewcommand{\P}{\mathcal{P}}
\newcommand{\Te}{\mathrm{Term}}

\newcommand{\om}{\omega}
\newcommand{\Pb}{\mathbb{P}}
\let\div\relax
\DeclareMathOperator{\div}{\mathrm{div}}

\def\Xint#1{\mathchoice
{\XXint\displaystyle\textstyle{#1}}%
{\XXint\textstyle\scriptstyle{#1}}%
{\XXint\scriptstyle\scriptscriptstyle{#1}}%
{\XXint\scriptscriptstyle\scriptscriptstyle{#1}}%
\!\int}
\def\XXint#1#2#3{{\setbox0=\hbox{$#1{#2#3}{\int}$ }
\vcenter{\hbox{$#2#3$ }}\kern-.6\wd0}}

\def\dashint{\Xint-}

\let\oldtocsection=\tocsection
 
\let\oldtocsubsection=\tocsubsection
 
\let\oldtocsubsubsection=\tocsubsubsection
 
\renewcommand{\tocsection}[2]{\hspace{0em}\oldtocsection{#1}{#2}}
\renewcommand{\tocsubsection}[2]{\hspace{1em}\oldtocsubsection{#1}{#2}}
\renewcommand{\tocsubsubsection}[2]{\hspace{2em}\oldtocsubsubsection{#1}{#2}}

\title[Global-in-time mean-field convergence for Singular Diffusive Flows]{Global-in-time mean-field convergence for singular Riesz-type diffusive flows}

\author[M. Rosenzweig]{Matthew Rosenzweig}
\email{mrosenzw@mit.edu}
\thanks{M.R. is supported by the Simons Foundation through the Simons Collaboration on Wave Turbulence and by NSF grant DMS-2052651.}
\author[S. Serfaty]{Sylvia Serfaty}
\email{serfaty@cims.nyu.edu}
\thanks{S.S. is supported by NSF grant DMS-2000205 and by the Simons Foundation through the Simons Investigator program.}

\begin{document}
\begin{abstract}
We consider the mean-field limit of systems of particles with singular interactions of the type $-\log|x|$ or $|x|^{-s}$, with $0< s<d-2$, and with an additive noise in dimensions $d \geq 3$. We use a modulated-energy approach to prove a quantitative  convergence rate to the solution of the  corresponding limiting PDE. When $s>0$, the convergence is global in time, and it is the first such result valid for both conservative and gradient flows in a singular setting on $\R^d$. The proof relies on an adaptation of an argument of Carlen-Loss \cite{CL1995} to show a decay rate of the solution to the limiting equation, and on an improvement of the modulated-energy method developed in \cite{Duerinckx2016,Serfaty2020,NRS2021}, making it so that all prefactors in the time derivative of the modulated energy  are controlled by a decaying bound on the limiting solution.
\end{abstract}
\maketitle

\section{Introduction}
\label{sec:intro}
\subsection{The problem}
\label{ssec:introP}
We consider the first-order mean-field dynamics of stochastic interacting particle systems of the form
\begin{equation}
\label{eq:SDE}
\begin{cases}
dx_{i}^t = \displaystyle\frac{1}{N}\sum_{1\leq j\leq N : j\neq i} \M\nabla\g(x_i^t-x_j^t)dt + \sqrt{2\sigma}dW_i^t\\
x_i^t|_{t=0} = x_i^0
\end{cases}\qquad i\in\{1,\ldots,N\}.
\end{equation}
Above, $x_i^0\in\R^d$ are the pairwise distinct initial positions, $\M$ is a $d\times d$ matrix such that
\begin{equation}\label{eq:Mnd}
\M\xi\cdot\xi \leq 0 \qquad \forall \xi\in\R^d,
\end{equation}
and $W_1,\ldots,W_N$ are independent standard Brownian motions in $\R^d$, so that the noise in \eqref{eq:SDE} is of so-called additive type. There are several choices for $\M$. For instance, choosing $\M =-\I$ yields \emph{gradient-flow/dissipative} dynamics, while choosing $\M$ to be antisymmetric yields \emph{Hamiltonian/conservative} dynamics. Mixed flows are also permitted. The potential $\g$ is assumed to be repulsive, which, as we shall later show in \cref{sec:Npd}, ensures that there is a unique, global strong solution to the system \eqref{eq:SDE}. In particular, with probability one, the particles never collide. The model case for $\g$ is either a logarithmic or Riesz potential indexed by a parameter $0\leq s<d-2$, according to
\begin{equation}\label{eq:gmod}
\g(x) = \begin{cases} - \log |x|, & {s=0} \\ |x|^{-s}, & {0<s<d-2}. \end{cases}
\end{equation}
The above restriction on $s$ means that we are considering potentials that are \emph{sub-Coulombic}: their singularity is below that of the Coulomb potential, which corresponds to $s=d-2$. As explained precisely in the next subsection, we can consider a general class of potentials $\g$ which have sub-Coulombic-type behavior.

Systems of the form \eqref{eq:SDE} have numerous applications in the physical and life sciences as well as economics. Examples include vortices in viscous fluids \cite{Onsager1949, Chorin1973, Osada1987pc, MP2012book}, models of the collective motion of microscopic organisms \cite{OS1997, TBL2006, Perthame2007, GQ2015, LY2016, FJ2017}, aggregation phenomena \cite{BGM2010, CGM2008, Malrieu2003}, and opinion dynamics \cite{HK2002, Krause2000, MT2011, XWX2011}. For more discussion on applications, we refer the reader to the survey of Jabin and Wang \cite{JW2017_survey} and references therein.

Establishing the  mean-field limit consists of showing the convergence in a suitable topology as $N \to \infty$ of the {\it empirical measure} 
\begin{equation}\label{munt}
\mu_N^t\coloneqq \frac1N \sum_{i=1}^N \delta_{x_i^t}
\end{equation}
associated to a solution $\ux_N^t \coloneqq (x_1^t, \dots, x_N^t)$ of the system \eqref{eq:SDE}. We remark that for fixed $t$, the empirical measure is a random Borel probability measure on $\R^d$. If the points $x_i^0$, which themselves depend on $N$, are such that $\mu_N^0$   converges  to some regular measure $\mu^0$, then a formal calculation using It\^o's lemma leads to the expectation that for $t>0$, $\mu_N^t$ converges to the solution of the Cauchy problem with initial datum $\mu^0$ of the limiting evolution equation
\begin{equation}
\label{eq:lim}
\begin{cases}
\p_t\mu = -\div(\mu\M\nabla\g\ast\mu) +\sigma\D \mu\\
\mu|_{t=0} = \mu^0
\end{cases}
\qquad (t,x)\in \R_+\times\R^d
\end{equation}
as the number of particles $N\rightarrow\infty$. While the underlying $N$-body dynamics are stochastic, we emphasize that the equation \eqref{eq:lim} is completely deterministic, and the noise has been averaged out to become diffusion. Proving the convergence of the empirical measure is closely related to proving {\it propagation of molecular chaos} (see \cite{Golse2016ln,HM2014,Jab2014} and references therein): if $f_N^0(x_1, \dots, x_N)$ is the initial law of the distribution of the $N$ particles in $\R^d$ and if $f_N^0$ converges to some factorized law  $(\mu^0)^{\otimes N}$, then  the $k$-point marginals $f_{N,k}^t$ converge for all time to $(\mu^t )^{\otimes k}$. %Our result implies a convergence of this type as well (see \cref{rem:pc} below).

The mean-field problem for the system \eqref{eq:SDE} with $\sigma>0$ and interactions which are regular (e.g. globally Lipschitz) has been understood for many years now \cite{Mckean1967, Sznitman1991, Meleard1996, Malrieu2003} (see also \cite{BGM2010, BCC2011, MMW2015, Lacker2021, DT2021} for more recent developments still in the regular case). The classical approach consists in comparing the trajectories of the original system \eqref{eq:SDE} to those of a cooked-up symmetric particle system coupled to \eqref{eq:SDE}. Subsequent work has focused on treating the more challenging singular interactions --- initially by compactness-type arguments that yield qualitative convergence \cite{Osada1986pc, Osada1987lp, Osada1987pc, FHM2014, GQ2015, LY2016, FJ2017, LLY2019} and later by more quantitative methods that yield an explicit rate for propagation of chaos \cite{Holding2016, JW2018, BJW2019edp, BJW2020}. To our knowledge, the best results in the literature can quantitatively prove propagation of chaos for singular interactions up to and including the Coulomb case $s=d-2$ for conservative dynamics \cite{JW2018} and arbitrary $0\leq s<d$ in for dissipative dynamics \cite{BJW2019edp, BJW2020}.  Unlike the previous aforementioned works which utilize the noise in an essential way, the methods of \cite{JW2018, BJW2019edp} allow for taking vanishing diffusion: $\sigma=\sigma_N\geq 0$, where $\sigma_N\rightarrow 0$ as $N\rightarrow\infty$. We also mention that the recent preprint \cite{WZZ2021} has gone beyond the mean-field limit and shown the convergence of the fluctuations of the empirical measure to \eqref{eq:SDE} to a generalized Ornstein-Uhlenbeck process for singular potentials including the two-dimensional Coulomb case. These state-of-the-art works are limited to the periodic setting.

When there is no noise in the system \eqref{eq:SDE} (i.e. $\sigma=0$), much more is known mathematically about the mean-field limit thanks to recent advances that are capable of treating the full potential case $s<d$. Approaches vary, but they all typically involve finding a good metric to measure the distance between the empirical measure and its expected limit and then proving a Gronwall relation for the evolution of this metric. The $\infty$-Wasserstein metric allowed to treat the sub-Coulombic case $s<d-2$ \cite{Hauray2009, CCH2014}. A Wassertein-gradient-flow approach \cite{CFP2012, BO2019} can also treat the one-dimensional case using the convexity of the Riesz interaction in (and only in) dimension one. The modulated-energy approach of \cite{Duerinckx2016, Serfaty2020}, inspired by the prior work \cite{Serfaty2017}, managed to treat the more difficult Coulomb and super-Coulombic case $d-2\leq s<d$ for the model potential \eqref{eq:gmod}. In very recent work by the authors together with Nguyen \cite{NRS2021}, this modulated-energy approach has been redeveloped to allow to treat the full range $s<d$ and under fairly general assumptions for the potential $\g$. In these works, the modulated energy is a Coulomb/Riesz-based metric that can be understood as a renormalization of the negative-order homogeneous Sobolev norm corresponding to the energy space of the equation \eqref{eq:lim}. More precisely, it  is defined to be
\begin{equation}\label{eq:introME}
\Fr_N(\ux_N,\mu) \coloneqq \int_{(\R^d)^2\setminus\triangle} \g(x-y)d\paren*{\frac{1}{N}\sum_{i=1}^N\d_{x_i} - \mu}^{\otimes 2}(x,y),
\end{equation}
where we remove the infinite self-interaction of each particle by excising the diagonal $\triangle$.

Contemporaneous to the development of the modulated-energy approach, Jabin and Wang \cite{JW2016, JW2018} introduced a relative-entropy method capable of treating the mean-field limit of \eqref{eq:SDE} when the interaction is moderately singular and which works well with or without noise. The relative-entropy and modulated-energy approaches were recently combined into a \emph{modulated free energy} method \cite{BJW2019crm, BJW2019edp, BJW2020} that allows for treating the mean-field limit of \eqref{eq:SDE} in the dissipative case, but not the conservative case, set on the torus and under fairly general assumptions on the interaction, impressively allowing even for attractive potentials (e.g. Patlak-Keller-Segel type).  

In this article, we show for the first time that the modulated-energy approach of \cite{Duerinckx2016, Serfaty2020, NRS2021} can be extended to treat the mean-field limit of \eqref{eq:SDE} in the sub-Coulombic case $0\leq s<d-2$.\footnote{Our ability to use the modulated-energy approach in the sub-Coulombic case crucially relies on our recent work \cite{NRS2021} with Nguyen, since the previous works \cite{Duerinckx2016, Serfaty2020} could only treat this way the Coulomb/super-Coulombic case.} No incorporation of the entropy, as in the modulated-free-energy approach of \cite{BJW2019crm, BJW2019edp, BJW2020} is needed. Moreover, the modulated-energy approach is well-suited for exploiting the dissipation of the limiting equation \eqref{eq:lim} to obtain rates of convergence in $N$ which are \emph{uniform} over the entire interval $[0,\infty)$. In other words, mean-field convergence holds globally in time. At the time of completion of this manuscript, this is, to the best of our knowledge, the first instance of such a result for singular potentials. Obtaining  a uniform-in-time convergence is important in both theory and practice --- for instance, when using a particle system to approximate the limiting equation or its equilibrium states and for quantifying stochastic gradient methods, such as those used in machine learning for other interaction kernels (for instance, see \cite{CB2018, MMN2018, RvE2018}).

Lastly, we mention that previous uses of the modulated energy in the stochastic setting \cite{Rosenzweig2020spv, NRS2021} were limited to the case of multiplicative noise, which behaves very differently in the limit as $N\rightarrow\infty$. Most notably, the limiting evolution equation is stochastic.

\subsection{Formal idea}
Let us sketch our main proof in the model case \eqref{eq:gmod}. For simplicity of exposition,  let us also   restrict ourselves to the simpler range $d-4<s<d-2$. We note that  the modulated energy \eqref{eq:introME}  is a real-valued continuous stochastic process. Formally by It\^o's lemma (see \cref{sec:ev} for the rigorous computation), it satisfies the stochastic differential inequality (cf. \cite[Lemma 2.1]{Serfaty2020})
\begin{multline}
\label{eq:MEid}
\frac{d}{dt}\Fr_N(\ux_N^t,\mu^t) \leq  \int_{(\R^d)^2\setminus\triangle}\nabla\g(x-y)\cdot\paren*{u^t(x) - u^t(y)}d(\mu_N^t-\mu^t)^{\otimes 2}(x,y) \\
+ \sigma\int_{(\R^d)^2\setminus\triangle}\D\g(x-y)d(\mu_N^t-\mu^t)^{\otimes 2}(x,y) +\frac{2\sqrt{2\sigma}}{N}\sum_{i=1}^N\int_{\R^d\setminus\{x_i^t\}}\nabla\g(x_i^t-y)d(\mu_N^t-\mu^t)(y)\cdot \dot{W}_i^t,
\end{multline}
where we have set $u^t\coloneqq \M\nabla\g\ast\mu^t$. The third term in the right-hand side (formally) has zero expectation and may be ignored for the purposes of this discussion. The first term is the contribution of the drift and also appears in the deterministic case, but the second term is  new and due to the nonzero quadratic variation of Brownian motion. Observe that
\begin{equation}
\D\g(x-y) = -(d-s-2)|x-y|^{-s-2}.
\end{equation}
Since $0\leq s<d-2$ by assumption, we see that $\D\g$ is superharmonic and equals a constant multiple of $-\tl{\g}$, where $\tl{\g}$ is the kernel of the Riesz potential operator $(-\D)^{\frac{s+2-d}{2}}$. We would like to conclude that the second term in the right-hand side of \eqref{eq:MEid} is nonpositive by Plancherel's theorem and therefore may be discarded, but the excision of the diagonal $\triangle$ obstructs this reasoning. Fortunately, prior work of the second author \cite[Proposition 3.3]{Serfaty2020} gives the lower bound
\begin{equation}
\label{eq:introrhs}
\begin{split}
\int_{(\R^d)^2\setminus\triangle}\tl{\g}(x-y)d\paren*{\mu_N^t-\mu^t}^{\otimes 2}(x,y) \geq -\frac{1}{N^2}\sum_{i=1}^N \tl{\g}(\eta_i) - \frac{C\|\mu^t\|_{L^\infty}}{N}\sum_{i=1}^N \eta_i^{d-s-2}
\end{split}
\end{equation}
for all choices of parameters $\eta_i>0$. Here, $C$ is a constant that just depends on $s,d$. We emphasize that \eqref{eq:introrhs} is a functional inequality which holds independently of any underlying dynamics. The choice of $\eta_i$ that balances the decay in $N$ between the two terms in the right-hand side of the inequality \eqref{eq:introrhs} is the typical interparticle distance $N^{-1/d}$. Since the $L^\infty$ norm of $\mu^t$ is a source of decay and we wish to distribute it between terms, we instead choose
\begin{equation}
\eta_i = (\|\mu^t\|_{L^\infty}N)^{-1/d} \qquad \forall 1\leq i\leq N,
\end{equation}
so that the right-hand side of \eqref{eq:introrhs} is bounded from below by
\begin{equation}
-C\sigma\|\mu^t\|_{L^\infty}^{\frac{s+2}{d}} N^{-\frac{d-s}{d}},
\end{equation}
providing a bound from above for the corresponding term in \eqref{eq:MEid}.
Note that since $\mu^t$ is time-dependent, our choice for $\eta_i$ above depends on time, a trick previously used by the first author \cite{Rosenzweig2020pvmf, Rosenzweig2020spv} to study the mean-field limit for point vortices with possible multiplicative noise when $\mu^t$ belongs to a function space which is invariant or critical under the scaling of the equation.

It remains to consider the first term in the right-hand side of \eqref{eq:MEid}, which, as previously mentioned, also appears in the deterministic case. This expression has the structure of a commutator which has been \emph{renormalized} through the exclusion of diagonal in order to accommodate the singularity of the Dirac masses. As shown in \cite{Rosenzweig2020pvmf, NRS2021}, one can make this commutator intuition rigorous (see \Cref{prop:rcomm,prop:imprcomm} below) and, revisiting the estimates there  together with some elementary potential analysis,  we are able to optimize the dependence in $\|\mu\|_{L^\infty}$ of the estimate and show the pathwise and pointwise-in-time bound
\begin{multline}
\left|\int_{(\R^d)^2\setminus\triangle }\nabla\g(x-y)\cdot\paren*{u^t(x) - u^t(y)}d(\mu_N^t-\mu^t)^{\otimes 2}(x,y)\right|\\
\leq C\|\mu^t\|_{L^\infty}^{\frac{s+2}{d}}\paren*{|\Fr_N(\ux_N^t,\mu^t)| + (1+\|\mu^t\|_{L^\infty}) N^{-\beta}},
\end{multline}
where again $C,\beta>0$ are constants depending only $s,d$.

Now taking expectations of both sides of \eqref{eq:MEid}, integrating with respect to time, and using \cref{rem:MEbal} below to control $|\Fr_N(\ux_N^t,\mu^t)|$ by $\Fr_N(\ux_N^t,\mu^t)$, we find that
\begin{equation}\label{forgronwall}
\begin{split}
\E(|\Fr_N(\ux_N^t,\mu^t)|) &\leq |\Fr_N(\ux_N^0,\mu^0)| + C\|\mu^t\|_{L^\infty}^{\frac{s}{d}}N^{-\beta} + C\int_0^t \|\mu^\tau\|_{L^\infty}^{\frac{s+2}{d}}\E(|\Fr_N(\ux_N^\tau,\mu^\tau)|)d\tau \\
&\ph+ C\sigma N^{-\beta}\int_0^t \|\mu^\tau\|_{L^\infty}^{\frac{s+2}{d}}d\tau.
\end{split}
\end{equation}
The structure of the right-hand side of \eqref{forgronwall} allows us to leverage the decay rate of the solution to \eqref{eq:lim}. In \cref{prop:dcay}, we show the decay rate
\begin{equation*}
\|\mu^t\|_{L^\infty} \leq \min\{C(\sigma t)^{-\frac{d}{2}}, \|\mu^0\|_{L^\infty}\}.
\end{equation*}
This is done by revisiting work of Carlen and Loss \cite{CL1995} on the optimal decay of nonlinear viscously damped conservation laws, which was essentially restricted to divergence-free vector fields, and adapting it to treat the case of \eqref{eq:lim}.
 
Once this is done, an application of the Gronwall-Bellman lemma to \eqref{forgronwall} yields a uniform-in-time bound if $s>0$, while if $s=0$, long-range effects only allow us to obtain an $O(t^{\frac{C}{\sigma}})$ growth estimate.

\subsection{Assumptions on the potential and main results}
We now state the precise assumptions for the class of interaction potentials we consider. This class corresponds to the sub-Coulombic sub-class of the larger class of potentials considered by the authors in collaboration with Nguyen in \cite{NRS2021}. In the statement below and throughout this article, the notation $\indic_{(\cdot)}$ denotes the indicator function for the condition $(\cdot)$.

For $d\geq 3$ and $0\leq s<d-2$, we assume the following:
\begin{enumerate}[(i)]
\item\label{ass0}
\begin{equation*}
\g(x) = \g(-x)
\end{equation*}
\item\label{ass1a}
\begin{equation*}
\lim_{x\rightarrow 0} \g(x) = \infty
\end{equation*}
\item\label{ass1}
\begin{equation*}
\text{$\exists r_0>0$ such that}\quad \Delta \g  \leq 0 \quad {\text{in $B(0,r_0)$}}\footnotemark
\end{equation*}
\footnotetext{Here, we mean $\g$ is superharmonic in $B(0,r_0)$ in the sense of distributions, which implies that $\D\g(x)\leq 0$ for almost every $x\in B(0,r_0)$.}
\item\label{ass2}
\begin{equation*}
\forall k \geq 0,\qquad 
|\nabla^{\otimes k} \g(x)|\le C\paren*{\frac{1}{|x|^{s+k}} + |\log |x||\indic_{s=k=0}}  \quad \forall {x\in\R^{d}\setminus\{0\}}
\end{equation*}
\item\label{ass2b}
\begin{equation*}
|x||\nabla\g(x)| + |x|^2|\nabla^{\otimes 2}\g(x)| \leq C\g(x) \quad \forall x\in B(0,r_0){\setminus\{0\}}
\end{equation*}
\item\label{ass3}
\begin{equation*}
\frac{C_1} {|\xi|^{ d-s}}  \leq  \hat \g(\xi)  \leq \frac{C_2}{|\xi|^{d-s}} \quad {\forall \xi\in\R^{d}\setminus\{0\}}\end{equation*}
where $\hat \cdot$ denotes the Fourier transform.
\item\label{ass3a}
\begin{equation*}
\begin{cases}
\text{$\exists c_s<1$ such that} \quad \g(x)<c_s\g(y)\quad \forall x,y\in B(0,r_0) \ \text{with $|y|\geq 2|x|$}, & {s>0} \\
\text{$\exists c_0>0$ such that} \quad \g(x)-\g(y) \geq c_0 \quad \forall x,y\in B(0,r_0) \ \text{with $|y|\geq 2|x|$}, & {s=0}
\end{cases}
\end{equation*}
\item\label{ass4}
If $d-4<s<d-2$, then we also assume that there is an $m\in\N$ and $\G:\R^{d+m}\rightarrow\R$ such that
\begin{equation*}
-\D\g(x) = \G(x,0) \quad \forall (x,0)\in\R^{d+m}
\end{equation*}
\begin{equation}\label{extass0}
\G(X) = \G(-X)
\end{equation}
%\begin{equation}\label{extass0'}
%\lim_{X\rightarrow 0} \G(X) = \infty
%\end{equation}
\begin{equation}\label{extass1}
\text{$\exists r_0>0$ such that} \quad \D\G(X) \leq 0 \quad \text{in $B(0,r_0)\subset\R^{d+m}$}
\end{equation}
\begin{equation}\label{extass2}
\forall k\geq 0, \quad |\nabla^{\otimes k}\G(X)| \leq \frac{C}{|X|^{s+2+k}} \quad \forall X\in B(0,r_0)
\end{equation}
\begin{equation}\label{extass3}
\hat{\G}(\Xi) \geq 0 \quad \forall \Xi \in\R^{d+m}\setminus\{0\}.
\end{equation}
\item\label{ass3'}
In the cases $s=d-2k\geq 0$, for some positive integer $k$, we also assume that the $(\R^d)^{\otimes (2k+2)}$-valued kernel
\begin{equation*}
\k(x-y) \coloneqq  (x-y)\otimes \nabla^{\otimes (2k+1)}\g(x-y)
\end{equation*}
is associated to a Calder\'{o}n-Zygmund operator.\footnote{Sufficient and necessary conditions for this Calder\'{o}n-Zygmund property are explained in \cite[Section 5.4]{Grafakos2014c}. The reader may check that this condition is satisfied if $\g$ is the Riesz potential $|x|^{2k-d}$.}
\item\label{ass5}
In all cases,
\begin{equation*}
\M:\nabla^{\otimes 2}\g(x) \geq 0 \quad \forall x\in \R^{d} \setminus\{0\},
\end{equation*}
where $:$ denotes the Frobenius inner product.
\end{enumerate}
We shall say that any potential $\g$ satisfying assumptions \ref{ass0} -- \ref{ass5} is an \emph{admissible potential}. We refer to \cite[Subsection 1.3]{NRS2021} for a discussion of the types of potentials permitted under these assumptions. Compared to that work, only assumptions \ref{ass4} and \ref{ass5} are new. The former is to ensure that our modulated-energy method can be applied to $\tilde \g$ and thus to the diffusion term in \eqref{eq:MEid}, while the latter is to ensure that the solutions of \eqref{eq:lim} satisfy the temporal decay bounds of the heat equation. Note that \ref{ass5} is automatically satisfied if $\M$ is antisymmetric. Additionally, if $\M=-\I$ so that we consider gradient-flow dynamics, then \ref{ass5} amounts to requiring that $\g$ is globally superharmonic (i.e. $r_0=\infty$ in \ref{ass1}). In general, though, we do not require global superharmonicity except where explicitly stated.
\medskip

We assume that we are given a filtered probability space $(\Omega,\mathcal{F},(\mathcal{F}_t)_{t\geq 0}, \Pb)$ on which a countable collection of independent standard $d$-dimensional Brownian motions $(W_n)_{n=1}^\infty$ are defined. Moreover, $(\mathcal{F}_t)_{t\geq 0}$ is the complete filtration generated by the Brownian motions. All stochastic processes considered in this article are defined on this probability space.

Let $\ux_{N}^0\in (\R^d)^N$ be an $N$-tuple of distinct points in $\R^d$. As shown in \cref{prop:Nwp} (more generally \cref{sec:Npd}), there exists a unique, global strong solution $\ux_N$ to the Cauchy problem for \eqref{eq:SDE}. Moreover, with probability one, the particles $x_i^t$ and $x_j^t$ never collide on the interval $[0,\infty)$. Let $\mu^0\in \P(\R^d) \cap L^\infty$ be a probability measure with $L^\infty$ density with respect to Lebesgue measure. We abuse notation here and throughout the article by using the same symbol to denote both the measure and its density. 

As shown in \cref{prop:dcay} (more generally \cref{sec:mfe}), there is a unique, global solution to the Cauchy problem for \eqref{eq:lim} in the class $C([0,\infty); \P(\R^d)\cap L^\infty(\R^d))$. In the case of logarithmic interactions (i.e. $s=0$), we also assume that $\mu^0$ satisfies the logarithmic growth condition $\int_{\R^d}\log(1+|x|)d\mu^0(x)$, which is propagated locally uniformly by the evolution (see \cref{rem:log}).

Since $\ux_N$ is stochastic, $\{\Fr_N(\ux_N^t,\mu^t)\}_{t\geq 0}$ is a real-valued stochastic process. It is  straightforward to check from the non-collision of the particles and H\"older's inequality that $\Fr_N(\ux_N^t,\mu^t)$ is almost surely finite on the interval $[0,\infty)$ and a continuous process. Our main theorem is a quantitative estimate for the expected magnitude of $\Fr_N(\ux_N^t,\mu^t)$.

The first result of this article is the following functional inequality for the expected magnitude of the modulated energy. In the case $0<s<d-2$, we get a linear growth estimate, while in the case $s=0$, we have superlinear growth of size $O(t^{\frac{\sigma+C}{\sigma}})$ as $t\rightarrow\infty$.

\begin{thm}\label{thm:lmain}
Let $d\geq 3$, $0\leq s<d-2$, and $\sigma>0$. Let $\ux_N$ be a solution to the system \eqref{eq:SDE}, and let $\mu \in C([0,\infty); \P(\R^d)\cap L^\infty(\R^d))$ be a solution to the PDE \eqref{eq:lim}. If $s=0$, also assume that $\int_{\R^d}\log(1+|x|)d\mu^0(x)<\infty$. There exists a constant $C>0$ depending only $s,d,\sigma,\|\mu\|_{L^\infty}$, and the potential $\g$ through assumptions $\mathrm{\ref{ass0} - \ref{ass3'}}$ and an exponent $\beta>0$ depending only $s,d$, such that following holds. For all $t\geq 0$ and all $N$ sufficiently large depending on $\|\mu^0\|_{L^\infty}$, we have that
\begin{equation}\label{eq:lmain}
\begin{split}
\E(|\Fr_N(\ux_N^t,\mu^t)|) \leq C\paren*{1+t+t^{\frac{\sigma+C}{\sigma}}\indic_{s=0}}\paren*{|\Fr_N(\ux_N^0,\mu^0)| + N^{-\beta}}.
\end{split}
\end{equation}
\end{thm}

Assume now that $r_0 = \infty$ in assumption \ref{ass1}. In other words, the potential $\g$ is globally superharmonic, as opposed to just in a neighborhood of the origin. The second result of this article is a functional inequality for the expected magnitude of the modulated energy which yields a global bound in the case $0<s<d-2$. In the case $s=0$, we have an almost global bound, in the sense that the growth is $O(t^{{\frac{1}{\sigma}}^+})$ as $t\rightarrow\infty$, which can be arbitrarily small by choosing the diffusion strength $\sigma$ arbitrarily large.

\begin{thm}\label{thm:gmain}
Impose the same assumptions as \cref{thm:lmain} with the additional condition that $r_0=\infty$. Choose any exponent $\frac{d}{s+2}<p\leq\infty$. Then there exist constants $C,C_p>0$ depending on $s,d,\sigma,\|\mu\|_{L^\infty}$, and the potential $\g$ through assumptions  $\mathrm{\ref{ass0} - \ref{ass3'}}$ and exponents $\beta_p$ depending on $s,d$, such that the following holds. For all $t\geq 0$ and all $N$ sufficiently large depending on $\|\mu^0\|_{L^\infty}$, we have that
\begin{equation}\label{eq:gmain}
\E(|\Fr_N(\ux_N^t,\mu^t)|) \leq C\paren*{1+ \paren*{t^{\frac{1}{\sigma}}\log(1+t)}\indic_{s=0}}\paren*{|\Fr_N(\ux_N^0,\mu^0)| + C_pN^{-\beta_p}}.
\end{equation}
\end{thm}

We close this subsection with some remarks on the statements of \Cref{thm:lmain,thm:gmain}, further extensions, and interesting questions for future work.

\begin{remark}\label{rem:opt}
An examination of \Cref{ssec:linGron,ssec:globGron} will reveal to the interested reader the precise dependence of the constants $C,\beta$ ($C_p,\beta_p$) on the norms of $\mu$, $s$, $d$ (on $p$), and other underlying parameters. We have omitted the explicit dependence and simplified the statements of our final bounds (see \eqref{eq:linFin} and \eqref{eq:globFin}) in order to make the results more accessible to the reader.

Additionally, we have not attempted to optimize the regularity/integrability assumptions for $\mu$. One can show that in the case $s>0$, the linear and global bounds of \Cref{thm:lmain,thm:gmain}, respectively, still hold if we replace the $L^\infty$ assumption with $\mu \in L^p$ for finite $p$ sufficiently large depending on $s,d$. This, though, comes at the cost of slower decay in $N$.
\end{remark}

\begin{remark}
Sufficient conditions for $\Fr_N(\ux_N^0,\mu^0)$ to vanish as $N\rightarrow\infty$ are that the energy of \eqref{eq:SDE} converges to the energy of \eqref{eq:lim} and that $\mu_N^0\xrightharpoonup{*}\mu^0$ in the weak-* topology for $\P(\R^d)$. See \cite[Remark 1.2(c)]{Duerinckx2016} for more details.
\end{remark}

\begin{remark}\label{rem:pc}
It is well-known \cite[Proposition 2.4]{NRS2021} that the modulated energy $\Fr_N(\ux_N,\mu)$ controls convergence in negative-order Sobolev spaces. Note that since we are restricted to the sub-Coulombic setting, the extension implicit in the cited proposition can be ignored. Consequently, \Cref{thm:lmain,thm:gmain} yield a quantitative bound for the expected squared $H^s$ norm of $\mu_N-\mu$, for $s<-\frac{d+2}{2}$, of the form
\begin{equation}\label{eq:Sobbnd}
\E\paren*{\left\|\mu_N^t - \mu^t\right\|_{H^s}^2} \leq C\rho(t)\paren*{|\Fr_N(\ux_N^0,\mu^0)| + N^{-\beta}},
\end{equation}
where $\rho(t)$ is the time factor from either \cref{thm:lmain} or \cref{thm:gmain}. From this Sobolev convergence and standard arguments (see \cite[Section 1]{HM2014}), one deduces convergence in law of the empirical measure $\mu_N$ to $\mu$.

We can also deduce an explicit rate for propagation of chaos for the system \eqref{eq:SDE}. Indeed, suppose that $\ux_N^0$ are initially distributed in $(\R^d)^N$ according to some probability density $f_N^0$. Let $f_N^t$ denote the law of $\ux_N^t$, and let $f_{N,k}^t$ denote the $k$-particle marginal of $f_N^t$. Then using for instance  \cite[(7.21)]{RS2016}, we see that for any symmetric test function $\varphi\in C_c^\infty((\R^d)^k)$,
\begin{equation}\label{eq:pcdual}
\left|\int_{(\R^d)^k}\varphi d\paren*{f_{N,k}^t-(\mu^t)^{\otimes k}} \right| \leq Ck \sup_{\ux_{k-1}\in (\R^d)^{k-1}} \|\varphi(\ux_{k-1},\cdot)\|_{H^{-s}(\R^d)} \int_{(\R^d)^N} \|\mu_N^t - \mu^t\|_{H^s(\R^d)}df_N^0,
\end{equation}
for any $s<-\frac{d+2}{2}$. Combining \eqref{eq:pcdual} with \eqref{eq:Sobbnd} and using duality now yields an explicit rate for propagation of chaos in $H^{s}$ norm.
\end{remark}

\begin{remark}\label{rem:pot}
It is an interesting problem to obtain analogues of \Cref{thm:lmain,thm:gmain} when a confining potential $\mathsf{V}$ is added to the right-hand sides of \eqref{eq:SDE}, \eqref{eq:lim}. In this case, one has a nontrivial equilibrium for the equation \eqref{eq:lim} as $t\rightarrow\infty$, which clearly breaks our proof, in particular the Carlen-Loss argument used to prove \cref{prop:dcay} below. But we still might expect that the difference $\mu_N^t-\mu^t$ decays to zero in a suitable topology as $t\rightarrow\infty$ and that our argument may be salvaged by incorporating the long-time equilibrium for $\mu^t$. We plan to investigate this problem in future work. Note that by using so-called similarity variables (see \cite[Section 1]{GW2005}), one can obtain a local-in-time bound for the modulated energy from \Cref{thm:lmain,thm:gmain} in the model interaction case \eqref{eq:gmod} and with a quadratic potential. In fact, using our prior work \cite{NRS2021} and by appropriately modifying the similarity coordinates (see \cite[Section 8]{SV2014}), one can also obtain a local-in-time bound for the log or Riesz case $0\leq s<d$ in any dimension $d\geq 1$ with a quadratic potential and without noise.
\end{remark}

\subsection{Comparison with prior results}
At the time of completion of this manuscript, our \cref{thm:gmain} is, to the best of our knowledge, the first time that a quantitative rate of convergence for the mean-field limit of \eqref{eq:SDE} has been shown to hold uniformly on the interval $[0,\infty)$. Additionally, in the case $s=0$, the rate of convergence as $N\rightarrow\infty$ is (up to a $\log N$ factor) optimal. Existing results in the literature for singular potentials (e.g. \cite{JW2018, BJW2019edp}) have at least an exponential growth in time due to a reliance on a Gronwall-type argument without exploiting the dissipation of the limiting equation. Those works (e.g. \cite{Malrieu2003, CGM2008, DEGZ2020, DT2021}) that do have a uniform convergence result are restricted to regular potentials and require confinement.

Additionally, our theorem is at the level of the empirical measure for the original SDE dynamics for \eqref{eq:SDE}, as opposed to their associated \emph{Liouville/forward Kolmogorov equations} for the joint law of the process $\ux_N^t = (x_1^t,\ldots,x_N^t)$,
\begin{equation}
\label{eq:Lio}
\p_t f_N = -\sum_{i=1}^N \div_{x_i}\paren*{f_N\frac{1}{N}\sum_{1\leq i\neq j\leq N}\M\nabla\g(x_i-x_j)} + \sigma\sum_{i=1}^N\D_{x_i}f_N.
\end{equation}
Namely, no randomization of the initial data is needed, although as discussed in \cref{rem:pc}, our result implies convergence of this form as well. This stands in contrast to previous work \cite{JW2018, BJW2019edp} whose starting point is the Liouville equation \eqref{eq:Lio}.

The costs of the strong bounds we obtain with \Cref{thm:lmain,thm:gmain} are two-fold. First, we need somewhat stronger assumptions on the potential $\g$ than in \cite{JW2018, BJW2019edp}--especially the latter work. The reader may find a detailed comparison in \cite[Subsection 1.3]{NRS2021}. Second, and more importantly, our results are limited to the sub-Coulombic range $0\leq s < d-2$ and dimensions $d\geq 3$. The Coulomb case is barely just out of the reach. Indeed, the diligent reader will note that if $\g$ is the Coulomb potential, then $\D\g = -\d$, which is obviously no longer a function. Thus, the argument we described in the previous subsection to bound the second term in the right-hand side of \eqref{eq:MEid} no longer applies. The situation is even worse when $s>d-2$ as $\D\g > 0$, meaning what was previously a dissipation term should now cause the modulated energy to grow in time. We mention again that it is an open problem to prove the mean-field limit of \eqref{eq:SDE} in the conservative case and when $d-2<s<d$. 

During the final proofreading of the manuscript of this article, Guillin, Le Bris, and Monmarch\'{e} posted to the arXiv their preprint \cite{GlBM2021} showing uniform-in-time propagation of chaos in $L^1$ norm in the periodic setting $\T^d$ for singular interactions in the range $0\leq s\leq d-2$ and $d\geq 2$ with $N^{-\frac{1}{2}}$ rate. In particular, they can treat the viscous vortex model corresponding to the $d=2$ Coulomb case. Their method is very different from ours, as it is based on the relative-entropy method of Jabin-Wang \cite{JW2018}. Moreover, they assume random initial data and work at the level of the Liouville equation, and their interaction kernel is assumed to be integrable, have zero distributional divergence, and equal to the divergence of an $L^\infty$ matrix field. Their result is limited to the periodic setting, due to a need for compactness of domain, and to conservative flows. We also note that they impose much stronger regularity assumptions on their solutions to \eqref{eq:lim} than we do.

\subsection{Organization of article}
In \cref{sec:pre}, we review some estimates for Riesz potential operators and the interaction potential $\g$ that are used frequently in the paper.

\cref{sec:mfe} is devoted to the study of the limiting equation \eqref{eq:lim}, showing that it is globally well-posed and moreover the $L^p$ norms of solutions satisfy the optimal decay bounds (see \Cref{prop:lwp,prop:dcay}).

 In \cref{sec:Npd}, we show that the $N$-particle system \eqref{eq:SDE} has well-defined dynamics (see \cref{prop:Nwp}) in the sense that there exists a unique strong solution and with probability one, the particles never collide. We also introduce in this section a truncation and stopping time procedure that will be used again later in \cref{sec:ev}. 
 
  In \cref{sec:ME}, we review properties of the modulated energy and renormalized commutator estimates from the perspective of our recent joint work \cite{NRS2021}. We also prove refinements (see \cref{ssec:MEnew}) of the results from that work in the case where $\g$ is globally superharmonic. 
  
  Finally, in \Cref{sec:ev,sec:Gron}, we prove our main results, \Cref{thm:lmain,thm:gmain}. \cref{sec:ev} gives the rigorous computation of the It\^o equation (cf. \eqref{eq:MEid}) satisfied by the process $\Fr_N(\ux_N^t,\mu^t)$. The main result is \cref{prop:MEgb}, which establishes an integral inequality for $\E(|\Fr_N(\ux_N^t,\mu^t)|)$. Using this inequality together with the decay bound of \cref{prop:dcay} and the results of \cref{sec:ME}, we close our Gronwall argument in \cref{sec:Gron}, completing the proofs of \Cref{thm:lmain,thm:gmain}.

\subsection{Notation}
\label{ssec:preN}
We close the introduction with the basic notation used throughout the article without further comment.

Given nonnegative quantities $A$ and $B$, we write $A\lesssim B$ if there exists a constant $C>0$, independent of $A$ and $B$, such that $A\leq CB$. If $A \lesssim B$ and $B\lesssim A$, we write $A\sim B$. To emphasize the dependence of the constant $C$ on some parameter $p$, we sometimes write $A\lesssim_p B$ or $A\sim_p B$. Also $(\cdot)_+$ denotes the positive part of a number.

We denote the natural numbers excluding zero by $\N$ and including zero by $\N_0$. Similarly, we denote the positive real numbers by $\R_+$.  Given $N\in\N$ and points $x_{1,N},\ldots,x_{N,N}$ in some set $X$, we will write $\ux_N$ to denote the $N$-tuple $(x_{1,N},\ldots,x_{N,N})$. Given $x\in\R^n$ and $r>0$, we denote the ball and sphere centered at $x$ of radius $r$ by $B(x,r)$ and $\p B(x,r)$, respectively. Given a function $f$, we denote the support of $f$ by $\supp f$. We use the notation $\nabla^{\otimes k}f$ to denote the tensor with components $\p_{i_1\cdots i_k}^k f$.

We denote the space of Borel probability measures on $\R^n$ by $\P(\R^n)$. When $\mu$ is in fact absolutely continuous with respect to Lebesgue measure, we shall abuse notation by writing $\mu$ for both the measure and its density function. We denote the Banach space of complex-valued continuous, bounded functions on $\R^n$ by $C(\R^n)$ equipped with the uniform norm $\|\cdot\|_{\infty}$. More generally, we denote the Banach space of $k$-times continuously differentiable functions with bounded derivatives up to order $k$ by $C^k(\R^n)$ equipped with the natural norm, and we define $C^\infty \coloneqq \bigcap_{k=1}^\infty C^k$. We denote the subspace of smooth functions with compact support by $C_c^\infty(\R^n)$. We denote the Schwartz space of functions by $\Sc(\R^n)$ and the space of tempered distributions by $\Sc'(\R^n)$.

\subsection{Acknowledgments}
The second author thanks Eric Vanden-Eijnden for helpful comments.

\section{Potential estimates}
\label{sec:pre}
We review some facts about Riesz potential estimates. For a more thorough discussion, we refer to \cite{Stein1970, Stein1993, Grafakos2014c, Grafakos2014m}.

Let $d\geq 1$ For $s>-d$, we define the Fourier multiplier $|\nabla|^{s} = (-\Delta)^{s/2}$ by
\begin{equation}
((-\Delta)^{s/2}f) \coloneqq (|\cdot|^{s}\wh{f}(\cdot))^\vee, \qquad f\in \Sc(\R^d).
\end{equation}
Since, for $s\in (-d,0)$, the inverse Fourier transform of $|\xi|^s$ is the tempered distribution
\begin{equation}
\frac{2^s\Gamma(\frac{d+s}{2})}{\pi^{\frac{d}{2}}\Gamma(-\frac{s}{2})} |x|^{-s-d},
\end{equation}
it follows that
\begin{equation}
((-\Delta)^{s/2}f)(x) = \frac{2^s \Gamma(\frac{d+s}{2})}{\pi^{\frac{d}{2}}\Gamma(-\frac{s}{2})}\int_{\R^d}\frac{f(y)}{|x-y|^{s+d}}dy \qquad \forall x\in\R^d.
\end{equation}
For $s\in (0,d)$, we define the \emph{Riesz potential operator} of order $s$ by $\mathcal{I}_s \coloneqq (-\Delta)^{-s/2}$ on $\Sc(\R^d)$.

\begin{remark}
If $0<s<d$, we see that the model potential \eqref{eq:gmod} corresponds to $\g$ is a constant times  the kernel of $\mathcal{I}_{d-s}$. If $s=0$, then $\g$ is a multiple of the inverse Fourier transform of the tempered distribution $\PV |\xi|^{-d}-c\d_{0}(\xi)$, for some constant $c$. The subtraction of the Dirac mass is to cure the singularity of $|\xi|^{-d}$ near the origin.
\end{remark}

$\mathcal{I}_s$ extends to a well-defined operator on any $L^p$ space, the extension also denoted by $\mathcal{I}_s$ with an abuse of  notation, as a consequence of the \emph{Hardy-Littlewood-Sobolev (HLS) lemma}.

\begin{prop}
\label{prop:HLS}
Let $d\geq 1$, $s \in (0,d)$, and $1<p<q<\infty$ satisfy the relation
\begin{equation}
\frac{1}{p}-\frac{1}{q} = \frac{s}{d}.
\end{equation}
Then for all $f\in\Sc(\R^d)$,
\begin{align}
\|\mathcal{I}_s(f)\|_{L^q} &\lesssim \|f\|_{L^p},\\
\|\mathcal{I}_s(f)\|_{L^{\frac{d}{d-s},\infty}} &\lesssim \|f\|_{L^1},
\end{align}
where $L^{r,\infty}$ denotes the weak-$L^r$ space. Consequently, $\mathcal{I}_s$ has a unique extension to $L^p$, for all $1\leq p<\infty$.
\end{prop}

The next lemma allows us to control the $L^\infty$ norm of $\mathcal{I}_s(f)$ in terms of the $L^1$ norm and $L^p$ norm, for some $p$ depending on $s,d$. We omit the proof as it is a straightforward application of H\"older's inequality.

\begin{lemma}
\label{lem:LinfRP}
Let $d\geq 1$, $s\in (0,d)$, and $\frac{d}{s}<p\leq\infty$. Then for all $f\in L^1(\R^d)\cap L^p(\R^d)$, it holds that $\mathcal{I}_s(f) \in C(\R^d)$ and
\begin{equation}
\|\mathcal{I}_s(f)\|_{L^\infty} \lesssim \|f\|_{L^{1}}^{1-\frac{d-s}{d(1-\frac{1}{p})}} \|f\|_{L^{p}}^{\frac{d-s}{d(1-\frac{1}{p})}}.
\end{equation}
\end{lemma}

When the convolution with a Riesz potential is replaced by convolution with the $\log$ potential, we have an analogue of \cref{lem:LinfRP}.

\begin{lemma}\label{lem:Linflog}
Let $d\geq 1$ and $1<p\leq\infty$. For all $f\in L^p(\R^d)$ with $\int_{\R^d}\log(1+|x|)|f(x)|dx<\infty$, it holds that $\log|\cdot|\ast f \in C_{loc}(\R^d)$ and
\begin{equation}
\left|(\log|\cdot|\ast f)(x)\right| \lesssim (1+|x|)^{\frac{d(p-1)}{p}}\log(1+|x|) + \int_{\R^d}\log(1+|y|)|f(y)|dy \qquad \forall x\in\R^d.
\end{equation}
\end{lemma}

\begin{remark}\label{rem:Linfg}
If $0\leq s<d$, then for any integer $1\leq k< d-s$, assumption \ref{ass2} implies that $|\nabla^{\otimes k}\g|$ is bounded from above by a constant multiple of the kernel of $\mathcal{I}_{d-s-k}$. \cref{lem:LinfRP} implies that
\begin{equation}
\|\nabla^{\otimes k}\g\ast f\|_{L^\infty} \lesssim \|f\|_{L^1}^{1-\frac{s+k}{d}} \|f\|_{L^\infty}^{\frac{s+k}{d}}, \qquad f\in L^1(\R^d) \cap L^\infty(\R^d).
\end{equation}
We shall use this estimate frequently in the sequel.
\end{remark}
 
\section{The mean-field equation}\label{sec:mfe}
We start by discussing the well-posedness of the Cauchy problem for and asymptotic decay of solutions to the mean-field PDE \eqref{eq:lim}. The latter property is strictly a consequence of the diffusion and is the crucial ingredient to obtain a rate of convergence for the mean-field limit beyond the standard exponential bound given by the Gronwall-Bellman lemma. The results of this section are perhaps mathematical folklore. We present them not for claim for originality but since we could not find them conveniently stated in the literature.

\subsection{Local well-posedness}\label{ssec:mfelwp}
We start by proving local well-posedness for the equation \eqref{eq:lim} for initial data $\mu^0$ in the Banach space $X\coloneqq L^1(\R^d)\cap L^\infty(\R^d)$. That is, we show existence, uniqueness, and continuous dependence on the initial data. The proof proceeds by a contraction mapping argument for the mild formulation of \eqref{eq:lim}. In the next subsection, we will upgrade this local well-posedness to global well-posedness through estimates for the temporal decay of the $L^p$ norms of the solution.

Let us introduce the mild formulation of equations \eqref{eq:lim}, on which we base our notion of solution. With $e^{t\Delta}$ denoting the heat flow, we write
\begin{equation}\label{eq:mild}
\mu^t = e^{t\sigma\Delta}\mu^0 - \int_0^t e^{(t-\ka)\sigma\Delta}\div(\mu^\ka \M\nabla\g\ast\mu^\ka )d\ka.
\end{equation}

\begin{prop}\label{prop:lwp}
Suppose $d\geq 3$ and $0\leq s<d-2$. Let $\mu^0\in X\coloneqq L^1(\R^d)\cap L^\infty(\R^d)$, and let $R>0$ be such that $\|\mu^0\|_{X}\leq R$. Then there exists a unique solution $\mu \in C([0,T];X)$ to equation \eqref{eq:mild} such that $T\sim \sigma R^{-2}$ and
\begin{equation}
\|\mu\|_{C([0,T];X)} \leq 2R.
\end{equation}
Moreover, if $\|\mu_1^0\|_{X}, \|\mu_2^0\|_{X} \leq R$, then there exists $T' \sim \sigma R^{-2}$ such that their associated solutions $\mu_1,\mu_2$ satisfy the bound
\begin{equation}
\|\mu_1-\mu_2\|_{C([0,T'];X)} \leq 2\|\mu_1^0-\mu_2^0\|_X.
\end{equation}
\end{prop}
\begin{proof}
Let $R>0$, let $\mu^0\in X$ with $\|\mu^0\|_X \leq R$, and let $\Tc$ denote the map
\begin{equation}
\mu \mapsto e^{(\cdot)\sigma\Delta}\mu^0 - \int_0^{(\cdot)} e^{(\cdot-\ka)\sigma\Delta}\div(\mu^\ka\M\nabla\g\ast\mu^\ka)d\ka.
\end{equation}
We claim that for any appropriate choice of $T$, this map is a contraction on the closed ball of radius $2R$ in $C([0,T]; X)$. Indeed,
\begin{align}
\|\Tc\mu\|_{C([0,T];X)} &\leq \|\mu^0\|_{X} + \int_0^T \|e^{(\cdot-\ka)\sigma\D}\div(\mu^\ka\M\nabla\g\ast\mu^\ka)\|_{C([0,T];X)} d\ka \nn\\
&\leq R + C\int_0^T (\sigma\ka)^{-1/2} \|\mu\M\nabla\g\ast\mu\|_{C([0,T];X)}d\ka \nn\\
&\leq R + C (T/\sigma)^{1/2}\|\mu\M\nabla\g\ast\mu\|_{C([0,T];X)}.
\end{align}
By H\"older's inequality and \cref{rem:Linfg},
\begin{equation}
\label{eq:hi}
\|\mu^t\M\nabla\g\ast\mu^t\|_{L^p} \leq \|\mu^t\|_{L^p} \|\nabla\g\ast\mu^t\|_{L^{\infty}} \lesssim \|\mu^t\|_{L^p} \|\mu^t\|_{L^1}^{1-\frac{s+1}{d}} \|\mu^t\|_{L^\infty}^{\frac{s+1}{d}}.
\end{equation}
for any exponent $1\leq p\leq\infty$. Consequently, if $\|\mu\|_{C([0,T]; X)} \leq 2R$, then
\begin{equation}
\|\mu\M\nabla\g\ast\mu\|_{C([0,T];X)} \lesssim R^2,
\end{equation}
which implies that
\begin{equation}
\|\Tc\mu\|_{C([0,T];X)} \leq R + C (T/\sigma)^{1/2} R^2,
\end{equation}
for some constant $C$ depending only on $s,d$ and the potential $\g$ through assumption \ref{ass2}. Similarly, for $\mu_1,\mu_2 \in B_{2R}\subset C([0,T];X)$,
\begin{multline}\label{eq:Tdiff}
\|\Tc\mu_1 - \Tc\mu_2\|_{C([0,T];X)} \lesssim (T/\sigma)^{1/2}\Big(\|(\mu_1-\mu_2)\nabla\M\g\ast\mu_1\|_{C([0,T];X)} \\
+ \|\mu_2\M\nabla\g\ast(\mu_1-\mu_2)\|_{C([0,T];X)}\Big).
\end{multline}
Using inequality \eqref{eq:hi}, the preceding right-hand side is $\lesssim$
\begin{equation}
(T/\sigma)^{1/2}R\|\mu_1-\mu_2\|_{C([0,T];X)}.
\end{equation}
After a little bookkeeping, we see that there is a constant $C>0$ such that if
\begin{equation}
C(T/\sigma)^{1/2}R < 1,
\end{equation}
then $\Tc$ is indeed a contraction on the closed ball $B_{2R}$. Consequently, the contraction mapping theorem implies there is a unique solution to equation \eqref{eq:mild} in $C([0,T];X)$.

We can also prove Lipschitz-continuous dependence on the initial data. Indeed, let $\|\mu_{i}^0\|_{X}\leq R$ for $i=1,2$. Then the preceding result tells us there exist unique solutions $\mu_i$ in $C([0,T];X)$ for some $T\sim \sigma R^{-2}$ and that $\|\mu_{i}\|_{C([0,T];X)} \lesssim R$. Using inequality \eqref{eq:Tdiff}, we find that
\begin{equation}
\|\mu_1 - \mu_2\|_{C([0,T];X)} \leq \|\mu_{1}^0-\mu_{2}^0\|_{X} + C (T/\sigma)^{1/2}R\|\mu_1-\mu_2\|_{C([0,T];X)}.
\end{equation}
Provided that $C(T/\sigma)^{1/2}R < 1$, we have the bound
\begin{equation}
\label{eq:Ldep}
\|\mu_1 - \mu_2\|_{C([0,T];X)} \leq  (1-C(T/\sigma)^{1/2}R)^{-1} \|\mu_{1}^0-\mu_{2}^0\|_{X}.
\end{equation}
\end{proof}

\begin{remark}\label{rem:Schw}
By a Gronwall argument for the energy
\begin{equation}
\sum_{k=0}^n\int_{\R^d}(1+|x|^2)^{m}|\nabla^{\otimes k}\mu(x)|^2 dx
\end{equation}
for arbitrarily large integers $m,n\in\N$, it is easy to see that if the initial datum $\mu^0\in\Sc(\R^d)$, then it remains spatially Schwartz on its lifespan. This property combined with the dependence bound \eqref{eq:Ldep} allows to approximate solutions in $C([0,T];X)$ by Schwartz-class solutions.
\end{remark}

\begin{remark}\label{rem:mass}
For a Schwartz-class solution $\mu$, equation \eqref{eq:lim} and the divergence theorem yield
\begin{equation}
\frac{d}{dt} \int_{\R^d}\mu^t dx = \int_{\R^d}\paren*{-\div(\mu^t\M\nabla\g\ast\mu^t)+\sigma\D\mu^t}dx = 0.
\end{equation}
So by approximation, solutions $\mu \in C([0,T];X)$ obey conservation of mass.
\end{remark}

\begin{remark}\label{rem:Lp}
If $\mu\in C([0,T];X)$, then for any $1\leq p\leq \infty$, it holds that $\|\mu^t\|_{L^p} \leq \|\mu^{t'}\|_{L^p}$ for all $0\leq t'\leq t\leq T$. Indeed, suppose that $\mu$ is Schwartz-class and $p\geq 1$ is finite. Then using equation \eqref{eq:lim}, we see that
\begin{align}
\frac{d}{dt} \|\mu^t\|_{L^p}^p &= -p\int_{\R^d} |\mu^t|^{p-2}\mu^t \div\paren*{\mu^t\M\nabla\g\ast\mu^t}dx +p\sigma\int_{\R^d}|\mu^t|^{p-2}\mu^t \D\mu^tdx.
\end{align}
It follows from integration by parts and the product rule that
\begin{equation}
\int_{\R^d}|\mu^t|^{p-2}\mu^t \D\mu^tdx = -\int_{\R^d}\paren*{(p-2)(|\mu^t|^{p-4}\mu^t\nabla\mu^t)\mu^t + |\mu^t|^{p-2}\nabla\mu^t}\cdot\nabla\mu^t dx \leq 0.
\end{equation}
Similarly,
\begin{multline}
-\int_{\R^d} |\mu^t|^{p-2}\mu^t \div\paren*{\mu^t\M\nabla\g\ast\mu^t}dx\\
= \int_{\R^d}\paren*{(p-2)(|\mu^t|^{p-4}\mu^t\nabla\mu^t)\mu^t + |\mu^t|^{p-2}\nabla\mu^t}\cdot\mu^t\M\nabla\g\ast\mu^t dx.
\end{multline}
Writing $(|\mu^t|^{p-4}\mu^t\nabla\mu^t)(\mu^t)^2 = p^{-1}\nabla(|\mu^t|^{p})$ and $|\mu^t|^{p-2}\mu^t\nabla\mu^t = p^{-1}\nabla(|\mu^t|^p)$, it follows from integrating by parts that
\begin{align}
-p\int_{\R^d} |\mu^t|^{p-2}\mu^t \div\paren*{\mu^t\M\nabla\g\ast\mu^t}dx =
-(p-1)\int_{\R^d} |\mu^t|^p \div(\M\nabla\g\ast\mu^t)dx \leq0 
\end{align}
where the final inequality follows from assumption \ref{ass5}. This takes care of the case $p<\infty$. For $p=\infty$, we take the limit $p\rightarrow\infty^{-}$.
\end{remark}

\begin{remark}\label{rem:gwp}
Since \cref{rem:Lp} implies the $L^1$ and $L^\infty$ norms of solutions are nonincreasing and the time of existence in \cref{prop:lwp} is proportional to $\|\mu^0\|_{X}^{-2}$, it follows from iterating \cref{prop:lwp} that solutions exist globally in $C([0,\infty); X)$.
\end{remark}

\begin{remark}\label{rem:log}
Let $\mu$ be a nonnegative Schwartz-class solution to \eqref{eq:lim}. Then using equation \eqref{eq:lim}, integrating by parts, and using the chain rule,
\begin{align}
\frac{d}{dt}\int_{\R^d}\log(1+|x|)\mu^t(x)dx &= -\int_{\R^d}\log(1+|x|)\div\paren*{\mu^t(\M\nabla\g\ast\mu^t)}(x)dx \nn\\
&\ph + \sigma \int_{\R^d}\log(1+|x|)\D\mu^t(x)dx \nn\\
&= \int_{\R^d}\frac{x}{|x|(1+|x|)}\cdot \M\nabla\g\ast\mu^t(x)\mu^t(x)dx \nn\\
&\ph -\sigma\int_{\R^d}(1+|x|)^{-2}\mu^t(x)dx.
\end{align}
Hence, for any $T>0$,
\begin{align}
\int_{\R^d}\log(1+|x|)\mu^T(x)dx &\leq \int_{\R^d}\log(1+|x|)\mu^0(x)dx + T\sup_{0\leq t\leq T}\|\nabla\g\ast\mu^t\|_{L^\infty} \|\mu^t\|_{L^1} \nn\\
&\lesssim  \int_{\R^d}\log(1+|x|)\mu^0(x)dx  + T\sup_{0\leq t\leq T}\|\mu^t\|_{L^1}^{2-\frac{s+1}{d}} \|\mu^t\|_{L^\infty}^{\frac{s+1}{d}}\nn\\
&\leq \int_{\R^d}\log(1+|x|)\mu^0(x)dx + T\|\mu^0\|_{L^1}^{2-\frac{s+1}{d}} \|\mu^0\|_{L^\infty}^{\frac{s+1}{d}},
\end{align}
where the penultimate line follows from \cref{rem:Linfg} and the ultimate line from the nonincreasing property of $L^p$ norms. By approximation and continuous dependence, it follows that if $\mu^0\in X$ satisfies $\int_{\R^d}\log(1+|x|)\mu^0(x)dx<\infty$, then $\mu^t$ does as well for all $t>0$. 
\end{remark}

\begin{remark}
Using assumption \ref{ass5}, it is not hard to also show that the minimum value of $\mu^t$ is nondecreasing in time. Consequently, if $\mu^0\geq 0$, then $\mu\geq 0$ on its lifespan.
\end{remark}

\subsection{Asymptotic decay}
\label{ssec:mfedk}
We now show the $L^p$ norms of the solutions obtained in previous subsection satisfy the same temporal decay estimates as the linear heat equation. This follows the method of \cite{CL1995} and extends it to non divergence-free vector fields. 

\begin{prop}\label{prop:dcay}
Suppose that $\mu \in C([0,\infty);X)$ is a solution to equation \eqref{eq:lim}. If $\M:\nabla^{\otimes 2}\g\neq 0$, then assume that $\mu\geq 0$. Let $1\leq p\leq q\leq\infty$. Then for all $t>0$,
\begin{equation}
\label{eq:dcay}
\|\mu^t\|_{L^q} \leq \paren*{\frac{K(q)}{K(p)}}^{\frac{d}{2}} \paren*{\frac{4\pi\sigma t}{1/p-1/q}}^{-\frac{d}{2}(\frac{1}{p}-\frac{1}{q})} \|\mu^0\|_{L^p},
\end{equation}
where
\begin{equation}
K(r) \coloneqq \frac{{r'}^{1/r'}}{r^{1/r}}, \qquad 1\leq r\leq\infty.
\end{equation}
\end{prop}

In the conservative case, the velocity field $\M\nabla\g\ast\mu$ is divergence-free and therefore \cref{prop:dcay} follows from the seminal work \cite{CL1995} of Carlen and Loss. In the general case where we only know that the matrix $\M$ is negative semidefinite (i.e. condition \eqref{eq:Mnd} holds), we make a small modification of their proof. As it is a crucial ingredient, we first recall the sharp form of Gross's log-Sobolev inequality \cite{Gross1975, CL1990, Carlen1991}. \cref{prop:lsob} below is reproduced from \cite{CL1995} (see equation (1.17) in that work).

\begin{prop}
\label{prop:lsob}
Let $a>0$. Then for all $f\in H^1(\R^d)$,
\begin{equation}
\int_{\R^d} |f(x)|^2\log\paren*{\frac{|f(x)|^2}{\|f\|_{L^2}^2}}dx + \paren*{d+\frac{d\log a}{2}}\int_{\R^d}|f(x)|^2dx \leq\frac{a}{\pi}\int_{\R^d} |\nabla f(x)|^2 dx.
\end{equation}
Moreover, equality holds if and only if $f$ is a scalar multiple and translate of $f_a(x)\coloneqq a^{-d/4}e^{-\pi |x|^2/2a}$.
\end{prop}

\begin{proof}[Proof of \cref{prop:dcay}]
Using \cref{rem:Schw} and continuous dependence on the initial data, we may assume without loss of generality that $\mu$ is spatially Schwartz on its lifespan and $\mu$ is $C^\infty$ in time. Therefore, there are no issues of regularity or decay in justifying the computations to follow. Additionally, let us rescale time by defining $\mu_\sigma(t,x)\coloneqq \mu(t/\sigma,x)$, which now satisfies the equation
\begin{equation}\label{eq:musig}
\p_t\mu_\sigma = -\sigma^{-1}\div(\mu_\sigma\M\nabla\g\ast\mu_\sigma) + \D\mu_{\sigma}.
\end{equation}
It suffices to show
\begin{equation}
\|\mu_\sigma^t\|_{L^q} \leq \paren*{\frac{K(q)}{K(p)}}^{\frac{d}{2}}\paren*{\frac{4\pi t}{1/p - 1/q}}^{-\frac{d}{2}(\frac{1}{p}-\frac{1}{q})} \|\mu_\sigma^0\|_{L^p},
\end{equation}
since replacing $t$ with $\sigma t$ yields the desired result. To simplify the notation, we drop the $\sigma$ subscript in what follows and assume that $\mu$ solves equation \eqref{eq:musig}.

For given $p,q$ as above, let $r:[0,T]\rightarrow [p,q]$ be a $C^1$ increasing function to be specified momentarily. Replacing the absolute value $|\cdot|$ with $(\vep^2+|\cdot|^2)^{1/2}$, differentiating, then sending $\vep\rightarrow 0^+$, we find that
\begin{equation}
\begin{split}
r(t)^2\|\mu^t\|_{L^{r(t)}}^{r(t)-1}\frac{d}{dt}\|\mu^t\|_{L^{r(t)}} &= \dot{r}(t)\int_{\R^d}|\mu^t|^{r(t)} \log\paren*{\frac{|\mu^t|^{r(t)}}{\|\mu^t\|_{L^{r(t)}}^{r(t)}}} dx \\
&\ph + r(t)^2\int_{\R^d}|\mu^t|^{r(t)-1}\sgn(\mu^t)\p_t\mu^t dx. \label{eq:CLrhs}
\end{split}
\end{equation}
Above, we have used the calculus identity
\begin{equation}
\frac{d}{dt} x(t)^{y(t)} = \dot{y}(t)x(t)^{y(t)}\log x(t) + y(t)\dot{x}(t)x(t)^{y(t)-1}
\end{equation}
for $C^1$ functions $x(t)>0$ and $y(t)$. Substituting in equation \eqref{eq:lim}, the right-hand side of \eqref{eq:CLrhs} equals
\begin{equation}
\begin{split}
&\dot{r}(t)\int_{\R^d}|\mu^t|^{r(t)} \log \paren*{\frac{|\mu^t|^{r(t)}}{\|\mu^t\|_{L^{r(t)}}^{r(t)}}} dx +  r(t)^2\int_{\R^d}|\mu^t|^{r(t)-1}\sgn(\mu^t)\D\mu^tdx \\
&\ph - \frac{r(t)^2}{\sigma}\int_{\R^d}|\mu^t|^{r(t)-1}\sgn(\mu^t)\div(\mu^t \M\nabla\g\ast\mu^t)dx.
\end{split}
\end{equation}
By the product rule,
\begin{multline}\label{eq:divbrhs}
-r(t)^2\int_{\R^d}|\mu^t|^{r(t)-1}\sgn(\mu^t)\div(\mu^t\M\nabla\g\ast\mu^t)dx \\
= -r(t)^2\int_{\R^d}|\mu^t|^{r(t)}\div(\M\nabla\g\ast\mu^t)dx  - r(t)^2\int_{\R^d}|\mu^t|^{r(t)-1}\sgn(\mu^t)\nabla\mu^t\cdot (\M\nabla\g\ast\mu^t)dx.
\end{multline}
We recognize
\begin{equation}
r(t)|\mu^t|^{r(t)-1}\sgn(\mu^t)\nabla\mu^t = \nabla(|\mu^t|^{r(t)}).
\end{equation}
Therefore integrating by parts, the second term in the right-hand side of \eqref{eq:divbrhs} equals
\begin{equation}
r(t)\int_{\R^d}|\mu^t|^{r(t)}\div(\M\nabla\g\ast\mu^t)dx.
\end{equation}
Thus, equality \eqref{eq:divbrhs} and assumption \ref{ass5} (and that $\mu\geq 0$ by assumption if $\M:\nabla^{\otimes 2}\g$ does not vanish on $\R^d\setminus\{0\}$) now give
\begin{align}
- \frac{r(t)^2}{\sigma}\int_{\R^d}|\mu^t|^{r(t)-1}\sgn(\mu^t)\div(\mu^t \M\nabla\g\ast\mu^t)dx &= -\frac{r(t)(r(t)-1)}{\sigma}\int_{\R^d}|\mu^t|^{r(t)} (\M:\nabla^{\otimes 2}\g\ast\mu^t) dx\nn\\
&\leq 0.
\end{align}
Finally, write
\begin{equation}
\sgn(\mu^t) = \lim_{\vep\rightarrow 0^+}\frac{\mu^t}{\sqrt{\vep^2 + |\mu^t|^2}}.
\end{equation}
Integrating by parts,
\begin{align}
&r(t)^2\int_{\R^d}|\mu^t|^{r(t)-1}\sgn(\mu^t)\D\mu^tdx \nn\\
&= \lim_{\vep\rightarrow 0^+} \Bigg(-r(t)^2(r(t)-1)\int_{\R^d} |\mu^t|^{r(t)-2}\frac{|\mu^t|}{\sqrt{\vep^2 + |\mu^t|^2}}|\nabla\mu^t|^2dx \nn\\
&\ph  -r(t)^2\int_{\R^d} |\mu^t|^{r(t)-1}\frac{|\nabla\mu^t|^2}{\sqrt{\vep^2+|\mu^t|^2}}dx  + r(t)^2\int_{\R^d} |\mu^t|^{r(t)-1}\frac{|\mu^t|^2 |\nabla \mu^t|^2}{(\vep^2 + |\mu^t|^2)^{3/2}} dx\Bigg) \nn\\
&= -  r(t)^2(r(t)-1)\int_{\R^d} |\mu^t|^{r(t)-2} |\nabla\mu^t|^2dx \nn\\
&= - 4(r(t)-1)\int_{\R^d} |\nabla |\mu^t|^{r(t)/2} |^2 dx.
\end{align}
After a little bookkeeping, we realize that we have shown
\begin{equation}
\begin{split}
r(t)^2\|\mu^t\|_{L^{r(t)}}^{r(t)-1}\frac{d}{dt}\|\mu^t\|_{L^{r(t)}} &\leq \dot{r}(t)\int_{\R^d}|\mu^t|^{r(t)} \log\paren*{\frac{|\mu^t|^{r(t)}}{\|\mu^t\|_{L^{r(t)}}^{r(t)}}} dx \\
&\ph - 4 (r(t)-1)\int_{\R^d} |\nabla |\mu^t|^{r(t)/2} |^2 dx.
\end{split}
\end{equation}
The remainder of the proof follows that of Carlen and Loss. We include the details for the sake of completeness.

We apply \cref{prop:lsob} pointwise in time with choice $a = \frac{4\pi(r(t)-1)}{\dot{r}(t)}$ and $f= |\mu^t|^{r(t)/2}$ to obtain that
\begin{equation}
r(t)^2\|\mu^t\|_{L^{r(t)}}^{r(t)-1}\frac{d}{dt}\|\mu^t\|_{L^{r(t)}} \leq -\dot{r}(t)\paren*{d+\frac{d}{2}\log\paren*{\frac{4\pi(r(t)-1)}{\dot{r}(t)}}}\|\mu^t\|_{L^{r(t)}}^{r(t)}.
\end{equation}
Implicit here is the requirement that $\dot{r}(t)>0$. Define the function
\begin{equation}
G(t) \coloneqq \log\paren*{\|\mu^t\|_{L^{r(t)}}}.
\end{equation}
Then
\begin{align}
\frac{d}{dt}G(t) = \frac{1}{ \|\mu^t\|_{L^{r(t)}}}\frac{d}{dt} \|\mu^t\|_{L^{r(t)}} \leq -\frac{\dot{r}(t)}{r(t)^2}\paren*{d+\frac{d}{2}\log\paren*{\frac{4\pi(r(t)-1)}{\dot{r}(t)}}} .
\end{align}
Set $s(t)\coloneqq 1/r(t)$, so that the preceding inequality becomes, after writing $\frac{r-1}{\dot{r}}= -\frac{s(1-s)}{\dot{s}}$,
\begin{equation}
\frac{d}{dt}G(t) \leq \dot{s}(t)\paren*{d+\frac{d}{2}\log\paren*{4\pi s(t)(1-s(t))}} + \frac{d}{2}(-\dot{s}(t))\log(-\dot{s}(t)).
\end{equation}
So by the fundamental theorem of calculus,
\begin{equation}
\label{eq:Gexp}
\begin{split}
G(t)-G(0)  &\leq \int_0^T\dot{s}(t)\paren*{d+\frac{d}{2}\log\paren*{4\pi s(t)(1-s(t))}}dt \\
&\ph - \frac{d}{2}\int_0^T \dot{s}(t)\log\paren*{-\dot{s}(t)}dt.
\end{split}
\end{equation}
We require that $s(0) = 1/p$ and $s(T) = 1/q$, so by the fundamental theorem of calculus,
\begin{equation}
\int_0^T\dot{s}(t)\paren*{d+\frac{d}{2}\log(4\pi s(t)(1-s(t)))}dt = \frac{d}{2}\paren*{\log(4\pi)s + \log\paren*{s^s(1-s)^{-(1-s)}}}|_{s=1/p}^{s=1/q}.
\end{equation}
Using the convexity of $a \mapsto a\log a$, we minimize the second integral by choosing $s(t)$ to linearly interpolate between $s(0)=1/p$ and $s(T) = 1/q$, i.e.
\begin{equation}
\dot{s}(t) = \frac{1}{T}\paren*{\frac{1}{q}-\frac{1}{p}}, \qquad 0\leq t\leq T.
\end{equation}
Thus,
\begin{equation}
-\frac{d}{2}\int_0^T \dot{s}(t)\log\paren*{-\dot{s}(t)}dt = -\frac{d}{2}\paren*{\frac{1}{p}-\frac{1}{q}}\log\paren*{\frac{T}{1/p-1/q}}.
\end{equation}
The desired conclusion now follows from a little bookkeeping and exponentiating both sides of the inequality \eqref{eq:Gexp}.
\end{proof}

\begin{remark}
Our extension of the Carlen-Loss \cite{CL1995} method to non-divergence-free vector fields is not limited to proving optimal decay estimates. In fact, it seems we have come across a more general property for which certain parabolic theory is valid. For example, suppose one considers linear equations of the form
\begin{equation}\label{eq:lpara}
\begin{cases}
\p_t\mu = \D\mu + \div(b\mu) + c\mu \\
\mu^t|_{t=s} = \mu^s
\end{cases}
\qquad (t,x) \in (s,\infty) \times\R^d,
\end{equation}
where $b$ is a vector field and $c$ is a scalar, for simplicity both assumed to be $C^\infty$. If $\div b\leq 0$, then using the same reasoning as in the proof of \cref{prop:dcay}, one can show that the solution $\mu$ to \eqref{eq:lpara} satisfies the decay estimates
\begin{equation}\label{eq:ldcay}
\|\mu^t\|_{L^q} \leq \paren*{\frac{K(q)}{K(p)}}^{\frac{d}{2}} \paren*{\frac{4\pi(t-s)}{1/p-1/q}}^{-\frac{d}{2}(\frac{1}{p}-\frac{1}{q})}e^{\int_s^t \|c^\ka\|_{L^\infty}d\ka} \|\mu^s\|_{L^p} \qquad \forall 0<s\leq t<\infty.
\end{equation}
Now one can easily show that equation \eqref{eq:lpara} has a (smooth) fundamental solution, and a well-known problem in parabolic theory is to obtain Gaussian upper and lower bounds for such fundamental solutions, since such bounds imply H\"older continuity of the fundamental solution by an argument of Nash \cite{Nash1958}. One can adapt the Carlen-Loss argument, which in turn is an adaptation of an earlier argument of Davies \cite{Davies1987}, to obtain a Gaussian upper bound from \eqref{eq:ldcay}. In certain cases (e.g. \cite{Osada1987dp, Maekawa2008}), this Gaussian upper bound then implies a corresponding lower bound, and it would seem that these results also hold under the assumption that $\div b\leq 0$. We hope to investigate this line of inquiry more in future work.
\end{remark}

\section{N-particle dynamics}
\label{sec:Npd}
In this section, we discuss the well-posedness of the SDE system \eqref{eq:SDE} for fixed $N\in\N$, in particular that with probability one, the particles never collide. We also discuss stability of the system under regularization of the potential. These regularizations will be important in the sequel when we attempt to apply It\^o's lemma to functions which are singular at the origin.

To prove the well-posedness of the system \eqref{eq:SDE}, we first consider the well-posedness of the corresponding system where the potential $\g$ has been smoothly truncated at a short distance $\vep>0$ from the origin but otherwise is the same. If the particles remain more than distance $\vep$ from one another, then they do not see the truncation and therefore the truncation plays no role. This observation will guide us throughout this subsection.

Let $\chi\in C_c^\infty(\R^d)$ be a nonnegative bump function satisfying
\begin{equation}
\chi(x) = \begin{cases} 1, & |x|\leq 1/2 \\ 0, & {|x|\geq 1}.\end{cases}
\end{equation}
Given $\vep>0$, define
\begin{equation}
\label{eq:defgvep}
\g_{(\vep)}(x) \coloneqq \g(x)(1-\chi(x/\vep)).
\end{equation}
The notation $\g_{(\vep)}$ should not be confused with the notation $\g_{\vep}$ in \eqref{eq:getadef} used later in \cref{sec:ME}. By assumption \ref{ass2}, $\g_{(\vep)} \in C^\infty$ with
\begin{equation}
\|\g_{(\vep)}\|_{L^\infty} \lesssim \begin{cases} -\log\vep, & {s=0} \\ \vep^{-s}, & {0<s<d-2} \end{cases}
\end{equation}
and $\|\nabla^{\otimes k}\g_{(\vep)}\|_{L^\infty}\lesssim \vep^{-(s+k)}$ for $k\geq 1$,
\begin{equation}
\label{eq:gvepeq}
\g_{(\vep)}(x) = \begin{cases} \g(x), & {|x|\geq \vep} \\ 0, & {|x|\leq \vep/2}.\end{cases}
\end{equation}
Define now the truncated version of the system \eqref{eq:SDE} by
\begin{equation}
\label{eq:SDEt}
\begin{cases}
dx_{i,\vep} = \displaystyle \frac{1}{N}\sum_{1\leq j\leq N : j\neq i} \M\nabla\g_{(\vep)}(x_{i,\vep}-x_{j,\vep})dt+\sqrt{2\sigma}dW_i \\
x_{i,\vep}|_{t=0} = x_{i}^0.
\end{cases}
\end{equation}
Since the vector field $\M\nabla\g_{(\vep)}$ is smooth with bounded derivatives of all order, global well-posedness of \eqref{eq:SDEt} is trivial. The equality \eqref{eq:gvepeq} implies that if
\begin{equation}
\inf_{0\leq t\leq T} \min_{1\leq i\neq j\leq N} |x_{i,\vep}^t-x_{j,\vep}^t| \geq \vep,
\end{equation}
then $x_{i,\vep} = x_i$ on $[0,T]$ for every $1\leq i \leq N$. In other words, the truncated dynamics coincide with the untruncated dynamics, just as remarked at the beginning of the subsection. Accordingly, we can define the stopping time
\begin{equation}\label{eq:taudef}
\tau_{\vep} \coloneqq \inf\{ 0\leq t\leq T: \min_{1\leq i\neq j\leq N} |x_{i,\vep}^t-x_{j,\vep}^t| \geq 2\vep\},
\end{equation}
so that on the random time interval $[0,\tau_{\vep}(\om)]$, $\ux_{N,\vep}(\om) \equiv \ux_N(\om)$, where $\om\in\Omega$ is a sample in the underlying probability space.

\begin{remark}
\label{rem:QV}
For later use, we observe (for instance, see \cite[Section 3.2.C]{KS1991}) that the quadratic variation of $x_{i,\vep}$ is the $d\times d$ matrix with components
\begin{align}
[x_{i,\vep}]^{t,\al\be} = 2\sigma t \d^{\al\be}, \qquad \al,\be\in\{1,\ldots,d\},
\end{align}
where $\d^{\al\be}$ is the Kronecker $\delta$-function. Similarly, for $i\neq j$, the quadratic variation of $x_{i,\vep}-x_{j,\vep}$ is given by
\begin{equation}
[x_{i,\vep}-x_{j,\vep}]^{t,\al\be} = [\sqrt{2\sigma}(W_i-W_j)^{\al}, \sqrt{2\sigma}(W_i-W_j)^{\be}]^t = 2\sigma[W_i^{\al},W_i^\be]^t + 2\sigma[W_j^\al,W_j^\be]^t = 4\sigma \d^{\al\be}t.
\end{equation}
\end{remark}

We first show that with probability one, the particles cannot escape to infinity by controlling the expectation of the \emph{moment of inertia}
\begin{equation}
\label{eq:midef}
I_{N,\vep} \coloneqq \sum_{i=1}^N |x_{i,\vep}|^2.
\end{equation}
\begin{lemma}
\label{lem:mi}
There exists a constant $C>0$ depending only on the dimension $d$, such that for all $T>0$,
\begin{equation}
\E\paren*{\sup_{0\leq t\leq T} I_{N,\vep}^t} \leq C\paren*{I_{N,\vep}^0 + \sigma(N+T)}e^{C \sigma T}.
\end{equation}
\end{lemma}
\begin{proof}
We sketch the proof. Applying It\^o's lemma with $f(x)=|x|^2$, we find that with probability one,
\begin{equation}
\begin{split}
\forall t\geq 0, \qquad |x_{i,\vep}^t|^2 - |x_{i}^0|^2 &= 2\int_0^t x_{i,\vep}^\ka \cdot\sum_{\substack{1\leq j\leq N \\ j\neq i}}\M\nabla\g(x_{i,\vep}^\ka-x_{j,\vep}^\ka)d\ka \\
&\ph +2\sqrt{2\sigma}\int_0^t  x_{i,\vep}^\ka\cdot dW_i^\ka + 2\sigma t.
\end{split}
\end{equation}
Since $\nabla\g$ is odd by assumption \ref{ass0}, it follows from the requirement \eqref{eq:Mnd} that
\begin{equation}
2\sum_{i=1}^N x_{i,\vep}\cdot\sum_{\substack{1\leq j\leq N \\ j\neq i}}\M\nabla\g(x_{i,\vep}-x_{j,\vep}) = \sum_{1\leq i\neq j\leq N} (x_{i,\vep}-x_{j,\vep})\cdot\M\nabla\g(x_{i,\vep}-x_{j,\vep}) \leq 0.
\end{equation}
By the Burkholder-Davis-Gundy inequality, denoting again $[\cdot]$ for the quadratic variation, we have
\begin{align}
\E\paren*{\sup_{0\leq t\leq T} \left|\int_0^t x_{i,\vep}^\ka\cdot dW_i^\ka\right|} &\lesssim \E\paren*{\sqrt{\left[\int_0^{(\cdot)} x_{i,\vep}^\ka\cdot dW_i^\ka\right]^T}} \lesssim \E\paren*{\sqrt{\int_0^T |x_{i,\vep}^\ka|^2d\ka}}.
\end{align}
 
We find after a little bookkeeping that
\begin{align}
\E\paren*{\sup_{0\leq t\leq T} {I}_{N,\vep}^t} &\lesssim I_{N,\vep}^0  + \sigma\paren*{T+ N^{1/2}\E\paren*{\left|\int_0^T I_{N,\vep}^\ka d\ka\right|^{1/2}}} \nn\\
&\leq I_{N,\vep}^0  + \sigma\paren*{T + N^{1/2} \E\paren*{\int_0^T I_{N,\vep}^\ka d\ka}^{1/2}} \nn\\
&\leq  I_{N,\vep}^0  + \sigma\paren*{T + N + \int_0^T\E\paren*{I_{N,\vep}^\ka}d\ka},
\end{align}
where the second line follows from Jensen's inequality and the third line from the elementary inequality $ab\leq\frac{a^2+b^2}{2}$ together with Fubini-Tonelli to interchange the expectation with the temporal integration. The desired conclusion now follows from the Gronwall-Bellman lemma.
\end{proof}

\begin{remark}
By Chebyshev's lemma, \cref{lem:mi} implies that with probability one,
\begin{equation}
\lim_{R\rightarrow\infty}\inf\{t\geq 0 : \max_{1\leq i\neq j\leq N} |x_{i,\vep}^t-x_{j,\vep}^t| \geq R\} = \infty.
\end{equation}
\end{remark}

Next, define the function
\begin{equation}
H_{N,\vep}(\ux_{N,\vep}) \coloneqq \sum_{1\leq i\neq j\leq N} \g_{(\vep)}(x_{i,\vep}-x_{j,\vep}),
\end{equation}
which has the interpretation of the energy of the system \eqref{eq:SDEt}.

\begin{lemma}
\label{lem:pbnd}
There exists a constant $C>0$ depending only on $s,d$, and the potential $\g$ through assumptions \ref{ass2}, \ref{ass2b}, and \ref{ass3a}, such that for all $0<\vep<\min\{\frac{1}{2},\frac{r_0}{2}\}$, where $r_0$ is as in \ref{ass1}, and $T>0$, it holds that
\begin{equation}
\Pb(\tau_\vep < T) \lesssim \begin{cases} \paren*{\min_{|x|\leq2\vep}\g(x)}^{-1}\E\paren*{\frac{C N^2\paren*{N+e^{C\sigma T}(\sigma(N+T)+I_{N,\vep}^0)}}{2} + H_{N,\vep}(\ux_{N}^0)}, & {s=0} \\
\paren*{\min_{|x|\leq2\vep}\g(x)}^{-1} \E(H_{N,\vep}(\ux_{N}^0)), &{0<s<d-2}.\end{cases}
\end{equation}
In particular, by assumption \ref{ass1a}, $\Pb(\tau_\vep < T)\rightarrow 0$ as $\vep\rightarrow 0^+$.
\end{lemma}
\begin{proof}
By It\^o's lemma applied to $\g_{(\vep)}(x_{i,\vep}-x_{j,\vep})$, it holds with probability one that
\begin{equation}
\label{eq:Hito}
\begin{split}
\forall t\geq 0, \qquad H_{N,\vep}(\ux_{N,\vep}^t) &= H_{N,\vep}(\ux_{N,\vep}^0) \\
&\ph + 2\sum_{1\leq i\neq j\leq N}\sum_{\substack{1\leq k\leq N \\ k\neq i}}\int_0^t \M\nabla\g_{(\vep)}(x_{i,\vep}^\ka-x_{k,\vep}^\ka) \cdot \nabla\g_{(\vep)}(x_{i,\vep}^\ka-x_{j,\vep}^\ka)d\ka \\
&\ph + \sqrt{2\sigma}\sum_{1\leq i\neq j\leq N}\int_0^t \nabla\g_{(\vep)}(x_{i,\vep}^\ka-x_{j,\vep}^\ka)\cdot d(W_i-W_j)^\ka \\
&\ph + \sigma\sum_{1\leq i\neq j\leq N} \int_0^t \nabla^{\otimes 2}\g_{(\vep)}(x_{i,\vep}^\ka-x_{j,\vep}^\ka) : \I \, d\ka .
\end{split}
\end{equation}
The second line is nonnegative by condition \eqref{eq:Mnd} and therefore may be discarded. For the third line, we note that the It\^o integral is a martingale with zero initial expectation. So by Doob's optional sampling theorem,
\begin{equation}
\E\paren*{\int_0^{\tau_{\vep}}\nabla\g_{(\vep)}(x_{i,\vep}^\ka-x_{j,\vep}^\ka)\cdot d(W_i-W_j)^\ka} = 0.
\end{equation}
For the fourth line, we note that
\begin{equation}
\nabla^{\otimes 2}\g_{(\vep)}(x_{i,\vep}^\ka-x_{j,\vep}^\ka) : \I  = \D\g_{(\vep)}(x_{i,\vep}^\ka-x_{j,\vep}^\ka).
\end{equation}
Since $\D\g\leq 0$ by assumption \ref{ass1} and $\g=\g_{(\vep)}$ outside the ball $B(x,\vep)$, it follows that
\begin{equation}
\D\g_{(\vep)}(x) \leq 0 \qquad \forall |x|\geq \vep.
\end{equation}
Consequently, it holds with probability one that
\begin{equation}
\int_0^{\tau_\vep}\nabla^{\otimes 2}\g_{(\vep)}(x_{i,\vep}^\ka-x_{j,\vep}^\ka) : \I \, d\ka \leq 0.
\end{equation}
Taking expectations of both sides of equation \eqref{eq:Hito}, we find that
\begin{equation}
\label{eq:hamub}
\E\paren*{H_{N,\vep}(\ux_{N,\vep}^{\tau_\vep})} \leq \E\paren*{H_{N,\vep}(\ux_{N}^0)}.
\end{equation}

We now want to use this bound to control the minimal distance between particles. To this end, we separately consider the cases $s=0$ and $0<s<d-2$. If $s=0$,  we use the moment of inertia to control the possible large negative values of $\g$. Using assumptions \ref{ass2}, \ref{ass2b}, and \ref{ass3a} (provided that $2\vep\leq r_0$), we see that for any $t\geq 0$,
\begin{align}
H_{N,\vep}(\ux_{N,\vep}^t) &\geq -C\sum_{\substack{1\leq i\neq j\leq N \\ |x_{i,\vep}^t-x_{j,\vep}^t|\geq 1}} \left|\log|x_{i,\vep}^t-x_{j,\vep}^t|\right| + \sum_{\substack{1\leq i\neq j\leq N \\ |x_{i,\vep}^t-x_{j,\vep}^t|\leq 2\vep}} \g_{(\vep)}(x_{i,\vep}^t-x_{j,\vep}^t) \nn\\
&\geq  C^{-1}\paren*{\min_{|x|\leq 2\vep}\g(x)}|\{(i,j)\in\{1,\ldots,N\}^2 : i\neq j \ \text{and} \ \vep\leq |x_{i,\vep}^t-x_{j,\vep}^t|\leq 2\vep\}| \nn\\
&\ph -C N^2\max\paren*{\log 2, \log(\sum_{i=1}^N |x_{i,\vep}^t|)} .
\end{align}
Note that if $\sum_{i=1}^N |x_{i,\vep}^t| \geq 1$, then by Cauchy-Schwarz and since $\log|x|\leq |x|$,
\begin{equation}
\log\paren*{\sum_{i=1}^N |x_{i,\vep}^t|} \leq N^{1/2}|I_{N,\vep}^t|^{1/2}\leq \frac{N+I_{N,\vep}^t}{2},
\end{equation}
where we recall that $I_{N,\vep}$ is the moment of inertia \eqref{eq:midef}. Modulo a null set, this lower bound implies that
\begin{align}
\{\tau_\vep < T\} &\subset \{\exists i\neq j\in \{1,\ldots,N\}^2 \ \text{such that} \ \vep\leq |x_{i,\vep}^{\tau_\vep}-x_{j,\vep}^{\tau_\vep}|\leq 2\vep\} \nn\\
&\subset \{H_{N,\vep}(\ux_{N,\vep}^{\tau_\vep}) \geq -\frac{C N^2(N+I_{N,\vep}^{\tau_\vep})}{2} + C^{-1}\min_{|x|\leq 2\vep}\g(x)\}.
\end{align}
So by Chebyshev's inequality and inequality \eqref{eq:hamub},
\begin{align}
\label{eq:probbnd}
\E\paren*{\frac{CN^2(N+I_{N,\vep}^{\tau_\vep})}{2}} + \E\paren*{H_{N,\vep}(\ux_{N}^0)} &\geq C^{-1} \paren*{\min_{|x|\leq2\vep}\g(x)} \Pb(\tau_\vep <T).
\end{align}
which in view of \cref{lem:mi}  concludes the proof if $s=0$.

If $0<s<d-2$, then it follows from assumptions \ref{ass2}, \ref{ass2b}, and \ref{ass3a} (provided that $2\vep\leq r_0$)that
\begin{align}
H_{N,\vep}(\ux_{N,\vep}) &\geq \sum_{\substack{1\leq i\neq j\leq N\\ |x_{i,\vep}-x_{j,\vep}|\leq 2\vep}} \g_{(\vep)}(x_{i,\vep}-x_{j,\vep}) \nn\\
&\geq -CN^2 + C^{-1}\paren*{\min_{|x|\leq 2\vep}\g(x)}|\{(i,j)\in\{1,\ldots,N\}^2 : i\neq j \ \text{and} \ \vep\leq |x_{i,\vep}-x_{j,\vep}|\leq 2\vep\}|,
\end{align}
which implies that
\begin{equation}
\label{eq:probbnds}
C^{-1}\paren*{\min_{|x|\leq 2\vep}\g(x)}\Pb(\tau_\vep < T)\leq \E(H_{N,\vep}(\ux_{N}^0)) + CN^2,
\end{equation}
completing the proof of the lemma.
\end{proof}

The next proposition, which is the main result of this section, shows that there is a unique solution to the Cauchy problem for \eqref{eq:SDE} in the strong sense.
\begin{prop}
\label{prop:Nwp}
With probability one,
\begin{equation}
\ux_N^t \coloneqq \lim_{\vep\rightarrow 0^+} \ux_{N,\vep}^t \ \text{exists} \ \forall t\geq 0,
\end{equation}
and we can unambiguously define $\ux_N$ as the unique strong solution to \eqref{eq:SDE}. Moreover,
\begin{equation}
\Pb\paren*{\forall t\geq 0, \quad \min_{1\leq i\neq j\leq N} |x_i^t-x_j^t|>0} = 1.
\end{equation}
\end{prop}
\begin{proof}
Note that if
\begin{equation}
2\vep\leq \min_{1\leq i\neq j\leq N} |x_{i}^0-x_{j}^0|,
\end{equation}
then $H_{N,\vep}(\ux_{N}^0) = H_N(\ux_{N}^0)$. Choose a sequence $\vep_k>0$ such that
\begin{equation}
\sum_{k=1}^\infty \paren*{\min_{|x|\leq 2\vep_k} \g(x)}^{-1} < \infty.
\end{equation}
Then by \cref{lem:pbnd},
\begin{equation}
\sum_{k=1}^\infty \Pb(\tau_{\vep_k}< T) <\infty,
\end{equation}
so by Borel-Cantelli,
\begin{equation}
\Pb\paren*{\limsup_{k\rightarrow\infty} \{\tau_{\vep_k}<T\}} = 0.
\end{equation}
Consequently, for almost every sample $\omega\in\Omega$, there exists $\vep(\omega)>0$ such that for all $0<\vep\leq \vep(\omega)$,
\begin{equation}
\begin{split}
\ux_{N,\vep}^t(\omega)=\ux_{N,\vep(\omega)}^t(\omega) \ \text{on} \ [0,T] \quad \text{and} \quad \inf_{0\leq t\leq T}\min_{1\leq i\neq j\leq N} |x_{i,\vep}^t(\omega)-x_{j,\vep}^t(\omega)| \geq 2\vep(\om).
\end{split}
\end{equation}
Since $T>0$ was arbitrary, we note that the preceding a.s. statement in fact holds globally in time.
\end{proof}

\section{Modulated energy and renormalized commutator estimates}\label{sec:ME}
In this section, we review the properties of the modulated energy $\Fr_N(\ux_N,\mu)$ established in the authors' joint work with Nguyen \cite{NRS2021} along with the associated renormalized commutator estimate proven in that work. These previous results will suffice to prove \cref{thm:lmain}. In the case of potentials which are globally superharmonic (i.e. $r_0=\infty$ in assumption \ref{ass1}), we also prove sharper versions (in terms of their $\|\mu\|_{L^\infty}$ dependence)  of the results of \cite{NRS2021} that are crucial to obtain global bounds of \cref{thm:gmain}. 

Throughout this section, we always assume that $\mu$ is a probability measure with density in $L^\infty(\R^d)$. If $0<s<d$, then it is immediate from \cref{lem:LinfRP} that $\g\ast\mu$ is a bounded, continuous function (it is actually $C^{k,\alpha}$ for some $k\in\N_0$ and $\alpha>0$ depending on the value of $s$) and therefore the modulated energy is well-defined. If $s=0$, then we need to impose a suitable decay assumption on $\mu$ to compensate for the logarithmic growth of $\g$ at infinity. For example,
\begin{equation}
\int_{\R^d}\log(1+|x|)d\mu(x) < \infty.
\end{equation}

\subsection{Review of results from \cite{NRS2021}}
We start by reviewing the results of \cite[Section 2]{NRS2021} on the properties of the modulated energy under the general assumptions on the potential contained in \ref{ass0} -- \ref{ass3'}. The statements presented below are specialized to the sub-Coulombic setting (i.e. $0\leq s \leq d-2$), and the relevant proofs, as well as further comments, may be found in \cite[Sections 2, 4]{NRS2021}.

With $r_0$ as in \ref{ass1} and $0<\eta<\min\{\frac{1}{2},\frac{r_0}{2}\}$, we let $\d_x^{(\eta)}$ denote the uniform probability measure on the sphere $\p B(x,\eta)$ and set
\begin{equation}\label{eq:getadef}
\g_\eta\coloneqq \g\ast\d_0^{(\eta)}.
\end{equation}
Since $\g$ is superharmonic in $B(0,r_0)$ by assumption \ref{ass1}, it follows from the formula ($\H^{d-1}$ is the $(d-1)$-dimensional Hausdorff measure in $\R^d$)
\begin{equation}
\frac{d}{dr}\dashint_{\p B(x,r)} f d\H^{d-1} = \frac{1}{d|B(0,1)|r^{d-1}}\int_{B(x,r)}\D f dy,
\end{equation}
which holds for any sufficiently integrable $f$, and an approximation argument that
\begin{equation}\label{eq:bgeta}
\g_\eta (x) \leq \g(x) \quad \forall x \in B(0, r_0-\eta) \setminus\{0\}
\end{equation}
and (using assumption \ref{ass2})
\begin{equation}\label{eq:bdiffgeta}
|\g(x)-\g_\eta(x)|  \le \frac{C \eta^2}{|x|^{s+2}}  \qquad \forall |x|\ge 2\eta,
\end{equation}
where the constant $C$ depends on $r_0$. By virtue of the mean value inequality \eqref{eq:bgeta} and assumption \ref{ass2}, the self-interaction of the smeared point mass $\d_{x_0}^{(\eta)}$ satisfies the relation
\begin{equation}\label{eq:selfinter}
\int_{(\R^d)^2}\g(x-y)d\d_{x_0}^{(\eta)}(x)d\d_{x_0}^{(\eta)}(y) = \int_{\R^d}\g_\eta d\delta_{0}^{(\eta)}  
\leq \int_{\R^d}\g d\delta_{0}^{(\eta)}  
 = \g_\eta(0) \leq C\paren*{\eta^{-s} + |\log\eta|\indic_{s=0}} .
\end{equation}

The next result we recall \cite[Proposition 2.1]{NRS2021} expresses the crucial monotonicity property of the modulated energy with respect to the smearing radii when expressed as the limit
\begin{equation}
\Fr_N(\ux_N,\mu) = \lim_{\al_i\rightarrow 0} \Bigg(\int_{(\R^d)^2} \g(x-y)d\paren*{\frac{1}{N}\sum_{i=1}^N \d_{x_i}^{(\al_i)}-\mu}^{\otimes 2}(x,y) - \frac{1}{N^2}\sum_{i=1}^N \int_{\R^d}\g_{\al_i}d\d_{0}^{(\al_i)}\Bigg).
\end{equation}
It also shows that the modulated energy is bounded from below, coercive, and controls the microscale interactions \cite[Corollary 2.3]{NRS2021}.

\begin{prop}\label{prop:MElb}
Let $d\geq 3$ and $0\leq s \leq d-2$. Suppose that $\ux_N\in ( \R^d)^N$ is pairwise distinct and $\mu \in \P(\R^d)\cap L^\infty(\R^d)$. In the case $s=0$, also suppose that $\int_{\R^d}\log(1+|x|)d\mu(x)<\infty$. There exists a constant $C$ depending only on $s,d$ and the potential $\g$ through assumptions $\mathrm{\ref{ass1}, \ref{ass2}, \ref{ass3}}$, such that for every choice of $0<\eta_1,\ldots,\eta_N<\min\{\frac{1}{2}, \frac{r_0}{2}\}$,
\begin{multline}\label{eq:propMElb}
\frac{1}{N^2}\sum_{\substack{1\leq i\neq j\leq N \\ |x_i-x_j|\leq\frac{r_0}{2}}} \paren*{\g(x_j-x_i)-\g_{\eta_i}(x_j-x_i)}_+ + C^{-1} \left\|\frac{1}{N}\sum_{i=1}^N\d_{x_i}^{(\eta_i)}-\mu\right\|_{\dot{H}^{\frac{s-d}{2}}}^2 \leq \Fr_N(\ux_N,\mu)\\
 + \frac{C}{N}\sum_{i=1}^{N}\Bigg(\paren*{\eta_i^2 + \frac{\eta_i^{-s}(1+|\log\eta_i|\indic_{s=0})}{N}}+ C\|\mu\|_{L^\infty}\eta_i^{d-s}\paren*{1+|\log\eta_i|(\indic_{s=0}+\indic_{s=d-2})}\Bigg).
\end{multline}
\end{prop}

\begin{remark}\label{rem:MEbal}
Since $0\leq s\leq d-2$ by assumption, we can balance the error terms in \eqref{eq:propMElb} by setting
\begin{equation}
\eta_i^2 = \frac{\eta_i^{-s}}{N} \Longleftrightarrow \eta_i = N^{-\frac{1}{s+2}},
\end{equation}
which, in particular, yields the lower bound
\begin{equation}
\Fr_N(\ux_N,\mu) \geq -C(1+\|\mu\|_{L^\infty})N^{-\frac{2}{s+2}}\paren*{1+(\log N)(\indic_{s=0} + \indic_{s=d-2})}.
\end{equation}
\end{remark}

\begin{remark}\label{rem:MElbft}
If instead of \ref{ass3}, we only assume that $\hat{\g}\geq 0$ on $\R^d\setminus\{0\}$, then the proof of \cite[Proposition 2.1]{NRS2021} yields the bound
\begin{multline}\label{eq:MElbft}
\frac{1}{N^2}\sum_{\substack{1\leq i\neq j\leq N \\ |x_i-x_j|\leq\frac{r_0}{2}}} \paren*{\g(x_j-x_i)-\g_{\eta_i}(x_j-x_i)}_+ \leq \Fr_N(\ux_N,\mu)  \\
+ \frac{C}{N}\sum_{i=1}^{N}\Bigg(\paren*{\eta_i^2 + \frac{\eta_i^{-s}(1+|\log\eta_i|\indic_{s=0})}{N}}+ C\|\mu\|_{L^\infty}\eta_i^{d-s}\paren*{1+|\log\eta_i|(\indic_{s=0}+\indic_{s=d-2})}\Bigg).
\end{multline}
\end{remark}

The next result we recall \cite[Proposition 2.2]{NRS2021} concerns the analogue of \cref{prop:MElb} in the case $d-2<s<d$. One of the key new insights from \cite{NRS2021} is that although superharmonicity may fail in the space $\R^d$, as it does for the Riesz potential $|x|^{-s}$, superharmonicity may be restored by considering the potential as the restriction of a potential $\G$ (i.e. $\g(x) = \G(x,0)$) in an extended space $\R^{d+m}$, where the size of $m$ depends on the value of $s$ so as to make $\G$ superharmonic in a neighborhood of the origin. Namely, suppose that $\g:\R^d \setminus\{0\}\rightarrow \R$ is such that there exists $\G:\R^{d+m}\setminus\{0\}\rightarrow\R$ with $\g(x)=\G(x,0)$ and satisfying conditions \eqref{extass0} -- \eqref{extass3}. With the notation $X=(x,z)\in\R^{d+m}$ and $X_i = (x_i,0)$, we let $\d_{X}^{(\eta)}$ denote the uniform probability measure on the sphere $\p B(X,r)\subset \R^{d+m}$ and set
\begin{equation}
\G_\eta \coloneqq \G\ast\d_0^{(\eta)}
\end{equation}
for $0<\eta<\min\{\frac{1}{2},\frac{r_0}{2}\}$. Analogously  to \eqref{eq:bgeta}, \eqref{eq:bdiffgeta}, \eqref{eq:selfinter}, we have
\begin{equation}\label{eq:bgetaext}
\G_\eta (X) \leq \G(X) \qquad \forall X \in B(0, r_0-\eta) \setminus\{0\},
\end{equation}
\begin{equation}\label{eq:bdiffgetaext}
|\G(X)-\G_\eta(X)|  \leq \frac{C \eta^2}{|X|^{s+2}}  \qquad \forall |X|\geq 2\eta,
\end{equation}
and
\begin{equation}\label{eq:selfinterext}
\int_{(\R^{d+m})^2}\G(x-y)d\d_{0}^{(\eta)}(x)d\d_{0}^{(\eta)}(y) \leq \G_\eta(0)  \leq C\paren*{\eta^{-s} + |\log\eta|\indic_{d=1 \wedge s=0}}.
\end{equation}
Again, the constant $C$ in \eqref{eq:bdiffgetaext} depends on $r_0$.

\begin{prop}\label{prop:extMElb}
Let $d\geq 3$ and $d-2<s<d$. Let $\g,\G$ be as above. Suppose that $\ux_N\in (\R^d)^N$ is a pairwise distinct configuration and $\mu\in\P(\R^d)\cap L^\infty(\R^d)$. There exists a constant $C>0$ only depending on $s,d$ and on $\g,\G$, such that for every $0<\eta_1,\ldots,\eta_N < \min\{\frac{1}{2}, \frac{r_0}{2}\}$, we have
\begin{multline}\label{eq:extMElb}
\frac{1}{N^2}\sum_{\substack{1\leq i\neq j\leq N\\ |x_i-x_j| \le \frac{r_0}{2}} }  
\paren*{\g(x_j-x_i)-\G_{\eta_i}(x_j-x_i,0)}_+ \leq \Fr_N(\ux_N,\mu) + \sum_{i=1}^N \frac{\eta_i^{-s}(1+|\log \eta_i|\indic_{s=0})}{N^2} \\
+\frac{C}{N}\sum_{i=1}^N\paren*{\|\mu\|_{L^\infty}\eta_i^{d-s} +\eta_i^2}.
\end{multline}
\end{prop}

\begin{remark}
Strictly speaking, the inequality \eqref{eq:extMElb} differs from \cite[(2.23), Proposition 2.2]{NRS2021} by the omission of  a positive term (a suitable squared norm  of  $\mu_N^t- \mu^t$). The reason we have omitted this term is because we no longer assume that $\hat{\G}(\Xi)\sim |\Xi|^{s-d-m}$. Instead, condition \eqref{extass3} only tells us that
\begin{equation}
\int_{(\R^{d+m})^2} \G(X-Y)d\paren*{\frac{1}{N}\sum_{i=1}^N\d_{X_i}^{(\eta_i)}-\tl\mu}(X,Y) \geq 0,
\end{equation}
which is good enough for the purposes of this article.
\end{remark}

As an application of \cref{prop:MElb} if $0\leq s\leq d-4$ and \cref{prop:extMElb} if $d-4<s<d-2$ using assumption \ref{ass4}, expressions like the second term appearing in the right-hand side of \eqref{eq:MEid}, which is due to the nonzero quadratic variation of the Brownian motion when we calculate the It\^o equation for the modulated energy $\Fr_N(\ux_N^t,\mu^t)$, are nonpositive up to a controllable error.

\begin{cor}\label{cor:MElb}
Let $d\geq 3$ and $0\leq s<d-2$. Suppose that $\ux_N \in (\R^d)^N$ is pairwise distinct and $\mu\in \P(\R^d)\cap L^\infty(\R^d)$. Let $\g$ be a potential satisfying assumptions $\mathrm{\ref{ass0}, \ref{ass1}, \ref{ass2}, \ref{ass3}, \ref{ass4}}$. There exists a constant $C>0$ depending only on $s,d$ and $\g$, such that
\begin{equation}
\int_{(\R^d)^2\setminus\triangle} (-\D\g)(x-y)d\paren*{\frac{1}{N}\sum_{i=1}^N\d_{x_i} - \mu}^{\otimes 2}(x,y) \geq - C(1+\|\mu\|_{L^\infty})N^{-\frac{\min\{2,d-s-2\}}{\min\{s+4,d\}}}.
\end{equation}
\end{cor}
\begin{proof}
If $0\leq s\leq d-4$, we use \eqref{eq:MElbft} applied with potential $-\D\g$ and with each $\eta_i = N^{-\frac{1}{s+4}}$. If $d-4<s<d-2$, we use \eqref{eq:extMElb} applied with extended potential $\G$ for $-\D\g$ given by assumption \ref{ass4} and with each $\eta_i = N^{-\frac{1}{d}}$.
\end{proof}

Finally, we close this section by recalling \cite[Proposition 4.1]{NRS2021} the renormalized commutator estimate from that work. As commented in the introduction, such  estimates are the main workhorse to close Gronwall arguments based on the modulated energy.

\begin{prop}
\label{prop:rcomm}
Let $d\geq 3$ and $0\leq s\leq d-2$. Let $\ux_N\in (\R^d)^N$ be a pairwise distinct configuration, and $\mu \in \P(\R^d)\cap L^\infty(\R^d)$.  If $s=0$, assume that $\int_{\R^d}\log(1+|x|)d\mu(x)<\infty$. Let $v$ be a continuous vector field on $\R^d$. There exists a constant $C$ depending only $d,s$ and on the potential $\g$ through assumptions \ref{ass0} -- \ref{ass3a}, \ref{ass3'}, such that
\begin{multline}\label{eq:rcomm}
\left| \int_{(\R^d)^2\setminus\triangle} (v(x)-v(y))\cdot\nabla\g(x-y)d\paren*{\frac{1}{N}\sum_{i=1}^N\d_{x_i}-\mu}^{\otimes 2}(x,y)\right| \\
\leq C\Bigg(\|\nabla v\|_{L^\infty}+\|\Dm^{\frac{d-s}{2}}v\|_{L^{\frac{2d}{d-2-s}}}\indic_{s<d-2}\Bigg)\Bigg(\Fr_N(\ux_N,\mu)  + C N^{-\frac{s+3}{(s+2)(s+1)}} \|\Dm^{s+1-d}\mu\|_{L^\infty}\\
+C(1+\|\mu\|_{L^\infty})N^{-\frac{2}{(s+2)(s+1)}}\Big(1+(\log N)(\indic_{s=0} + \indic_{s=d-2})\Big)\Bigg).
\end{multline}
\end{prop}

\subsection{New estimates for globally superharmonic potentials}\label{ssec:MEnew}
We now assume that $r_0=\infty$ in the condition \ref{ass1}, i.e. $\g$ is globally superharmonic. Under this more restrictive assumption, which holds in the model potential case \eqref{eq:gmod}, we can obtain versions of \cref{prop:MElb} and \cref{prop:extMElb} that have a better balance of factors of $\|\mu\|_{L^\infty}$ between terms. In particular, there is no $\eta^2$ error term like there is in the right-hand side of inequality \eqref{eq:propMElb}. This is important because if $\mu^t$ is time-dependent and satisfies the decay bound \eqref{eq:dcay}, this term will contribute linear growth in time when integrated. As we shall see in \cref{ssec:globGron}, this better balance will be crucial to show that the error terms (i.e. those which are not $\Fr_N(\ux_N^t,\mu^t)$) that result when estimating the right-hand side of \eqref{eq:MEid} are integrable in time over the interval $[0,\infty)$.

\begin{prop}\label{prop:impMElb}
Let $d\geq 3$ and $0\leq s\leq d-2$. Assume that $\g$ is a potential satisfying conditions \ref{ass0}, \ref{ass1}, \ref{ass2}, \ref{ass3} with $r_0=\infty$. Suppose that $\ux_N\in ( \R^d)^N$ is pairwise distinct and $\mu \in \P(\R^d)\cap L^\infty(\R^d)$. In the case $s=0$, also suppose that $\int_{\R^d}\log(1+|x|)d\mu(x)<\infty$. For any $\infty\geq  p>\frac{d}{s+2}$, there exist constants $C,C_{p}>0$ depending only on $s,d$ and the potential $\g$ through the assumed conditions, such that for every choice of $0<\eta_1,\ldots,\eta_N< 2^{-\frac{dp-d+2p}{d(p-1)}}\|\mu\|_{L^\infty}^{-\frac{1}{d}}$,
\begin{multline}\label{eq:imppropMElb}
\frac{1}{N^2}\sum_{{1\leq i\neq j\leq N}} \paren*{\g(x_j-x_i)-\g_{\eta_i}(x_j-x_i)}_+  + C^{-1}\left\|\frac{1}{N}\sum_{i=1}^N\d_{x_i}^{(\eta_i)}-\mu\right\|_{\dot{H}^{\frac{s-d}{2}}}^2 \leq \Fr_N(\ux_N,\mu) \\
+ \frac{C_{p}\|\mu\|_{L^\infty}^{\gamma_{s,p}} }{N}\sum_{i=1}^N \eta_i^{\la_{s,p}}\paren*{1+ \paren*{|\log \eta_i| + |\log \|\mu\|_{L^\infty}| }\indic_{s=0}} + \sum_{i=1}^N\frac{C\eta_i^{-s}(1+|\log\eta_i|\indic_{s=0})}{N^2},
\end{multline}
where the exponents $\gamma_{s,p},\la_{s,p}$ are defined by
\begin{equation}\label{eq:gamdef}
\gamma_{s,p} \coloneqq \frac{2p+sp-s}{dp+2p-d}, \qquad \la_{s,p} \coloneqq \frac{2p(d-s)}{dp+2p-d}.
\end{equation}
\end{prop}
\begin{proof}
We modify the proof of \cite[Proposition 2.1]{NRS2021}. Adding and subtracting $\d_{x_i}^{(\eta_i)}$ and regrouping terms yields the decomposition
\begin{equation}\label{eq:firstME}
\begin{split}
\Fr_N(\ux_N,\mu) &= \int_{(\R^d)^2}\g(x-y)d\paren*{\frac{1}{N}\sum_{i=1}^N\d_{x_i}^{(\eta_i)}-\mu}^{\otimes 2}(x,y)-\frac{1}{N^2} \sum_{i=1}^N \int_{\R^d} \g_{\eta_i} d\delta_0^{(\eta_i)}  \\
&\ph-\frac{2}{N}\sum_{i=1}^N\int_{\R^d} \paren*{\g(y-x_i)-\g_{\eta_i}(y-x_i)} d\mu(y) \\
&\ph+ \frac{1}{N^2}\sum_{1\leq i\neq j\leq N} \int_{\R^d} \paren*{\g(y-x_i)-\g_{\eta_i}(y-x_i)}d(\d_{x_j}+\d_{x_j}^{(\eta)})(y).
\end{split}
\end{equation}

Since $\D\g\leq 0$ on $\R^d$ by assumption \ref{ass1} and $\d_{x_j}^{(\eta_j)}$ is a positive measure, we have from \eqref{eq:bgeta} that
\begin{equation}\label{eq:impMEd}
\begin{split}
&\frac{1}{N^2}\sum_{1\leq i\neq j\leq N}\int_{\R^d} \paren*{\g(y-x_i)-\g_{\eta_i}(y-x_i)}d(\d_{x_j} + \d_{x_j}^{(\eta_j)})(y)\\
&\geq \frac{1}{N^2}\sum_{{1\leq i\neq j\leq N}} \paren*{\g(x_j-x_i)-\g_{\eta_i}(x_j-x_i)}_+.
\end{split}
\end{equation}

Let $R\geq 2\eta$ be a parameter to be specified shortly. Using assumption \ref{ass2} and the estimate \eqref{eq:bdiffgeta}, we find that
\begin{align}
&\left|\int_{\R^d} \paren*{\g(y-x_i)-\g_{\eta}(y-x_i)} d\mu(y)\right| \nn\\
&\leq \|\mu\|_{L^\infty}\int_{|y-x_i|\leq R} \paren*{|\g(y-x_i)|+|\g_{\eta}(y-x_i)|}dy + \int_{|y-x_i|>R} \left|\g(y-x_i)-\g_{\eta}(y-x_i)\right|d\mu(y) \nn\\
&\leq C\|\mu\|_{L^\infty}\paren*{ R^{d-s}(1+|\log R|\indic_{s=0})} + C\eta^2\int_{|y-x_i|>R} |y-x_i|^{-s-2}d\mu(y) \label{eq:impMEsplit}
\end{align}
By H\"older's inequality,
\begin{align}
\int_{|y-x_i|>R} |y-x_i|^{-s-2}d\mu(y) &\leq C(p(s+2)-d)^{-\frac{1}{p}} R^{\frac{d-p(s+2)}{p}}\|\mu\|_{L^{p'}} \nn\\
&\leq C(p(s+2)-d)^{-\frac{1}{p}} R^{\frac{d-p(s+2)}{p}}\|\mu\|_{L^\infty}^{\frac{1}{p}}
\end{align}
for any H\"older conjugate $p,p'$ with $\frac{d}{s+2}<p\leq \infty$. Implicitly, we have used that $\mu$ is a probability density, and the constant $C$ is independent of $p$. Setting
\begin{equation}\label{eq:impRch}
\|\mu\|_{L^\infty}R^{d-s} = \eta^2 R^{\frac{d-p(s+2)}{p}}\|\mu\|_{L^\infty}^{\frac{1}{p}},
\end{equation}
we get
\begin{equation}
R=\paren*{\eta^2 \|\mu\|_{L^\infty}^{\frac{1-p}{p}}}^{\frac{p}{dp-d+2p}}.
\end{equation}
In order for $R\geq 2\eta$, we need
\begin{equation}
\paren*{\eta^2 \|\mu\|_{L^\infty}^{\frac{1-p}{p}}}^{\frac{p}{dp-d+2p}} \geq 2\eta \Longleftrightarrow \eta\leq 2^{-\frac{dp-d+2p}{d(p-1)}}\|\mu\|_{L^\infty}^{-\frac{1}{d}}.
\end{equation}
Substituting in the above choice of $R$, we obtain that
\begin{multline}\label{eq:impMEmu}
\frac{1}{N}\sum_{i=1}^N\left|\int_{\R^d} \paren*{\g(y-x_i)-\g_{\eta_i}(y-x_i)} d\mu(y)\right| \leq \frac{C}{N}\|\mu\|_{L^\infty}^{\frac{2p-s+sp}{dp-d+2p}}\paren*{p(s+2)-d}^{-\frac{1}{dp+2p-d}}\\
\times\sum_{i=1}^N \eta_i^{\frac{2p(d-s)}{dp+2p-d}}\paren*{1+ \paren*{\frac{p}{dp+2p-d}\left|\log\eta_i^2\|\mu\|_{L^\infty}^{\frac{1-p}{p}}\right| }\indic_{s=0}}.
\end{multline}

Next, we use the relation \eqref{eq:selfinter} to bound
\begin{equation}\label{eq:impMEsi}
\frac{1}{N^2} \sum_{i=1}^N \left|\int_{\R^d} \g_{\eta_i} d\delta_0^{(\eta_i)}\right| \leq \sum_{i=1}^N\frac{C\eta_i^{-s}(1+|\log\eta_i|\indic_{s=0})}{N^2}.
\end{equation}
Collecting \eqref{eq:impMEd}, \eqref{eq:impMEmu}, \eqref{eq:impMEsi} and using the assumption \ref{ass3} with Plancherel's theorem for the remaining term in \eqref{eq:firstME}, we arrive at the inequality in the statement of the proposition.
\end{proof}

\begin{remark}
Evidently, the constant $C_p$ in \cref{prop:impMElb} blows up as $p\rightarrow {\frac{d}{s+2}}^{+}$.
\end{remark}

\begin{remark}\label{rem:impMEbal}
Dropping the $s,p$ subscripts in $\ga_{s,p},\la_{s,p}$, we balance the error terms by setting
\begin{equation}
C_{p}\|\mu\|_{L^\infty}^{\ga}\eta_i^{\la} = \frac{\eta_i^{-s}}{N} \Longleftrightarrow \eta_i = C_{p}^{-\frac{1}{\la+s}} \|\mu\|_{L^\infty}^{-\frac{\ga}{\la+s}} N^{-\frac{1}{\la+s}},
\end{equation}
which, for possibly larger constant $C_p>0$, implies the lower bound
\begin{equation}
\Fr_N(\ux_N,\mu) \geq -C_{p}\|\mu\|_{L^\infty}^{\frac{s}{d}}N^{-\frac{\la}{\la+s}}\paren*{1+\paren*{|\log\|\mu\|_{L^\infty}| + \log N}\indic_{s=0}}.
\end{equation}
\end{remark}

\begin{remark}\label{rem:impMElbft}
Just as in \cref{rem:MElbft}, if instead of \ref{ass3}, we only assume that $\hat{\g}\geq 0$ on $\R^d\setminus\{0\}$, then we have the bound
\begin{multline}\label{eq:impMElbft}
\frac{1}{N^2}\sum_{1\leq i\neq j\leq N}\paren*{\g(x_j-x_i)-\g_{\eta_i}(x_j-x_i)}_+ \leq \Fr_N(\ux_N,\mu) + \sum_{i=1}^N\frac{C\eta_i^{-s}(1+|\log\eta_i|\indic_{s=0})}{N^2} \\
+ \frac{C_{p}\|\mu\|_{L^\infty}^{\ga} }{N}\sum_{i=1}^N \eta_i^{\la}\paren*{1+ \paren*{|\log \eta_i| + |\log\|\mu\|_{L^\infty}| }\indic_{s=0}}.
\end{multline}
\end{remark}

Under the global superharmonicity assumption, we can also obtain a version of \cref{prop:extMElb} without an $\eta^2$ term and where every error term that is increasing in $\eta$ has a factor of $\|\mu\|_{L^\infty}$.

\begin{prop}\label{prop:extimpMElb}
Let $d\geq 3$ and $d-2<s<d$. Suppose that $\ux_N\in (\R^d)^N$ is a pairwise distinct configuration and $\mu\in\P(\R^d)\cap L^\infty(\R^d)$. Let $\g$ be a potential satisfying \ref{ass0}, \ref{ass1}, \ref{ass2} with $r_0=\infty$. Let $\G:\R^{d+m}\setminus\{0\}\rightarrow\R$ be an extension $\G(x,0) = \g(x)$ such that $\G$ satisfies conditions \eqref{extass0}, \eqref{extass1}, \eqref{extass2}, \eqref{extass3} with $r_0=\infty$. There exists a constant $C>0$ only depending on $s,d$ and on $\g,\G$, such that for every $\eta_1,\ldots,\eta_N > 0$, we have
\begin{equation}\label{eq:extimpMElb}
\frac{1}{N^2}\sum_{1\leq i\neq j\leq N}\paren*{\g(x_j-x_i)-\G_{\eta_i}(x_j-x_i,0)}_+\leq \Fr_N(\ux_N,\mu) + \frac{C}{N}\sum_{i=1}^N\paren*{\frac{\eta_i^{-s}}{N} + \|\mu\|_{L^\infty}\eta_i^{d-s}}.
\end{equation}
\end{prop}
\begin{proof}
We modify the proof of \cite[Proposition 2.2]{NRS2021} in the same spirit as we did for \cref{prop:impMElb}. Adding and subtracting $\d_{X_i}^{(\eta_i)}$ and regrouping terms, we find that
\begin{equation}\label{eq:firstext}
\begin{split}
\Fr_N(\ux_N,\mu) &=\int_{(\R^{d+m})^2} \G(X-Y)d\paren*{\frac{1}{N}\sum_{i=1}^N\d_{X_i}^{(\eta_i)}-\tl{\mu}}^{\otimes 2}(x,y) \\
&\ph-\frac{1}{N^2}\sum_{i=1}^N \int_{(\R^{d+m})^2}\G(X-Y)d(\d_{0}^{(\eta_i)})^{\otimes 2}(X,Y) \\
&\ph -\frac{2}{N}\sum_{i=1}^N \int_{\R^{d+m}}\paren*{\G(Y-X_i)-\G_{\eta_i}(Y-X_i)}d\tl{\mu}(Y) \\
&\ph + \frac{1}{N^2}\sum_{1\leq i\neq j\leq N} \int_{\R^d} \paren*{\G(Y-X_i)-\G_{\eta_i}(Y-X_i)}d({\d}_{X_j} + {\d}_{X_j}^{(\eta_j)})(Y).
\end{split}
\end{equation}

By the inequality \eqref{eq:bgetaext} and since $\d_{X_i}^{(\eta_i)}$ is a positive measure in $\R^{d+m}$, we have the lower bound
\begin{equation}\label{eq:extimpMEd}
\begin{split}
&\frac{1}{N^2}\sum_{1\leq i\neq j\leq N} \int_{\R^d} \paren*{\G(Y-X_i)-\G_{\eta_i}(Y-X_i)}d({\d}_{X_j} + {\d}_{X_j}^{(\eta_j)})(Y)\\
&\geq \frac{1}{N^2}\sum_{1\leq i\neq j\leq N}\paren*{\g(x_j-x_i)-\G_\eta(x_j-x_i,0)}_+.
\end{split}
\end{equation}

Letting $R\geq 2\eta$ be a parameter to be determined, we have that
\begin{align}
\int_{\R^{d+m}}\paren*{\G(Y-X_i)-\G_{\eta_i}(Y-X_i)}d\tl{\mu}(Y) = \int_{\R^{d}}\paren*{\g(y-x_i) - \G_{\eta_i}(y-x_i,0)}d\mu(y)
\end{align}
by definition of $\tl{\mu}$. Since $s>d-2$ by assumption, we can use \eqref{eq:bdiffgetaext} and \eqref{extass2} to obtain the bound
\begin{equation}
\int_{|y-x_i|\geq R}|\g(y-x_i) - \G_{\eta_i}((y-x_i,0))|d\mu(y) \leq C\eta_i^2 R^{d-s-2}\|\mu\|_{L^\infty}.
\end{equation}
Using \ref{extass2} to estimate directly the integral over $|y-x_i|<R$ as in \eqref{eq:impMEsplit}, it follows that
\begin{equation}
\left|\int_{\R^{d+m}}\paren*{\g(y-x_i) - \G_{\eta_i}(y-x_i,0)}d\mu(y)\right| \leq C\|\mu\|_{L^\infty}\paren*{ R^{d-s} + \eta_i^2 R^{d-s-2}}.
\end{equation}
We can them optimize the choice of $R$ by setting $R=2\eta_i$. After a little bookkeeping, we have shown that
\begin{equation}\label{eq:extimpMEmu}
\frac{1}{N}\sum_{i=1}^N \left|\int_{\R^{d+m}}\paren*{\g(y-x_i) - \G_{\eta_i}(y-x_i,0)}d\mu(y)\right| \leq C\eta_i^{d-s}\|\mu\|_{L^\infty}.
\end{equation}

Finally, using the relation \eqref{eq:selfinterext}, we have the self-interaction bound
\begin{equation}
\frac{1}{N^2}\sum_{i=1}^N \left|\int_{(\R^{d+m})^2}\G(X-Y)d(\d_{0}^{(\eta_i)})^{\otimes 2}(X,Y)\right| \leq \frac{C\eta^{-s}}{N}
\end{equation}
and using assumption \ref{extass3} with Plancherel's theorem, we have the lower bound
\begin{equation}
\int_{(\R^{d+m})^2}\G(X-Y)d\paren*{\frac{1}{N}\sum_{i=1}^N\d_{x_i}^{(\eta_i)}-\tl{\mu}}(X,Y) \geq 0.
\end{equation}
Combining these observations with \eqref{eq:extimpMEd}, \eqref{eq:extimpMEmu}, we arrive at the inequality in the statement of the proposition.
\end{proof}

\begin{remark}\label{rem:extimpMEbal}
Similar to \cref{rem:impMEbal}, we can balance the error terms in \eqref{eq:extimpMElb} by choosing $\eta_i = (\|\mu\|_{L^\infty} N)^{-\frac{1}{d}}$, which implies the lower bound
\begin{equation}
\Fr_N(\ux_N,\mu) \geq -C\|\mu\|_{L^\infty}^{\frac{s}{d}}N^{-\frac{d-s}{d}}.
\end{equation}
\end{remark}

Analogous to \cref{cor:MElb}, we can use assumption \ref{ass4} for admissible potentials $\g$ together with \cref{prop:impMElb} (if $0\leq s\leq d-4$) and \cref{prop:extimpMElb} (if $d-4<s<d-2$) to obtain the following result.

\begin{cor}\label{cor:impMElb}
Let $d\geq 3$ and $0\leq s<d-2$. Suppose that $\ux_N \in (\R^d)^N$ is pairwise distinct and $\mu\in \P(\R^d)\cap L^\infty(\R^d)$. Let $\g$ be a potential satisfying assumptions \ref{ass0}, \ref{ass1}, \ref{ass2}, \ref{ass3}, \ref{ass4} with $r_0=\infty$. There exists a constant $C>0$ depending only on $s,d$ and $\g,\G$, such that the following holds. If $0\leq s\leq d-4$, then for any $\infty\geq p>\frac{d}{s+4}$, $C$ also depends on $p$ and
\begin{equation}
\int_{(\R^d)^2\setminus\triangle} (-\D\g)(x-y)d\paren*{\frac{1}{N}\sum_{i=1}^N\d_{x_i} - \mu}^{\otimes 2}(x,y) \geq -C_{p}\|\mu\|_{L^\infty}^{\frac{s+2}{d}}N^{-\frac{\la_{s+2,p}}{\la_{s+2,p}+s+2}}.
\end{equation}
where $\la_{s+2,p}$ is as defined in \eqref{eq:gamdef}. If $d-4<s<d-2$, then
\begin{equation}
\int_{(\R^d)^2\setminus\triangle} (-\D\g)(x-y)d\paren*{\frac{1}{N}\sum_{i=1}^N\d_{x_i} - \mu}^{\otimes 2}(x,y) \geq -C\|\mu\|_{L^\infty}^{\frac{s+2}{d}} N^{-\frac{d-s-2}{d}}.
\end{equation}
\end{cor}
\begin{proof}
If $0\leq s\leq d-4$, then we use \eqref{eq:impMElbft} with $s$ replaced by $s+2$ and choosing \\ $\eta_i = \|\mu\|_{L^\infty}^{-\frac{\ga_{s+2,p}}{\la_{s+2,p}+s+2}} N^{-\frac{1}{\la_{s+2,p}+s+2}}$. If $d-4<s<d-2$, then we use \eqref{eq:extimpMElb} with $s$ replaced by $s+2$ and choosing $\eta_i=(\|\mu\|_{L^\infty}N)^{-\frac{1}{d}}$.
\end{proof}

Repeating the proof of \cite[Proposition 4.1]{NRS2021}, except now using \cref{prop:impMElb} instead of \cref{prop:MElb}, we can obtain a renormalized commutator estimate (cf. \cref{prop:rcomm}) with better distribution of norms of $\mu$.

\begin{prop}\label{prop:imprcomm}
Let $d\geq 3$ and $0\leq s\leq d-2$. Let $\ux_N\in (\R^d)^N$ be pairwise distinct, and $\mu \in \P(\R^d)\cap L^\infty(\R^d)$.  If $s=0$, assume that $\int_{\R^d}\log(1+|x|)d\mu(x)<\infty$. Let $v$ be a continuous vector field on $\R^d$. For every $\infty\geq p>\frac{d}{s+2}$, there exists a constant $C_p$ depending only $d,s$ and on the potential $\g$ through assumptions \ref{ass0} -- \ref{ass3a}, \ref{ass3'} such that for all $N > (2^{\frac{dp-d+p}{(p-1)}}\|\mu\|_{L^\infty})^{\frac{(s+1)(s+\la_{s,p})}{d}}$,
\begin{multline}\label{eq:imprcomm}
\left| \int_{(\R^d)^2\setminus\triangle} (v(x)-v(y))\cdot\nabla\g(x-y)d\paren*{\frac{1}{N}\sum_{i=1}^N\d_{x_i}-\mu}^{\otimes 2}(x,y)\right| \\
\leq  C\|\nabla v\|_{L^\infty}\|\Dm^{s+1-d}\mu\|_{L^\infty}N^{-\frac{s+1+\la_{s,p}}{(s+\la_{s,p})(1+s)}} + C\Bigg(\|\nabla v\|_{L^\infty}+\|\Dm^{\frac{d-s}{2}}v\|_{L^{\frac{2d}{d-2-s}}}\indic_{s<d-2}\Bigg)\Bigg(\Fr_N(\ux_N,\mu) \\
+ C_p \paren*{1+\|\mu\|_{L^\infty}^{\gamma_{s,p}}} N^{-\frac{\la_{s,p}}{(s+\la_{s,p})(1+s)}}\paren*{1+ \paren*{\log N  + |\log\|\mu\|_{L^\infty}|}\indic_{s=0}}\Bigg),
\end{multline}
where $\ga_{s,p},\la_{s,p}$ are as in \eqref{eq:gamdef}.
\end{prop}
\begin{proof}
Since $s,p$ are fixed, we drop the subscripts in $\ga_{s,p},\la_{s,p}$. Repeating the steps in the proof of \cite[Proposition 4.1]{NRS2021}, we find that the left-hand side of \eqref{eq:imprcomm} is controlled by
\begin{multline}\label{eq:rcommsub}
C\Bigg(\|\nabla v\|_{L^\infty}+\|\Dm^{\frac{d-s}{2}}v\|_{L^{\frac{2d}{d-2-s}}}\indic_{s<d-2}\Bigg)\Bigg(\Fr_N(\ux_N,\mu) \\
+ C_{p}\|\mu\|_{L^\infty}^{\ga} \eta^{\la}\paren*{1+ \paren*{|\log\eta| + |\log\|\mu\|_{L^\infty}| }\indic_{s=0}} + \frac{C\eta^{-s}(1+|\log\eta|\indic_{s=0})}{N}\Bigg)\\
+ C\|\nabla v\|_{L^\infty}\Bigg(\frac{\vep^{-s}(1+|\log\vep|\indic_{s=0})}{N} + C_{p}\|\mu\|_{L^\infty}^{\ga} \vep^{\la}\paren*{1+ \paren*{|\log\vep| + |\log\|\mu\|_{L^\infty}| }\indic_{s=0}} \\
+\eta \|\Dm^{s+1-d}\mu\|_{L^\infty} + \frac{\eta}{\vep^{s+1}}\Bigg),
\end{multline}
where $0<2\eta\leq \vep<2^{-\frac{dp-d+2p}{d(p-1)}}\|\mu\|_{L^\infty}^{-\frac{1}{d}}$ and $\infty\geq p>\frac{d}{s+2}$. Since $0\leq s\leq d-2$ by assumption, we balance error terms (i.e. those terms which are not $\Fr_N(\ux_N,\mu)$) by setting
\begin{equation}
\frac{\eta}{\vep^{s+1}} = \frac{\eta^{-s}}{N} = \vep^{\la},
\end{equation}
which yields $\eta = \vep^{\la+s+1}$ and $\vep = N^{-\frac{1}{(1+s)(s+\la)}}$. To ensure that $\vep < 2^{-\frac{dp-d+2p}{d(p-1)}}\|\mu\|_{L^\infty}^{-\frac{1}{d}}$, we require that
\begin{equation}
N^{-\frac{1}{(1+s)(s+\la)}} < 2^{-\frac{dp-d+2p}{d(p-1)}}\|\mu\|_{L^\infty}^{-\frac{1}{d}} \Longleftrightarrow 2^{\frac{(dp-d+2p)(1+s)(s+\la)}{d(p-1)}}\|\mu\|_{L^\infty}^{\frac{(1+s)(s+\la)}{d}} < N.
\end{equation}
Substituting these choices back into \eqref{eq:rcommsub}, we arrive at the inequality in the statement of the proposition.
\end{proof}

\section{Evolution of the modulated energy}\label{sec:ev}
Our next task is to rigorously compute the time-derivative of the modulated energy, which we recall is a real-valued stochastic process. Since the potential $\g$ is not $C^2$ due to its singularity at the origin, we cannot directly apply It\^o's lemma to $\Fr_{N}(\ux_N^t,\mu^t)$, as we formally did in the introduction to obtain \eqref{eq:MEid}. Instead, we proceed by a truncation and stopping time argument, similar to that used \cref{sec:Npd} to prove the well-posedness of the $N$-body dynamics.

We define the \emph{truncated modulated energy}
\begin{equation}
\Fr_{N,\vep}(\ux_{N,\vep}^t,\mu^t) \coloneqq \int_{(\R^d)^2\setminus\triangle}\g_{(\vep)}(x-y)d(\mu_{N,\vep}^t-\mu^t)^{\otimes 2}(x,y),
\end{equation}
where $\g_{(\vep)}$ is as defined in \eqref{eq:defgvep}, $\ux_{N,\vep}$ is the solution to the truncated system \eqref{eq:SDEt}, and $\mu_{N,\vep}$ denotes the empirical measure induced by $\ux_{N,\vep}$. Comparing  this expression to the definition of $\Fr_N(\ux_N,\mu)$ above, we have just replaced the potential $\g$ with the potential $\g_{(\vep)}$ and replaced $\ux_N$ with $\ux_{N,\vep}$. Thanks to the regularity of the truncated potential $\g_{(\vep)}$, we can \emph{rigorously} apply It\^o's lemma to $\Fr_{N,\vep}(\ux_{N,\vep}^t,\mu^t)$.

\begin{lemma}
\label{lem:ItoMEreg}
Let $\ux_{N,\vep}$ be a solution to the system \eqref{eq:SDEt}, and let $\mu \in C([0,\infty); L^1(\R^d)\cap L^\infty(\R^d))$ be a solution to equation \eqref{eq:lim}. Then for every choice of $\vep$ satisfying
\begin{equation}\label{eq:epcon}
0<\vep \leq \frac{1}{2}\min_{1\leq i\neq j\leq N} |x_{i}^0-x_{j}^0|,
\end{equation}
it holds with probability one that for all $t\geq 0$,
\begin{multline}
\Fr_{N,\vep}(\ux_{N,\vep}^t,\mu^t) = \Fr_{N,\vep}(\ux_N^0,\mu^0) +\frac{2}{N^3}\sum_{\substack{1\leq i,j,k\leq N \\ k,j\neq i}}\int_0^t \nabla\g_{(\vep)}(x_{i,\vep}^\ka-x_{j,\vep}^\ka)\cdot\M\nabla\g_{(\vep)}(x_{i,\vep}^\ka-x_{k,\vep}^\ka)d\ka \\
+ \frac{2}{N}\sum_{i=1}^N\int_0^t (\g_{(\vep)}\ast\div(u^\ka\mu^\ka))(x_{i,\vep}^\ka)d\ka +\frac{2}{N^2}\sum_{1\leq i\neq j\leq N} \int_0^t (\nabla\g_{(\vep)}\ast\mu^\ka)(x_{i,\vep}^\ka)\cdot\M\nabla\g_{(\vep)}(x_{i,\vep}^\ka-x_{j,\vep}^\ka)d\ka\\
-2\int_0^t \ipp{\g_{(\vep)}\ast\div(u^\ka\mu^\ka),\mu^\ka}_{L^2}d\ka +2\sigma\int_0^t \int_{(\R^2)^2\setminus\triangle}\D\g_{(\vep)}(x-y)d(\mu_{N,\vep}^\ka-\mu^\ka)^{\otimes 2}(x,y)d\ka\\
+ \frac{2\sqrt{2\sigma}}{N}\sum_{i=1}^N \int_0^t\int_{\R^d\setminus\{x_{i,\vep}^\ka\}} \nabla\g_{(\vep)}(x_{i,\vep}^\ka-y)d(\mu_{N,\vep}^\ka-\mu^\ka)(y)\cdot dW_i^\ka,
\end{multline}
where $\mu_{N,\vep} \coloneqq \frac{1}{N}\sum_{i=1}^N \d_{x_{i,\ep}}$ and $u\coloneqq \M\nabla\g\ast\mu$.
\end{lemma}
\begin{proof}
By approximation, we may assume without loss of generality that $\mu$ is smooth and rapidly decaying at infinity. We split the modulated energy into a sum of three terms, defined by
\begin{align}
\Te_1 &\coloneqq \frac{1}{N^2}\sum_{1\leq i\neq j\leq N}\g_{(\vep)}(x_{i,\vep}-x_{j,\vep}),\\
\Te_2 &\coloneqq -\frac{2}{N}\sum_{i=1}^N \int_{\R^2}\g_{(\vep)}(x_{i,\vep}-y)d\mu(y) = -\frac{2}{N}\sum_{i=1}^N \g_{(\vep)}\ast\mu(x_{i,\vep}),\\
\Te_3 &\coloneqq \int_{(\R^2)^2}\g_{(\vep)}(x-y)d\mu^{\otimes 2}(x,y) = \ipp{\g_{(\vep)}\ast\mu,\mu}_{L^2},
\end{align}
and compute the stochastic/deterministic differential equation satisfied by each of these terms. Of course, $\Te_1,\ldots,\Te_3$ depend on $\vep$, but since $\vep$ is fixed, we omit this dependence.

\begin{description}
\item[$\Te_1$]
By It\^o's lemma and \cref{rem:QV}, we have, for $i\neq j$, that $\g_{(\vep)}(x_{i,\vep}-x_{j,\vep})$ satisfies the SDE
\begin{align}
d\g_{(\vep)}(x_{i,\vep}-x_{j,\vep}) &= \nabla\g_{(\vep)}(x_{i,\vep}-x_{j,\vep})\cdot d(x_{i,\vep}-x_{j,\vep}) + \frac{1}{2}\nabla^{\otimes 2}\g_{(\vep)} (x_{i,\vep}-x_{j,\vep}) : d[x_{i,\vep}-x_{j,\vep}] \nn\\
&=\nabla\g_{(\vep)}(x_{i,\vep}-x_{j,\vep})\cdot\paren*{\frac{1}{N}\sum_{\substack{1\leq k\leq N \\ k\neq i}} \M\nabla\g_{(\vep)}(x_{i,\vep}-x_{k,\vep}) - \frac{1}{N}\sum_{\substack{1\leq k\leq N \\ k\neq j}} \M\nabla\g_{(\vep)}(x_{j,\vep}-x_{k,\vep})}dt \nn\\
&\ph + \sqrt{2\sigma}\nabla\g_{(\vep)}(x_{i,\vep}-x_{j,\vep})\cdot d(W_i-W_j) + 2\sigma(\nabla^{\otimes 2}\g_{(\vep)}(x_{i,\vep}-x_{j,\vep}) : \I)dt.
\end{align}
Evidently,
\begin{equation}
2\sigma\nabla^{\otimes 2}\g_{(\vep)}(x_{i,\vep}-x_{j,\vep}) : \I = 2\sigma \D\g_{(\vep)}(x_{i,\vep}-x_{j,\vep})
\end{equation}
Thus by symmetry under swapping $i\leftrightarrow j$, and after integrating in time, we obtain 
\begin{multline}\label{eq:T1fin}
\Te_1(t) =\frac{1}{N^2}\sum_{1\leq i\neq j\leq N} \g(x_{i}^0-x_{j}^0) +\frac{2\sqrt{2\sigma}}{N^2}\sum_{1\leq i\neq j\leq N}\int_0^t\nabla\g_{(\vep)}(x_{i,\vep}^\ka-x_{j,\vep}^\ka)\cdot dW_i^\ka \\
+\frac{2}{N^3}\sum_{1\leq i\neq j\leq N}\sum_{\substack{1\leq k\leq N \\ k\neq i}}  \int_0^t \nabla\g_{(\vep)}(x_{i,\vep}^\ka-x_{j,\vep}^\ka)\cdot \M\nabla\g_{(\vep)}(x_{i,\vep}^\ka-x_{k,\vep}^\ka)d\ka \\
+2\sigma\sum_{1\leq i\neq j\leq N} \int_0^t \D\g_{(\vep)}(x_{i,\vep}^\ka-x_{j,\vep}^\ka)d\ka,
\end{multline}
provided that $\vep \leq \frac{1}{2}\min_{i\neq j}|x_{i,0}-x_{j,0}|$.
\item[$\Te_2$]
Defining $f(t,x)\coloneqq (\g_{(\vep)}\ast\mu^t)(x)$, we first observe from equation \eqref{eq:lim} that
\begin{equation}
\label{eq:f}
\p_t f = \g_{(\vep)}\ast\paren*{-\div(\mu u) + \sigma\D\mu)}, 
\end{equation}
Applying It\^o's lemma with the time-dependent function $f$, we find that
\begin{align}
df(t,x_{i,\vep}) &= \p_t f(t,x_{i,\vep})dt + \nabla f(t,x_{i,\vep}) \cdot dx_{i,\vep} + \frac{1}{2}\nabla^{\otimes 2}f(t,x_{i,\vep}) : d[x_{i,\vep}] \nn\\
&=-\g_{(\vep)}\ast \div(u^t\mu^t)(x_{i,\vep})dt + \sigma(\g_{(\vep)}\ast\D\mu^t)(x_{i,\vep})dt \nn\\
&\ph + \nabla(\g_{(\vep)}\ast\mu^t)(x_{i,\vep})\cdot \frac{1}{N}\sum_{\substack{ 1\leq k\leq N \\ k\neq i}}\M\nabla\g_{(\vep)}(x_{i,\vep}-x_{k,\vep}) dt\nn\\
&\ph + \sqrt{2\sigma}\nabla(\g_{(\vep)}\ast\mu^t)(x_{i,\vep})\cdot dW_i + \sigma(\nabla^{\otimes 2}(\g_{(\vep)}\ast\mu^t)(x_{i,\vep}) : \I) dt ,
\end{align}
where we also use \cref{rem:QV} to obtain the ultimate line. Noting that
\begin{equation}
\sigma\nabla^{\otimes 2}(\g_{(\vep)}\ast\mu^t)(x_{i,\vep}) : \I = \sigma\D(\g_{(\vep)}\ast\mu^t)(x_{i,\vep}),
\end{equation}
we conclude that
\begin{multline}\label{eq:T2fin}
\Te_2(t) = -\frac{2}{N}\g_{(\vep)}\ast\mu^0(x_{i}^0) +\frac{2}{N}\sum_{i=1}^N \int_0^t \g_{(\vep)}\ast\div(u^\ka\mu^\ka)(x_{i,\vep}^\ka)d\ka -\frac{2\sigma}{N}\sum_{i=1}^N \int_0^t (\g_{(\vep)}\ast\D\mu^\ka)(x_{i,\vep}^\ka)d\ka\\
-\frac{2}{N^2}\sum_{1\leq i\neq k\leq N} \int_0^t \nabla(\g_{(\vep)}\ast\mu^\ka)(x_{i,\vep}^\ka)\cdot\M\nabla\g_{(\vep)}(x_{i,\vep}^\ka-x_{k,\vep}^\ka)d\ka - \frac{2\sqrt{2\sigma}}{N}\sum_{i=1}^N \int_0^t \nabla(\g_{(\vep)}\ast\mu^\ka)(x_{i,\vep}^\ka)\cdot dW_i^\ka.
\end{multline}

\item[$\Te_3$]
Using equation \eqref{eq:f} and symmetry under swapping $x\leftrightarrow y$, we find that
\begin{equation}
\label{eq:T3fin}
\begin{split}
\Te_3(t) &= \ipp{\g_{\vep}\ast\mu^0,\mu^0} +2\int_0^t \ipp{\g_{(\vep)}\ast(-\div(u^\ka\mu^\ka) + \sigma\D\mu^\ka), \mu^\ka}_{L^2}d\ka.
\end{split}
\end{equation}
\end{description}

Combining the identities \eqref{eq:T1fin}, \eqref{eq:T2fin} and \eqref{eq:T3fin} completes the proof of the lemma.
\end{proof}

We now proceed to group terms following the proof of \cite[Lemma 2.1]{Serfaty2020}. We leave filling in the details to the reader, as it requires nothing new from the aforementioned work.
\begin{lemma}
\label{lem:MEgroup}
Let $\ux_{N,\vep}$ and $\mu$ be as in \cref{lem:ItoMEreg}. Then for every $\vep$ satisfying \eqref{eq:epcon}, it holds with probability 1 that for all $t\geq 0$,
\begin{multline}\label{eq:group}
\Fr_{N,\vep}(\ux_{N,\vep}^t,\mu^t) -  \Fr_{N,\vep}(\ux_{N}^0,\mu^0) \leq \int_0^t \int_{(\R^d)^2\setminus\triangle}\paren*{u_\vep^\ka(x)-u_\vep^\ka(y)}\cdot\nabla\g_{(\vep)}(x-y)d(\mu_{N,\vep}^\ka-\mu^\ka)^{\otimes 2}(x,y)\\
+ 2\sigma\int_0^t \int_{(\R^d)^2\setminus\triangle}\D\g_{(\vep)}(x-y)d(\mu_{N,\vep}^\ka-\mu^\ka)^{\otimes 2}(x,y)d\ka + \frac{2}{N}\sum_{i=1}^N\int_0^t (\g_{(\vep)}\ast\div((u^\ka-u_\vep^\ka)\mu^\ka))(x_{i,\vep}^\ka)d\ka\\
-2\int_0^t \ipp{\g_{(\vep)}\ast\div((u^\ka-u_\vep^\ka)\mu^\ka),\mu^\ka}_{L^2}d\ka  + \frac{2\sqrt{2\sigma}}{N}\sum_{i=1}^N \int_0^t \PV\int_{\R^d\setminus\{x_{i,\vep}^\ka\}}\nabla\g_{(\vep)}(x_{i,\vep}^\ka-y)d(\mu_{N,\vep}^\ka-\mu^\ka)(y)\cdot dW_i^\ka,
\end{multline}
where $u_\vep \coloneqq \M\nabla\g_{(\vep)}\ast\mu$.
\end{lemma}

We are now prepared to remove the truncation by passing to the limit $\vep\rightarrow 0^+$. To this end, we recall from \cref{sec:Npd} the stopping time $\tau_\vep$ defined in \eqref{eq:taudef}. From \cref{prop:Nwp}, we know that $\lim_{\vep \to 0} \tau_\vep =\infty$ a.s. The next proposition, the culmination of our work so far, is the main result of this subsection. It gives a functional inequality for the expected magnitude of the modulated energy, which serves as the first step in our Gronwall argument, and should be interpreted as the ``rigorous version'' of the inequality \eqref{eq:MEid} from the introduction.

\begin{prop}
\label{prop:MEgb}
For all $t\geq 0$, we have the inequality
\begin{multline}\label{eq:MEgb}
\E\paren*{\Fr_N(\ux_N^t,\mu^t)-\Fr_N(\ux_{N}^0,\mu^0)} \leq 2\sigma\E\paren*{\int_0^t\int_{(\R^d)^2\setminus\triangle}\D\g(x-y)d(\mu_{N}^\ka-\mu^\ka)^{\otimes 2}(x,y)d\ka}\\
+ \E\paren*{\int_0^{t} \left|\int_{(\R^d)^2\setminus\triangle}\paren*{u^\ka(x)-u^\ka(y)}\cdot\nabla\g(x-y)d(\mu_N^\ka-\mu^\ka)^{\otimes 2}(x,y)\right|d\ka}.
\end{multline}
\end{prop}

\begin{remark}
By \cref{prop:MElb}, \cref{prop:MEgb} also implies that there is a constant $C>0$ such that
\begin{multline}
\E\paren*{\left|\Fr_N(\ux_N^t,\mu^t)\right|-\left|\Fr_N(\ux_{N}^0,\mu^0)\right|} \leq C\paren*{\frac{\eta^{-s}(1+|\log\eta|\indic_{s=0})}{N}+\eta^2 + \|\mu\|_{L^\infty}\eta^{d-s}(1+|\log\eta|\indic_{s=0})}\\
+ 2\sigma\E\paren*{\int_0^t\int_{(\R^d)^2\setminus\triangle}\D\g(x-y)d(\mu_{N}^\ka-\mu^\ka)^{\otimes 2}(x,y)d\ka}\\
 + \E\paren*{\int_0^{t} \left|\int_{(\R^d)^2\setminus\triangle}\paren*{u^\ka(x)-u^\ka(y)}\cdot\nabla\g(x-y)d(\mu_N^\ka-\mu^\ka)^{\otimes 2}(x,y)\right|d\ka}
\end{multline}
for any choice of $0<\eta<\min\{\frac{1}{2},\frac{r_0}{2}\}$. Similarly, using \cref{prop:impMElb} if $\g$ is globally superharmonic, \cref{prop:MEgb} also implies that for any $\infty\geq p > \frac{d}{s+2}$, there is a $C_p>0$ such that
\begin{multline}
\E\paren*{\left|\Fr_N(\ux_N^t,\mu^t)\right|-\left|\Fr_N(\ux_{N}^0,\mu^0)\right|} \leq C_p\|\mu\|_{L^\infty}^{\gamma_{s,p}}\eta^{\la_{s,p}}\paren*{1+ \paren*{|\log \eta| + |\log \|\mu\|_{L^\infty}| }\indic_{s=0}} \\
+ \frac{C\eta^{-s}(1+|\log\eta|\indic_{s=0})}{N} + 2\sigma\E\paren*{\int_0^t\int_{(\R^d)^2\setminus\triangle}\D\g(x-y)d(\mu_{N}^\ka-\mu^\ka)^{\otimes 2}(x,y)d\ka}\\
 + \E\paren*{\int_0^{t} \left|\int_{(\R^d)^2\setminus\triangle}\paren*{u^\ka(x)-u^\ka(y)}\cdot\nabla\g(x-y)d(\mu_N^\ka-\mu^\ka)^{\otimes 2}(x,y)\right|d\ka}
\end{multline}
for any $0<\eta<2^{-\frac{dp-d+2p}{d(p-1)}}\|\mu\|_{L^\infty}^{-\frac{1}{d}}$.
\end{remark}

\begin{proof}[Proof of \cref{prop:MEgb}]
Fix $t>0$. By mollifying the initial datum and using the continuous dependence in \cref{prop:lwp}, we assume without loss of generality that $\mu$ is $C^\infty$. Since
\begin{equation}
\begin{split}
&\frac{2\sqrt{2\sigma}}{N}\sum_{i=1}^N\int_0^{(\cdot)}\PV\int_{\R^d\setminus\{x_{i,\vep}^\ka\}} \nabla\g_{(\vep)}(x_{i,\vep}^\ka-y)d(\mu_{N,\vep}^\ka-\mu^\ka)(y)\cdot dW_i^\ka
\end{split}
\end{equation}
is a sum of square-integrable martingales with zero initial expectation, Doob's optional sampling theorem implies that for every $\vep>0$,
\begin{equation}
\E\paren*{\frac{2\sqrt{2\sigma}}{N}\sum_{i=1}^N\int_0^{\tau_\vep\wedge t}\PV\int_{\R^d\setminus\{x_{i,\vep}^\ka\}} \nabla\g_{(\vep)}(x_{i,\vep}^\ka-y)d(\mu_{N,\vep}^\ka-\mu^\ka)(y)\cdot dW_i^\ka} = 0.
\end{equation}

Next, consider the expression
\begin{equation}
\int_{(\R^d)^2\setminus\triangle}\D(\g_{(\vep)}-\g)(x-y)d(\mu_{N,\vep}^\ka-\mu^\ka)^{\otimes 2}(x,y).
\end{equation}
Observe that by definition \eqref{eq:defgvep} of $\g_{(\vep)}$,
\begin{equation}
(\g_{(\vep)} -\g )\ast\D\mu^\ka(x) = -\int_{\R^d}\g(x-y)\chi_\vep(x-y)\D\mu^\ka(y)dy.
\end{equation}
Integrating by parts twice to move the derivatives off $\mu^\ka$ and then applying Cauchy-Schwarz and using \ref{ass2} for $\g$, we find
\begin{align}
\label{eq:sr}
\left|\int_{\R^d}\g(x-y)\chi_\vep(x-y)\D\mu^\ka(y)dy\right| &\lesssim \frac{\| \mu^\ka\|_{L^\infty}}{\vep^{2+s}}\paren*{\int_{|x-y|\leq 2\vep}dy} \lesssim \|\mu^0\|_{L^\infty} \vep^{d-2-s},
\end{align}
where in the ultimate inequality we use the nonincreasing property of $L^p$ norms. Thus,
\begin{equation}
\left|\int_{\R^d}(\g_{(\vep)}-\g)\ast\D\mu^\ka(x)d(\mu_{N,\vep}^\ka-\mu^\ka)(x)\right| \lesssim \|\mu^0\|_{L^\infty} \vep^{d-2-s}.
\end{equation}
Since for every $0\leq \ka\leq \tau_\vep$,
\begin{equation}
\sum_{1\leq i\neq j\leq N} \D\g_{(\vep)}(x_{i,\vep}^\ka-x_{j,\vep}^\ka) = \sum_{1\leq i\neq j\leq N} \D\g(x_i^\ka-x_j^\ka),
\end{equation}
using $\lim_{\vep \to 0} \tau_\vep=\infty$ a.s., it follows that with probability one, for all $0\leq \ka\leq t$,
\begin{equation}
\lim_{\vep\rightarrow 0^+} \left|\int_{(\R^d)^2\setminus\triangle}\paren*{\D\g_{(\vep)}(x-y)d(\mu_{N,\vep}^\ka-\mu^\ka)^{\otimes 2}(x,y) - \D\g(x-y)d(\mu_N^\ka-\mu^\ka)^{\otimes 2}(x,y)}\right| = 0.
\end{equation}
So by dominated convergence,
\begin{equation}
\begin{split}
\lim_{\vep\rightarrow 0^+} \E\paren*{\int_0^{t\wedge \tau_\vep} \int_{(\R^d)^2\setminus\triangle}\D\g_{(\vep)}(x-y)d(\mu_{N,\vep}^\ka-\mu^\ka)^{\otimes 2}(x,y)d\ka} \\
=\E\paren*{\int_0^t \int_{(\R^d)^2\setminus\triangle}\D\g(x-y)d(\mu_{N}^\ka-\mu^\ka)^{\otimes 2}(x,y)d\ka}.
\end{split}
\end{equation}

Next, consider the expression
\begin{equation}
\int_0^t \int_{(\R^d)^2\setminus\triangle}\paren*{u_\vep^\ka(x) -u_\vep^\ka(y)}\cdot\nabla\g_{(\vep)}(x-y)d(\mu_{N,\vep}^\ka-\mu^\ka)^{\otimes 2}(x,y)d\ka.
\end{equation}
We want to show that
\begin{multline}
\label{eq:wts}
\lim_{\vep\rightarrow 0}\E\Bigg(\int_0^{t\wedge\tau_\vep} \left|\int_{(\R^d)^2\setminus\triangle}\paren*{u_\vep^\ka(x)-u_\vep^\ka(y)}\cdot\nabla\g_{(\vep)}(x-y)d(\mu_{N,\vep}^\ka-\mu^\ka)^{\otimes 2}(x,y) \right.\\
-\left.\int_{(\R^d)^2\setminus\triangle}\paren*{u^\ka(x)-u^\ka(y)}\cdot\nabla\g(x-y)d(\mu_{N}^\ka-\mu^\ka)^{\otimes 2}(x,y)\right|d\ka\Bigg) = 0.
\end{multline}
We break up the demonstration of \eqref{eq:wts} into three parts.

\medskip
\textbullet \ Almost surely, we have that for all $0\leq \ka\leq \tau_\vep$
\begin{equation}
\begin{split}
&\int_{(\R^d)^2\setminus\triangle}\paren*{u_\vep^\ka(x)-u_\vep^\ka(y)}\cdot\nabla\g_{(\vep)}(x-y)d(\mu_{N,\vep}^\ka)^{\otimes 2}(x,y)\\
&= \frac{1}{N^2}\sum_{1\leq i\neq j\leq N} \paren*{u_\vep^\ka(x_i^\ka) - u_\vep^\ka(x_j^\ka)}\cdot\nabla\g(x_i^\ka-x_j^\ka).
\end{split}
\end{equation}
Write
\begin{equation}
u-u_\vep = \M\nabla(\g-\g_{(\vep)})\ast\mu.
\end{equation}
We see from the same reasoning as in the estimate \eqref{eq:sr} that
\begin{equation}
\|\nabla^{\otimes 2}(\g-\g_{(\vep)})\ast\mu^\ka\|_{L^\infty} \lesssim \|\mu^\ka\|_{L^\infty} \vep^{d-s-2} \leq \|\mu^0\|_{L^\infty}\vep^{d-s-2}.
\end{equation}
Hence by the mean-value theorem and using assumption \ref{ass2b} for $\g$,
\begin{align}
& \left|\paren*{(\nabla(\g_{(\vep)}-\g)\ast\mu^\ka)(x_i^\ka) - (\nabla(\g_{(\vep)}-\g)\ast\mu^\ka)(x_j^\ka)}\cdot\nabla\g(x_i^\ka-x_j^\ka)\right| \nn  \\ 
&\quad \lesssim \|\mu^0\|_{L^\infty} \vep^{d-s-2} |x_i^\ka-x_j^\ka||\nabla\g(x_i^\ka-x_j^\ka)| \nn\\
&\quad \lesssim \begin{cases} \|\mu^0\|_{L^\infty} \vep^{d-2}, & {s=0} \\ \|\mu^0\|_{L^\infty} \vep^{d-2-s}\g(x_i^\ka-x_j^\ka), & 0<s<d-2,\end{cases}
\end{align}
where to obtain the ultimate line in the case $0<s<d-2$ we assume $|x_i^\ka-x_j^\ka|<r_0$. For the case $s=0$,  the preceding estimate suffices. For the case $0<s<d-2$, we need to deal with the factor $\g(x_i-x_j)$. To this end, we set $H_N(\ux_N)\coloneqq \sum_{1\leq i\neq j\leq N}\g(x_i-x_j)$. Since $\g$ is positive inside the ball $B(0,r_0)$ and $|\g|\leq C$ outside $B(0,r_0)$ by assumptions \ref{ass2} and \ref{ass2b}, we find that
\begin{align}
\E\paren*{\int_0^{t\wedge\tau_\vep} H_N(\ux_N^{\ka})d\ka} &= \E\paren*{\int_0^{t\wedge\tau_\vep} H_{N,\vep}(\ux_{N,\vep}^{\ka})d\ka} \nn\\
&\quad \leq \E\paren*{ \int_0^t\paren*{H_{N,\vep}(\ux_{N,\vep}^{\ka}) + CN^2}d\ka} \nn\\
&\quad = \int_0^t \E\paren*{H_{N,\vep}(\ux_{N,\vep}^{\ka}) + CN^2}d\ka \nn\\
&\quad \leq t\paren*{\E(H_{N}(\ux_{N}^0)) +CN^2},
\end{align}
where the ultimate line follows from the proof of \eqref{eq:hamub}, which is valid for any $t$, and the fact that we may always assume $\vep < \min_{i\neq j}|x_{i}^0-x_{j}^0|$. It now follows that
\begin{equation}
\E\paren*{\int_0^{t\wedge\tau_\vep}\left|\int_{(\R^d)^2\setminus\triangle}\paren*{(u_\vep^\ka-u^\ka)(x)-(u_\vep^\ka-u^\ka)(y)}\cdot\nabla\g_{(\vep)}(x-y)d(\mu_{N,\vep}^\ka)^{\otimes 2}(x,y)\right|d\ka}
\end{equation}
vanishes as $\vep\rightarrow 0^+$.

\textbullet  \ 
Almost surely, for $0\leq \ka\leq \tau_\vep$,
\begin{equation}
\begin{split}
&\int_{(\R^d)^2\setminus\triangle} \paren*{u_\vep^\ka(x)-u_\vep^\ka(y)}\cdot\nabla\g_{(\vep)}(x-y)d\mu_{N,\vep}^\ka(x)d\mu^\ka(y) \\
&=\frac{1}{N}\sum_{i=1}^N (\M\nabla\g_{(\vep)}\ast\mu^\ka)(x_i^\ka)\cdot (\nabla\g_{(\vep)}\ast\mu^\ka)(x_i^\ka) - \frac{1}{N}\sum_{i=1}^N\paren*{\g_{(\vep)}\ast\paren*{\div(\mu^\ka u_\vep^\ka)}}(x_i^\ka).
\end{split}
\end{equation}
By \cref{rem:Linfg}, the same reasoning as \eqref{eq:sr} and the nonincreasing property of $L^p$ norms, we have 
\begin{align}
&\left|(\M\nabla\g_{(\vep)}\ast\mu^\ka)\cdot (\nabla\g_{(\vep)}\ast\mu^\ka) - (\M\nabla\g\ast\mu^\ka)\cdot (\nabla\g\ast\mu^\ka)\right| \nn\\
&\quad \leq \|\nabla(\g_{(\vep)}-\g)\ast\mu^\ka\|_{L^\infty}(\|\nabla\g_{(\vep)}\ast\mu^\ka\|_{L^\infty}+\|\nabla\g\ast\mu^\ka\|_{L^\infty}) \nn\\
&\quad \lesssim \vep^{d-s-1}\|\mu^0\|_{L^\infty}^{\frac{d+s+1}{d}}.
\end{align}
Similarly,
\begin{align}
&\left|\paren*{\g_{(\vep)}\ast\div(\mu^\ka u_\vep^\ka)}(x_i^\ka) - \g\ast\paren*{\div(\mu^\ka u^\ka)}(x_i^\ka)\right| \nn\\
&\quad \lesssim \vep^{d-s-1} \|\mu^\ka u_\vep^\ka\|_{L^\infty} + \|\mu^\ka (u_\vep^\ka-u^\ka)\|_{L^1}^{1-\frac{s+1}{d}} \|\mu^\ka(u_\vep^\ka - u^\ka)\|_{L^\infty}^{\frac{s+1}{d}} \nn\\
&\quad \lesssim \vep^{d-s-1} \|\mu^0\|_{L^\infty}^{\frac{d+s+1}{d}}.
\end{align}
After a little bookkeeping, we find that
\begin{multline}
\E\Bigg(\int_0^{t\wedge\tau_\vep} \left|\int_{(\R^d)^2\setminus\triangle}\paren*{(u_\vep^\ka(x)-u_\vep^\ka(y)}\cdot\nabla\g_{(\vep)}(x-y)d\mu_{N,\vep}^\ka(x)d\mu^\ka(y)\right. \\
-\left.\int_{(\R^d)^2\setminus\triangle}\paren*{u^\ka(x)-u^\ka(y)}\cdot\nabla\g(x-y)d\mu_{N}^\ka(x)d\mu^\ka(y)\right|d\ka\Bigg) \lesssim t\vep^{d-s-1} \|\mu^0\|_{L^\infty}^{\frac{d+s+1}{d}},
\end{multline}
which evidently vanishes as $\vep\rightarrow 0^+$.

\textbullet \ 
Observe that
\begin{equation}
\begin{split}
&\int_{(\R^d)^2\setminus\triangle}\paren*{u_\vep^\ka(x)-u_\vep^\ka(y)}\cdot \nabla\g_{(\vep)}(x-y)d(\mu^\ka)^{\otimes 2}(x,y)=2\int_{\R^d} u_\vep^\ka(x) \cdot (\nabla\g_{(\vep)}\ast\mu^\ka)(x)d\mu^\ka(x)
\end{split}
\end{equation}
By \cref{lem:LinfRP} and the nonincreasing property of $L^p$ norms, arguing similarly as above, we have
\begin{align}
&\int_{\R^d}\left| u_\vep^\ka(x) \cdot (\nabla\g_{(\vep)}\ast\mu^\ka)(x) - u^\ka(x)\cdot(\nabla\g\ast\mu^\ka)(x)\right|d\mu^\ka(x) \nn\\
&\quad \leq \|\nabla(\g-\g_{(\vep)})\ast\mu^\ka\|_{L^\infty}(\|\nabla\g_{(\vep)}\ast\mu^\ka\|_{L^\infty} + \|\nabla\g\ast\mu^\ka\|_{L^\infty}) \nn\\
&\quad \lesssim \vep^{d-s-1}\|\mu^0\|_{L^\infty}^{\frac{d+s+1}{d}}.
\end{align}
Therefore,
\begin{equation}
\E\paren*{\int_0^{t\wedge\tau_\vep}\int_{\R^d}\left|u_\vep^\ka(x) \cdot (\nabla\g_{(\vep)}\ast\mu^\ka)(x) - u^\ka(x)\cdot(\nabla\g\ast\mu^\ka)(x) \right|  d\mu^\ka(x)d\ka} \lesssim t\vep^{d-s-1}\|\mu^0\|_{L^\infty}^{\frac{d+s+1}{d}},
\end{equation}
which evidently tends to zero as $\vep\rightarrow 0^+$. With this last bit, the desired result \eqref{eq:wts} now follows.

\medskip
Combining the above results, we see that we have shown the inequality
\begin{multline}
\lim_{\vep\rightarrow 0^+} \E\paren*{\Fr_{N,\vep}(\ux_{N,\vep}^{\tau_\vep\wedge t}, \mu^{\tau_\vep\wedge t}) - \Fr_N(\ux_{N}^0,\mu^0)} \leq  2\sigma\E\paren*{\int_0^t\int_{\R^d}\D\g(x-y)d(\mu_N^\ka-\mu^\ka)^{\otimes 2}(x,y)d\ka}\\
+ \E\paren*{\int_0^{t} \left|\int_{(\R^2)^2\setminus\triangle}\paren*{(u^\ka(x)-u^\ka(y)}\cdot\nabla\g(x-y)d(\mu_N^\ka-\mu^\ka)^{\otimes 2}(x,y)\right|d\ka}.
\end{multline}
Since $\Fr_{N,\vep}(\ux_{N,\vep},\mu)$ is bounded from below uniformly in the noise and time by virtue of \cref{prop:MElb}, we can conclude the proof by applying Fatou's lemma.
\end{proof}

\section{Gronwall argument}\label{sec:Gron}
We now have all the ingredients necessary to conclude our Gronwall argument for the modulated energy, thereby proving \Cref{thm:lmain,thm:gmain}. We divide this section into two subsections. In the first subsection, we consider the case where the admissible potential $\g$ is only superharmonic in a neighborhood of the origin (i.e. $r_0<\infty$ in assumption \ref{ass1}). For such potentials, we obtain decay bounds for the modulated $\Fr_N(\ux_N^t,\mu^t)$ as $N\rightarrow\infty$ which grow linearly in time. The conclusion is the proof of \cref{thm:lmain}. In the second subsection, we consider  admissible potentials which are superharmonic on $\R^{d}$ (i.e. $r_0=\infty$ in assumption \ref{ass1}). Under this stronger assumption, we can prove decay bounds for $\Fr_N(\ux_N^t,\mu^t)$ which are uniform on the interval $[0,\infty)$. The conclusion is the proof of \cref{thm:gmain}.

\subsection{Linear-in-time estimates}\label{ssec:linGron}
Applying \cref{cor:MElb} and \cref{prop:rcomm} pointwise in time to the first and second terms, respectively, of the right-hand side of inequality \eqref{eq:MEgb} and using \cref{rem:MEbal} to control $|\Fr_N(\ux_N^t,\mu^t)|$ in terms of $\Fr_N(\ux_N^t,\mu^t)$, we find that
\begin{multline}\label{eq:inGron}
\E(|\Fr_N(\ux_N^t,\mu^t)|) \leq |\Fr_N(\ux_N^0,\mu^0)| + C(1+\|\mu^t\|_{L^\infty})N^{-\frac{2}{2+s}}\paren*{1+(\log N)\indic_{s=0}} \\
+C\sigma\int_0^t (1+\|\mu^\ka\|_{L^\infty})N^{-\frac{\min\{2,d-s-2\}}{\min\{s+4,d\}}}d\ka + C\int_0^t \Bigg(\|\nabla u^\ka\|_{L^\infty}+\|\Dm^{\frac{d-s}{2}}u^\ka\|_{L^{\frac{2d}{d-2-s}}}\Bigg)\Bigg(\Fr_N(\ux_N^\ka,\mu^\ka)  \\
+ C N^{-\frac{s+3}{(s+2)(s+1)}} \|\Dm^{s+1-d}\mu^\ka\|_{L^\infty}+C(1+\|\mu^\ka\|_{L^\infty})N^{-\frac{2}{(s+2)(s+1)}}\Big(1+(\log N)\indic_{s=0}\Big)\Bigg)d\ka,
\end{multline}
where we have defined $u\coloneqq \M\nabla\g\ast\mu$. We remind the reader that the constant $C$ depends on $r_0$ from assumption \ref{ass1}.

Using \cref{rem:Linfg} and the fact that $\|\mu^\ka\|_{L^1}=1$, we see that
\begin{equation}\label{eq:gradu}
\|\nabla u^\ka\|_{L^\infty} \leq C \|\mu^\ka\|_{L^\infty}^{\frac{s+2}{d}}.
\end{equation}
Similarly, using the commutativity of Fourier multipliers together with the Hardy-Littlewood-Sobolev lemma, followed by H\"older's inequality
\begin{align}
\|\Dm^{\frac{d-s}{2}}u^\ka\|_{L^{\frac{2d}{d-2-s}}} = \|\M\nabla\g\ast(\Dm^{\frac{d-s}{2}}\mu^\ka)\|_{L^{\frac{2d}{d-2-s}}} &\leq C\|\mathcal{I}_{\frac{d-s-2}{2}}(\mu^\ka)\|_{L^{\frac{2d}{d-2-s}}} \nn\\
&\leq C\|\mu^\ka\|_{L^{\frac{d}{d-2-s}}}\nn\\
&\leq C\|\mu^\ka\|_{L^\infty}^{\frac{2+s}{d}}. \label{eq:dmu}
\end{align}
By another application of \cref{lem:LinfRP},
\begin{equation}\label{eq:dmmu}
\|\Dm^{s+1-d}\mu^\ka\|_{L^\infty} \leq \|\mu^\ka\|_{L^\infty}^{\frac{s+1}{d}}.
\end{equation}
Applying the bounds \eqref{eq:gradu}, \eqref{eq:dmu}, \eqref{eq:dmmu} to the right-hand side of \eqref{eq:inGron}, then applying the Gronwall-Bellman lemma, we find that
\begin{equation}\label{eq:Gron}
\E(|\Fr_N(\ux_N^t,\mu^t)|) \leq A_N^t\exp\paren*{C\int_0^t \|\mu^\ka\|_{L^\infty}^{\frac{s+2}{d}}d\ka},
\end{equation}
where the time-dependent prefactor $A_N^t$ is defined by
\begin{multline}
A_N^t \coloneqq |\Fr_N(\ux_N^0,\mu^0)| + C(1+\|\mu^t\|_{L^\infty})N^{-\frac{2}{2+s}}\paren*{1+(\log N)\indic_{s=0}} \\
+C\sigma\int_0^t (1+\|\mu^\ka\|_{L^\infty})N^{-\frac{\min\{2,d-s-2\}}{\min\{s+4,d\}}}d\ka + C\int_0^t \|\mu^\ka\|_{L^\infty}^{\frac{s+2}{d}}\Bigg(\|\mu^\ka\|_{L^\infty}^{\frac{s+1}{d}} N^{-\frac{s+3}{(s+2)(s+1)}} \\
+ C(1+\|\mu^\ka\|_{L^\infty})N^{-\frac{2}{(s+2)(s+1)}}\paren*{1+(\log N)\indic_{s=0}}\Bigg)d\ka.
\end{multline}
Without loss of generality, we may assume that $t\geq 1$. Split the interval $[0,t]$ into $[0,1]$ and $[1,t]$. On $[0,1]$, we use the trivial bound $\|\mu^\ka\|_{L^\infty}\leq \|\mu^0\|_{L^\infty}$; and on $[1,t]$, we use the bound $\|\mu^\ka\|_{L^\infty} \leq C(\sigma\ka)^{-\frac{d}{2}}$, which comes from \cref{prop:dcay}. It then follows that
\begin{equation}\label{eq:muint}
\int_0^t \|\mu^\ka\|_{L^\infty}^{\frac{s+2}{d}}d\ka \leq \|\mu^0\|_{L^\infty}^{\frac{s+2}{d}} + C\int_1^t (\sigma\ka)^{-\frac{s+2}{2}}d\ka \leq \|\mu^0\|_{L^\infty}^{\frac{s+2}{d}} + C\paren*{\frac{2}{s\sigma^{\frac{s+2}{2}}}\indic_{s>0}+ \frac{(\log t)}{\sigma}\indic_{s=0}},
\end{equation}
\begin{equation}
\int_0^t \|\mu^\ka\|_{L^\infty}^{\frac{2s+3}{d}}d\ka \leq \|\mu^0\|_{L^\infty}^{\frac{2s+3}{d}} + \frac{C}{\sigma^{\frac{2s+3}{2}}},
\end{equation}
and
\begin{multline}
A_N^t \leq |\Fr_N(\ux_N^0,\mu^0)| + C(1+\|\mu^0\|_{L^\infty})N^{-\frac{2}{2+s}}\paren*{1+(\log N)\indic_{s=0}} + C\sigma t(1+\|\mu^0\|_{L^\infty})N^{-\frac{\min\{2,d-s-2\}}{\min\{s+4,d\}}}\\
+ C\paren*{\|\mu^0\|_{L^\infty}^{\frac{2s+3}{d}} + \sigma ^{-\frac{2s+3}{2}}}N^{-\frac{s+3}{(s+2)(s+1)}} + Ct(1+\|\mu^0\|_{L^\infty}) N^{-\frac{2}{(s+2)(s+1)}}\paren*{1+(\log N)\indic_{s=0}}.
\end{multline}
Applying the preceding bounds  and inserting into \eqref{eq:Gron}, we conclude
\begin{multline}\label{eq:linFin}
\E(|\Fr_N(\ux_N^t,\mu^t)|) \leq \exp\paren*{C\paren*{\|\mu^0\|_{L^\infty}^{\frac{s+2}{d}} + \frac{2}{s\sigma^{\frac{s+2}{2}}}\indic_{s>0}+ (\log t^{\frac{1}{\sigma}})\indic_{s=0}}} \Bigg(|\Fr_N(\ux_N^0,\mu^0)|\\
+C(1+\|\mu^0\|_{L^\infty})N^{-\frac{2}{2+s}}\paren*{1+(\log N)\indic_{s=0}} + C\sigma t(1+\|\mu^0\|_{L^\infty})N^{-\frac{\min\{2,d-s-2\}}{\min\{s+4,d\}}} \\
+C\paren*{\|\mu^0\|_{L^\infty}^{\frac{2s+3}{d}} + \sigma ^{-\frac{2s+3}{2}}}N^{-\frac{s+3}{(s+2)(s+1)}} + Ct(1+\|\mu^0\|_{L^\infty}) N^{-\frac{2}{(s+2)(s+1)}}\paren*{1+(\log N)\indic_{s=0}}\Bigg).
\end{multline}
Comparing \eqref{eq:linFin} to \eqref{eq:lmain}, we see that we have proved \cref{thm:lmain}.

\subsection{Global-in-time estimates}\label{ssec:globGron}
We now assume that the potential $\g$ is globally superharmonic and show, using the results of \cref{ssec:MEnew}, global-in-time bounds for the modulated energy $\Fr_N(\ux_N^t,\mu^t)$ for the range $0<s<d-2$ and almost-global-in-time bounds if $s=0$. This proves \cref{thm:gmain}.

Applying \cref{cor:impMElb} and \cref{prop:imprcomm} pointwise in time to the first and second terms, respectively, of the right-hand side of inequality \eqref{eq:MEgb} and using \cref{rem:impMEbal} to control $|\Fr_N(\ux_N^t,\mu^t)|$ in terms of $\Fr_N(\ux_N^t,\mu^t)$, we find that
\begin{multline}\label{eq:globinGron}
\E\paren*{|\Fr_N(\ux_N^t,\mu^t)|} \leq |\Fr_N(\ux_N^0,\mu^0)| + C_{p}\|\mu^t\|_{L^\infty}^{\frac{s}{d}}N^{-\frac{\la_{s,p}}{\la_{s,p}+s}}\paren*{1+\paren*{|\log\|\mu\|_{L^\infty}| + \log N}\indic_{s=0}}\\
+ C\sigma\int_0^t \|\mu^\ka\|_{L^\infty}^{\frac{s+2}{d}}\paren*{C_q N^{-\frac{\la_{s+2,q}}{\la_{s+2,q}+s+2}}\indic_{0\leq s\leq d-4} + N^{-\frac{d-s-2}{d}}\indic_{s>d-4}}d\ka \\
+  C\int_0^t \|\nabla u^\ka\|_{L^\infty}\|\Dm^{s+1-d}\mu^\ka\|_{L^\infty}N^{-\frac{s+1+\la_{s,p}}{(s+\la_{s,p})(1+s)}}d\ka \\
+ C\int_0^t \Bigg(\|\nabla u^\ka\|_{L^\infty}+\|\Dm^{\frac{d-s}{2}}u^\ka\|_{L^{\frac{2d}{d-2-s}}}\Bigg)\Bigg(\Fr_N(\ux_N^\ka,\mu^\ka) \\
+ C_p \paren*{1+\|\mu^\ka\|_{L^\infty}^{\ga_{s,p}}} N^{-\frac{\la_{s,p}}{(s+\la_{s,p})(1+s)}}\paren*{1+ \paren*{\log N  + |\log\|\mu^\ka\|_{L^\infty}|}\indic_{s=0}}\Bigg)d\ka
\end{multline}
for any choices $\infty \geq p > \frac{d}{s+2}$ and $\infty\geq q> \frac{d}{s+4}$. The reader will recall the exponents $\ga_{s,p},\la_{s,p}$ from \eqref{eq:gamdef}. Implicit here is the assumption that $N > (2^{\frac{dp-d+p}{(p-1)}}\|\mu^\ka\|_{L^\infty})^{\frac{(s+1)(s+\la_{s,p})}{d}}$ for every $\ka\in [0,t]$, as required by \cref{prop:imprcomm}. We can satisfy this constraint by assuming that $N>(2^{\frac{dp-d+p}{(p-1)}}\|\mu^0\|_{L^\infty})^{\frac{(s+1)(s+\la_{s,p})}{d}}$, since $\|\mu^\ka\|_{L^\infty}$ is nonincreasing.

Applying the bounds \eqref{eq:gradu}, \eqref{eq:dmu}, \eqref{eq:dmmu} to the right-hand side of \eqref{eq:globinGron}, then applying the Gronwall-Bellman lemma, we find that
\begin{equation}\label{eq:globGron}
\E(|\Fr_N(\ux_N^t,\mu^t)|) \leq B_N^t\exp\paren*{C\int_0^t \|\mu^\ka\|_{L^\infty}^{\frac{s+2}{d}}d\ka},
\end{equation}
where the time-dependent prefactor $B_N^t$ is given by
\begin{multline}
B_N^t \coloneqq |\Fr_N(\ux_N^0,\mu^0)| + C_{p}\|\mu^t\|_{L^\infty}^{\frac{s}{d}}N^{-\frac{\la_{s,p}}{\la_{s,p}+s}}\paren*{1+\paren*{|\log\|\mu^t\|_{L^\infty}| + \log N}\indic_{s=0}}\\
+ C \int_0^t \|\mu^\ka\|_{L^\infty}^{\frac{3+2s}{d}}N^{-\frac{s+1+\la_{s,p}}{(s+\la_{s,p})(1+s)}}d\ka+ C\sigma\int_0^t \|\mu^\ka\|_{L^\infty}^{\frac{s+2}{d}}\paren*{C_q N^{-\frac{\la_{s+2,q}}{\la_{s+2,q}+s+2}}\indic_{0\leq s\leq d-4} + N^{-\frac{d-s-2}{d}}\indic_{s>d-4}}d\ka \\
+ C_p\int_0^t \|\mu^\ka\|_{L^\infty}^{\frac{s+2}{d}}\paren*{1+\|\mu^\ka\|_{L^\infty}^{\ga_{s,p}}} N^{-\frac{\la_{s,p}}{(s+\la_{s,p})(1+s)}}\paren*{1+ \paren*{\log N  + |\log\|\mu^\ka\|_{L^\infty}|}\indic_{s=0}}d\ka.
\end{multline}
Assuming $t\geq 1$ and splitting the interval $[0,t]$ into $[0,1], [1,t]$ exactly as in the last subsection, we find that
\begin{multline}
B_N^t \leq |\Fr_N(\ux_N^0,\mu^0)|+C\paren*{\|\mu^0\|_{L^\infty}^{\frac{3+2s}{d}} + \sigma^{-\frac{3+2s}{2}}}N^{-\frac{s+1+\la_{s,p}}{(s+\la_{s,p})(1+s)}} \\
+ C_{p}\min\{\|\mu^0\|_{L^\infty}^{\frac{s}{d}}, (\sigma t)^{-\frac{s}{2}}\}N^{-\frac{\la_{s,p}}{\la_{s,p}+s}}\paren*{1+\paren*{\max\{|\log\|\mu^0\|_{L^\infty}|, |\log(\sigma t)|\} + \log N}\indic_{s=0}} \\
+ C\sigma\paren*{\|\mu^0\|_{L^\infty}^{\frac{s+2}{d}} + \frac{2}{s\sigma^{\frac{s+2}{2}}}\indic_{s>0}+ (\log t^{\frac{1}{\sigma}})\indic_{s=0}}\paren*{C_q N^{-\frac{\la_{s+2,q}}{\la_{s+2,q}+s+2}}\indic_{0\leq s\leq d-4} + N^{-\frac{d-s-2}{d}}\indic_{s>d-4}}\\
+ C_p\paren*{1+\|\mu^0\|_{L^\infty}^{\ga_{s,p}}}\paren*{\|\mu^0\|_{L^\infty}^{\frac{s+2}{d}} +\frac{2}{s\sigma^{\frac{s+2}{2}}}\indic_{s>0}+ (\log t^{\frac{1}{\sigma}})\indic_{s=0}}.
\end{multline}
Applying this bound to the right-hand side of \eqref{eq:globGron} and using \eqref{eq:muint} for the exponential factor, we conclude that
\begin{multline}\label{eq:globFin}
\E\paren*{|\Fr_N(\ux_N^t,\mu^t)|} \leq \exp\paren*{C\paren*{\|\mu^0\|_{L^\infty}^{\frac{s+2}{d}} +\frac{2}{s\sigma^{\frac{s+2}{2}}}\indic_{s>0}+ (\log t^{\frac{1}{\sigma}})\indic_{s=0}}} \Bigg(|\Fr_N(\ux_N^0,\mu^0)|\\
+ C_{p}\min\{\|\mu^0\|_{L^\infty}^{\frac{s}{d}}, (\sigma t)^{-\frac{s}{2}}\}N^{-\frac{\la_{s,p}}{\la_{s,p}+s}}\paren*{1+\paren*{\max\{|\log\|\mu^0\|_{L^\infty}|, |\log(\sigma t)|\} + \log N}\indic_{s=0}}\\
+ C\sigma\paren*{\|\mu^0\|_{L^\infty}^{\frac{s+2}{d}} + \frac{2}{s\sigma^{\frac{s+2}{2}}}\indic_{s>0}+ (\log t^{\frac{1}{\sigma}})\indic_{s=0}}\paren*{C_q N^{-\frac{\la_{s+2,q}}{\la_{s+2,q}+s+2}}\indic_{0\leq s\leq d-4} + N^{-\frac{d-s-2}{d}}\indic_{s>d-4}}\\
+ C\paren*{\|\mu^0\|_{L^\infty}^{\frac{3+2s}{d}} + \sigma^{-\frac{3+2s}{2}}}N^{-\frac{s+1+\la_{s,p}}{(s+\la_{s,p})(1+s)}} +C_p\paren*{1+\|\mu^0\|_{L^\infty}^{\ga_{s,p}}}\paren*{\|\mu^0\|_{L^\infty}^{\frac{s+2}{d}} +\frac{2}{s\sigma^{\frac{s+2}{2}}}\indic_{s>0}+ (\log t^{\frac{1}{\sigma}})\indic_{s=0}}\Bigg).
\end{multline}

\bibliographystyle{alpha}
\bibliography{PointVortex}

\begin{thebibliography}{MMW15}

\bibitem[BCnC11]{BCC2011}
Fran\c{c}ois Bolley, Jos\'{e}~A. Ca\~{n}izo, and Jos\'{e}~A. Carrillo.
\newblock Stochastic mean-field limit: non-{L}ipschitz forces and swarming.
\newblock {\em Math. Models Methods Appl. Sci.}, 21(11):2179--2210, 2011.

\bibitem[BGM10]{BGM2010}
Fran\c{c}ois Bolley, Arnaud Guillin, and Florent Malrieu.
\newblock Trend to equilibrium and particle approximation for a weakly
  selfconsistent {V}lasov-{F}okker-{P}lanck equation.
\newblock {\em M2AN Math. Model. Numer. Anal.}, 44(5):867--884, 2010.

\bibitem[BJW19a]{BJW2019edp}
Didier Bresch, Pierre-Emmanuel Jabin, and Zhenfu Wang.
\newblock Modulated free energy and mean field limit.
\newblock {\em S{\'e}minaire Laurent Schwartz--EDP et applications}, pages
  1--22, 2019.

\bibitem[BJW19b]{BJW2019crm}
Didier Bresch, Pierre-Emmanuel Jabin, and Zhenfu Wang.
\newblock On mean-field limits and quantitative estimates with a large class of
  singular kernels: application to the {P}atlak-{K}eller-{S}egel model.
\newblock {\em C. R. Math. Acad. Sci. Paris}, 357(9):708--720, 2019.

\bibitem[BJW20]{BJW2020}
Didier Bresch, Pierre-Emmanuel Jabin, and Zhenfu Wang.
\newblock Mean-field limit and quantitative estimates with singular attractive
  kernels.
\newblock {\em arXiv preprint arXiv:2011.08022}, 2020.

\bibitem[BO19]{BO2019}
Robert~J. Berman and Magnus \"{O}nnheim.
\newblock Propagation of chaos for a class of first order models with singular
  mean field interactions.
\newblock {\em SIAM J. Math. Anal.}, 51(1):159--196, 2019.

\bibitem[Car91]{Carlen1991}
Eric~A. Carlen.
\newblock Some integral identities and inequalities for entire functions and
  their application to the coherent state transform.
\newblock {\em J. Funct. Anal.}, 97(1):231--249, 1991.

\bibitem[CB18]{CB2018}
{L\'{e}na\"{\i}c} Chizat and Francis Bach.
\newblock On the global convergence of gradient descent for over-parameterized
  models using optimal transport.
\newblock In {\em Proceedings of the 32nd International Conference on Neural
  Information Processing Systems}, NIPS18, pages 3040--3050, Red Hook, NY, USA,
  2018. Curran Associates Inc.

\bibitem[CCH14]{CCH2014}
Jos\'{e}~Antonio Carrillo, Young-Pil Choi, and Maxime Hauray.
\newblock The derivation of swarming models: mean-field limit and {W}asserstein
  distances.
\newblock In {\em Collective dynamics from bacteria to crowds}, volume 553 of
  {\em CISM Courses and Lect.}, pages 1--46. Springer, Vienna, 2014.

\bibitem[CFP12]{CFP2012}
Jos\'{e}~A. Carrillo, Lucas C.~F. Ferreira, and Juliana~C. Precioso.
\newblock A mass-transportation approach to a one dimensional fluid mechanics
  model with nonlocal velocity.
\newblock {\em Adv. Math.}, 231(1):306--327, 2012.

\bibitem[CGM08]{CGM2008}
P.~Cattiaux, A.~Guillin, and F.~Malrieu.
\newblock Probabilistic approach for granular media equations in the
  non-uniformly convex case.
\newblock {\em Probab. Theory Related Fields}, 140(1-2):19--40, 2008.

\bibitem[Cho73]{Chorin1973}
Alexandre~Joel Chorin.
\newblock Numerical study of slightly viscous flow.
\newblock {\em J. Fluid Mech.}, 57(4):785--796, 1973.

\bibitem[CL90]{CL1990}
Eric~A. Carlen and Michael Loss.
\newblock Extremals of functionals with competing symmetries.
\newblock {\em J. Funct. Anal.}, 88(2):437--456, 1990.

\bibitem[CL95]{CL1995}
Eric~A. Carlen and Michael Loss.
\newblock {Optimal smoothing and decay estimates for viscously damped
  conservation laws, with applications to the $2$-D Navier-Stokes equation}.
\newblock {\em Duke Math. J.}, 81(1):135 -- 157, 1995.

\bibitem[Dav87]{Davies1987}
E.~B. Davies.
\newblock Explicit constants for {G}aussian upper bounds on heat kernels.
\newblock {\em Amer. J. Math.}, 109(2):319--333, 1987.

\bibitem[DEGZ20]{DEGZ2020}
Alain Durmus, Andreas Eberle, Arnaud Guillin, and Raphael Zimmer.
\newblock An elementary approach to uniform in time propagation of chaos.
\newblock {\em Proc. Amer. Math. Soc.}, 148(12):5387--5398, 2020.

\bibitem[DT21]{DT2021}
Fran{\c{c}}ois Delarue and Alvin Tse.
\newblock Uniform in time weak propagation of chaos on the torus.
\newblock {\em arXiv preprint arXiv:2104.14973}, 2021.

\bibitem[Due16]{Duerinckx2016}
M.~Duerinckx.
\newblock Mean-field limits for some {Riesz} interaction gradient flows.
\newblock {\em SIAM Journal on Mathematical Analysis}, 48(3):2269--2300, 2016.

\bibitem[FHM14]{FHM2014}
Nicolas Fournier, Maxime Hauray, and St\'{e}phane Mischler.
\newblock Propagation of chaos for the 2{D} viscous vortex model.
\newblock {\em J. Eur. Math. Soc. (JEMS)}, 16(7):1423--1466, 2014.

\bibitem[FJ17]{FJ2017}
Nicolas Fournier and Benjamin Jourdain.
\newblock Stochastic particle approximation of the {K}eller-{S}egel equation
  and two-dimensional generalization of {B}essel processes.
\newblock {\em Ann. Appl. Probab.}, 27(5):2807--2861, 2017.

\bibitem[GBM21]{GlBM2021}
Arnaud Guillin, Pierre~Le Bris, and Pierre Monmarch\'{e}.
\newblock Uniform in time propagation of chaos for the 2d vortex model and
  other singular stochastic systems.
\newblock {\em arXiv preprint arXiv:2108.08675}, 2021.

\bibitem[Gol16]{Golse2016ln}
Fran\c{c}ois Golse.
\newblock On the dynamics of large particle systems in the mean field limit.
\newblock In {\em Macroscopic and large scale phenomena: coarse graining, mean
  field limits and ergodicity}, volume~3 of {\em Lect. Notes Appl. Math.
  Mech.}, pages 1--144. Springer, 2016.

\bibitem[GQn15]{GQ2015}
David Godinho and Cristobal Qui\~ninao.
\newblock Propagation of chaos for a subcritical {Keller-Segel} model.
\newblock {\em Annales de l'I.H.P. Probabilit\'es et statistiques},
  51(3):965--992, 2015.

\bibitem[Gra14a]{Grafakos2014c}
Loukas Grafakos.
\newblock {\em Classical Fourier Analysis}.
\newblock Number 249 in Graduate Texts in Mathematics. Springer, third edition,
  2014.

\bibitem[Gra14b]{Grafakos2014m}
Loukas Grafakos.
\newblock {\em Modern Fourier Analysis}.
\newblock Number 250 in Graduate Texts in Mathematics. Springer, third edition,
  2014.

\bibitem[Gro75]{Gross1975}
Leonard Gross.
\newblock Logarithmic {S}obolev inequalities.
\newblock {\em Amer. J. Math.}, 97(4):1061--1083, 1975.

\bibitem[GW05]{GW2005}
Thierry Gallay and C.~Eugene Wayne.
\newblock Global stability of vortex solutions of the two-dimensional
  {N}avier-{S}tokes equation.
\newblock {\em Comm. Math. Phys.}, 255(1):97--129, 2005.

\bibitem[Hau09]{Hauray2009}
Maxime Hauray.
\newblock Wasserstein distances for vortices approximation of {E}uler-type
  equations.
\newblock {\em Math. Models Methods Appl. Sci.}, 19(8):1357--1384, 2009.

\bibitem[HK02]{HK2002}
Rainer Hegselmann and Ulrich Krause.
\newblock {Opinion Dynamics and Bounded Confidence Models, Analysis and
  Simulation}.
\newblock {\em Journal of Artificial Societies and Social Simulation},
  5(3):1--2, 2002.

\bibitem[HM14]{HM2014}
Maxime Hauray and St\'{e}phane Mischler.
\newblock On {K}ac's chaos and related problems.
\newblock {\em J. Funct. Anal.}, 266(10):6055--6157, 2014.

\bibitem[Hol16]{Holding2016}
Thomas Holding.
\newblock Propagation of chaos for {H\"older} continuous interaction kernels
  via {Glivenko-Cantelli}.
\newblock {\em arXiv preprint arXiv:1608.02877}, 2016.

\bibitem[Jab14]{Jab2014}
Pierre-Emmanuel Jabin.
\newblock A review of the mean field limits for {V}lasov equations.
\newblock {\em Kinet. Relat. Models}, 7(4):661--711, 2014.

\bibitem[JW16]{JW2016}
Pierre-Emmanuel Jabin and Zhenfu Wang.
\newblock Mean field limit and propagation of chaos for {V}lasov systems with
  bounded forces.
\newblock {\em J. Funct. Anal.}, 271(12):3588--3627, 2016.

\bibitem[JW17]{JW2017_survey}
Pierre-Emmanuel Jabin and Zhenfu Wang.
\newblock {Mean field limit for stochastic particle systems}.
\newblock In {\em Act. Part. {V}ol. 1. {A}dvances theory, Model. Appl.}, Model.
  Simul. Sci. Eng. Technol., pages 379--402. Birkh{\"{a}}user/Springer, Cham,
  2017.

\bibitem[JW18]{JW2018}
Pierre-Emmanuel Jabin and Zhenfu Wang.
\newblock Quantitative estimates of propagation of chaos for stochastic systems
  with {$W^{-1,\infty}$} kernels.
\newblock {\em Invent. Math.}, 214(1):523--591, 2018.

\bibitem[Kra00]{Krause2000}
U.~Krause.
\newblock A discrete nonlinear and non-autonomous model of consensus formation.
\newblock {\em Communications in Difference Equations}, 2000.

\bibitem[KS91]{KS1991}
Ioannis Karatzas and Steven~E. Shreve.
\newblock {\em Brownian motion and stochastic calculus}, volume 113 of {\em
  Graduate Texts in Mathematics}.
\newblock Springer-Verlag, New York, second edition, 1991.

\bibitem[Lac21]{Lacker2021}
Daniel Lacker.
\newblock Hierarchies, entropy, and quantitative propagation of chaos for mean
  field diffusions.
\newblock {\em arXiv preprint arXiv:2105.02983}, 2021.

\bibitem[LLY19]{LLY2019}
Lei Li, Jian-Guo Liu, and Pu~Yu.
\newblock On the mean field limit for {B}rownian particles with {C}oulomb
  interaction in 3{D}.
\newblock {\em J. Math. Phys.}, 60(11):111501, 34, 2019.

\bibitem[LY16]{LY2016}
Jian-Guo Liu and Rong Yang.
\newblock Propagation of chaos for large {B}rownian particle system with
  {C}oulomb interaction.
\newblock {\em Res. Math. Sci.}, 3:Paper No. 40, 33, 2016.

\bibitem[M\'96]{Meleard1996}
Sylvie M\'{e}l\'{e}ard.
\newblock Asymptotic behaviour of some interacting particle systems;
  {M}c{K}ean-{V}lasov and {B}oltzmann models.
\newblock In {\em Probabilistic models for nonlinear partial differential
  equations ({M}ontecatini {T}erme, 1995)}, volume 1627 of {\em Lecture Notes
  in Math.}, pages 42--95. Springer, Berlin, 1996.

\bibitem[Mae08]{Maekawa2008}
Yasunori Maekawa.
\newblock A lower bound for fundamental solutions of the heat convection
  equations.
\newblock {\em Arch. Ration. Mech. Anal.}, 189(1):45--58, 2008.

\bibitem[Mal03]{Malrieu2003}
Florent Malrieu.
\newblock Convergence to equilibrium for granular media equations and their
  {E}uler schemes.
\newblock {\em Ann. Appl. Probab.}, 13(2):540--560, 2003.

\bibitem[McK67]{Mckean1967}
H.~P. McKean, Jr.
\newblock Propagation of chaos for a class of non-linear parabolic equations.
\newblock In {\em Stochastic {D}ifferential {E}quations ({L}ecture {S}eries in
  {D}ifferential {E}quations, {S}ession 7, {C}atholic {U}niv., 1967)}, pages
  41--57. Air Force Office Sci. Res., Arlington, Va., 1967.

\bibitem[MMN18]{MMN2018}
Song Mei, Andrea Montanari, and Phan-Minh Nguyen.
\newblock A mean field view of the landscape of two-layer neural networks.
\newblock {\em Proc. Natl. Acad. Sci. USA}, 115(33):E7665--E7671, 2018.

\bibitem[MMW15]{MMW2015}
St\'{e}phane Mischler, Cl\'{e}ment Mouhot, and Bernt Wennberg.
\newblock A new approach to quantitative propagation of chaos for drift,
  diffusion and jump processes.
\newblock {\em Probab. Theory Related Fields}, 161(1-2):1--59, 2015.

\bibitem[MP12]{MP2012book}
Carlo Marchioro and Mario Pulvirenti.
\newblock {\em Mathematical theory of incompressible nonviscous fluids},
  volume~96.
\newblock Springer Science \& Business Media, 2012.

\bibitem[MT11]{MT2011}
Sebastien Motsch and Eitan Tadmor.
\newblock A new model for self-organized dynamics and its flocking behavior.
\newblock {\em J. Stat. Phys.}, 144(5):923--947, 2011.

\bibitem[Nas58]{Nash1958}
J.~Nash.
\newblock Continuity of solutions of parabolic and elliptic equations.
\newblock {\em Amer. J. Math.}, 80:931--954, 1958.

\bibitem[NRS21]{NRS2021}
Quoc~Hung Nguyen, Matthew Rosenzweig, and Sylvia Serfaty.
\newblock Mean-field limits of {Riesz-typ}e singular flows with possible
  multiplicative transport noise.
\newblock {\em arXiv preprint arXiv:2107.02592}, 2021.

\bibitem[Ons49]{Onsager1949}
L.~Onsager.
\newblock Statistical hydrodynamics.
\newblock {\em Nuovo Cimento (9)}, 6(Supplemento, 2 (Convegno Internazionale di
  Meccanica Statistica)):279--287, 1949.

\bibitem[OS97]{OS1997}
Hans~G. Othmer and Angela Stevens.
\newblock Aggregation, blowup, and collapse: the {ABC}s of taxis in reinforced
  random walks.
\newblock {\em SIAM J. Appl. Math.}, 57(4):1044--1081, 1997.

\bibitem[Osa86]{Osada1986pc}
Hirofumi Osada.
\newblock Propagation of chaos for the two-dimensional {N}avier-{S}tokes
  equation.
\newblock {\em Proc. Japan Acad. Ser. A Math. Sci.}, 62(1):8--11, 1986.

\bibitem[Osa87a]{Osada1987dp}
Hirofumi Osada.
\newblock Diffusion processes with generators of generalized divergence form.
\newblock {\em J. Math. Kyoto Univ.}, 27(4):597--619, 1987.

\bibitem[Osa87b]{Osada1987lp}
Hirofumi Osada.
\newblock Limit points of empirical distributions of vortices with small
  viscosity.
\newblock In {\em Hydrodynamic behavior and interacting particle systems
  ({M}inneapolis, {M}inn., 1986)}, volume~9 of {\em IMA Vol. Math. Appl.},
  pages 117--126. Springer, New York, 1987.

\bibitem[Osa87c]{Osada1987pc}
Hirofumi Osada.
\newblock Propagation of chaos for the two-dimensional {N}avier-{S}tokes
  equation.
\newblock In {\em Probabilistic methods in mathematical physics
  ({K}atata/{K}yoto, 1985)}, pages 303--334. Academic Press, Boston, MA, 1987.

\bibitem[Per07]{Perthame2007}
Beno\^{\i}t Perthame.
\newblock {\em Transport equations in biology}.
\newblock Frontiers in Mathematics. Birkh\"{a}user Verlag, Basel, 2007.

\bibitem[Ros20a]{Rosenzweig2020pvmf}
Matthew Rosenzweig.
\newblock Mean-field convergence of point vortices without regularity.
\newblock {\em arXiv preprint arXiv:2004.04140}, 2020.

\bibitem[Ros20b]{Rosenzweig2020spv}
Matthew Rosenzweig.
\newblock The mean-field limit of stochastic point vortex systems with
  multiplicative noise.
\newblock {\em arXiv preprint arXiv:2011.12180}, 2020.

\bibitem[RS16]{RS2016}
Nicolas Rougerie and Sylvia Serfaty.
\newblock Higher-dimensional {C}oulomb gases and renormalized energy
  functionals.
\newblock {\em Comm. Pure Appl. Math.}, 69(3):519--605, 2016.

\bibitem[RVE18]{RvE2018}
Grant~M Rotskoff and Eric Vanden-Eijnden.
\newblock Trainability and accuracy of neural networks: An interacting particle
  system approach.
\newblock {\em arXiv preprint arXiv:1805.00915}, 2018.

\bibitem[Ser17]{Serfaty2017}
Sylvia Serfaty.
\newblock Mean field limits of the {G}ross-{P}itaevskii and parabolic
  {G}inzburg-{L}andau equations.
\newblock {\em J. Amer. Math. Soc.}, 30(3):713--768, 2017.

\bibitem[Ser20]{Serfaty2020}
Sylvia Serfaty.
\newblock Mean field limit for {Coulomb-type} flows.
\newblock {\em Duke Math. J.}, 169(15):2887--2935, 10 2020.
\newblock Appendix with Mitia Duerinckx.

\bibitem[SM93]{Stein1993}
Elias~M. Stein and Timothy~S. Murphy.
\newblock {\em Harmonic Analysis (PMS-43): Real-Variable Methods,
  Orthogonality, and Oscillatory Integrals. (PMS-43)}.
\newblock Princeton University Press, 1993.

\bibitem[Ste70]{Stein1970}
Elias~M Stein.
\newblock {\em Singular Integrals and Differentiability Properties of
  Functions}, volume~2.
\newblock Princeton University Press, 1970.

\bibitem[SV14]{SV2014}
Sylvia Serfaty and Juan~Luis V\'{a}zquez.
\newblock A mean field equation as limit of nonlinear diffusions with
  fractional {L}aplacian operators.
\newblock {\em Calc. Var. Partial Differential Equations}, 49(3-4):1091--1120,
  2014.

\bibitem[Szn91]{Sznitman1991}
Alain-Sol Sznitman.
\newblock Topics in propagation of chaos.
\newblock In {\em \'{E}cole d'\'{E}t\'{e} de {P}robabilit\'{e}s de
  {S}aint-{F}lour {XIX}---1989}, volume 1464 of {\em Lecture Notes in Math.},
  pages 165--251. Springer, Berlin, 1991.

\bibitem[TBL06]{TBL2006}
Chad~M. Topaz, Andrea~L. Bertozzi, and Mark~A. Lewis.
\newblock A nonlocal continuum model for biological aggregation.
\newblock {\em Bull. Math. Biol.}, 68(7):1601--1623, 2006.

\bibitem[WZZ21]{WZZ2021}
Zhenfu Wang, Xianliang Zhao, and Rongchan Zhu.
\newblock Gaussian fluctuations for interacting particle systems with singular
  kernels.
\newblock {\em arXiv preprint arXiv:2105.13201}, 2021.

\bibitem[XWX11]{XWX2011}
Haoxiang Xia, Huili Wang, and Zhaoguo Xuan.
\newblock Opinion dynamics: A multidisciplinary review and perspective on
  future research.
\newblock {\em Int. J. Knowl. Syst. Sci.}, 2(4):72--91, October 2011.

\end{thebibliography}
\end{document}